%% file: main.tex
\documentclass{amsart}
    \input{preamble}
    
\keywords{Deligne--Lusztig varieties, Galois module structure,
equivariant Euler characteristics, Lefschetz--Riemann--Roch theorem}

\begin{document}
        \addtocontents{toc}{\protect\setcounter{tocdepth}{1}} % Temporarily hide subsections
% \subsection*{This subsection is hidden from ToC}
\addtocontents{toc}{\protect\setcounter{tocdepth}{2}} % Restore standard depth

\addtocontents{toc}{\protect\setcounter{tocdepth}{0}} % Temporarily hide sections
% \section*{This section is hidden from ToC}
\addtocontents{toc}{\protect\setcounter{tocdepth}{2}} % Restore standard depth

    \raggedbottom
    
    \begin{abstract}
    Let \(\overline {\mathbf{Y}}(w)\to\overline {\mathbf{X}}(w)\) be the ramified Galois cover of compactified Deligne--Lusztig varieties. Our objects of study are the Lefschetz numbers for the action of the finite covering group \( T^{wF}\) on the coherent
    cohomology of the structure sheaf of \(\overline {\mathbf{Y}}(w)\).  The Lefschetz numbers for all elements of \(T^{wF}\)
    determine the virtual representation of \(T^{wF}\) given by the alternating sum of these cohomology groups. In the two cases of Deligne--Lusztig surfaces of type \(A_2\) and \(^{2}A_2\)% for
    % \( \mathbf{G}^F=\mathrm{GL}_3(\mathbb F_p)\) and
    % \( \mathbf{G}^F=\mathrm U_3(\mathbb F_p)\) respectively
    , we give explicit formulas for the Lefschetz numbers for a generator of the stabilizer of an irreducible component.
    It follows that the equivariant Euler characteristic of the compactification \(\overline{\mathbf{Y}}(w)\) is not virtually free as a virtual module over the covering group.
    \end{abstract}
    
    \maketitle
    %\clearpage
    \tableofcontents
    %\clearpage

\section{Introduction}\label{sec:introduction}
    
    Let \(k=\mathbb F_p\) and
    \(\bar{k}=\overline{\mathbb F}_p\).  Let \( \mathbf{G}\) be a connected
    reductive group over \(\bar{k}\), let \(F\colon \mathbf{G}\to \mathbf{G}\)
    be an isogeny such that some positive power of \(F\) is a Frobenius
    endomorphism, and let \(w\) belong to the Weyl group \(W\) of
    \( \mathbf{G}\).  Associated to this datum, Deligne and Lusztig defined  two varieties
    \( \mathbf{X}(w)\) and \(\mathbf{Y}(w)\) that fit into a finite \'{e}tale
    map
    \[
     \mathbf{Y}(w)\longrightarrow \mathbf{X}(w).
    \]
    
    Furthermore, these varieties are equipped with actions of a finite group of Lie type \( \mathbf{G}^F\) on the left and a finite group \( T^{wF}\) on the right such that, \(  \mathbf{X}(w) =  \mathbf{Y}(w) /  T^{wF} \)  
    \cite[\S1]{deligneRepresentationsReductiveGroups1976}. For a reduced expression \(w=s_1\cdots s_r\),  Deligne and Lusztig also constructed a Bott--Samelson-type smooth projective compactification
    \(\overline{ \mathbf{X}}(w)\) whose boundary is a simple
    normal-crossings divisor
    \cite[\S\S9.10--9.11]{deligneRepresentationsReductiveGroups1976}.
    
    \medskip
    
    Later, Bonnaf\'e and Rouquier constructed the normalization
    \(\overline{ \mathbf{Y}}(w)\) of
    \(\overline{ \mathbf{X}}(w)\) in \( \mathbf{Y}(w)\).  They proved
    that the right \( T^{wF}\)-action extends to the normalization and
    that its finite
    quotient is \(\overline{ \mathbf{X}}(w)\)
    \cite[Theorem~1.2]{bonnafeCompactificationVarietesDeligneLusztig2009}.
    Thus one has a commutative diagram
    \[
    \begin{tikzcd}
     \mathbf{Y}(w) \arrow[r,hook] \arrow[d] &
    \overline{ \mathbf{Y}}(w) \arrow[d] \\
     \mathbf{X}(w) \arrow[r,hook] &
    \overline{ \mathbf{X}}(w)
    \end{tikzcd}
    \]
    Although \( \mathbf{Y}(w)\to \mathbf{X}(w)\) is unramified, the compactified
    cover \(\overline{ \mathbf{Y}}(w)\to
    \overline{ \mathbf{X}}(w)\) is ramified
    along the boundary. Furthermore,
    \(p\nmid| T^{wF}|\) and Maschke's theorem
    implies that \(\bar{k}[ T^{wF}]\) is semisimple, and the boundary
    ramification is tame.
    
    \medskip

Deligne--Lusztig varieties were introduced in \cite{deligneRepresentationsReductiveGroups1976} to construct representations
of finite groups of Lie type through their compactly supported
$\ell$-adic \'etale cohomology. They also occur in the geometry of
Shimura varieties. In particular, certain generalized Deligne--Lusztig
varieties for unitary groups describe the closures of strata in the
supersingular loci of unitary Shimura varieties
\cite{vollaardSupersingularLocusShimura2011}.
Their intersection theory enters the study of special cycles
\cite{kudlaSpecialCyclesUnitary2011,kudlaSpecialCyclesUnitary2014}
and plays a role in Li and Zhang's proof of the local
Kudla--Rapoport conjecture~\cite{li_kudlarapoport_2021}.
% , which relates arithmetic intersection
% numbers on unitary Rapoport--Zink spaces to derivatives of local
% representation densities of hermitian forms

\medskip

In this paper, we study the equivariant Euler characteristic
of the structure sheaf of the compactified Deligne--Lusztig
cover \({\overline{\mathbf{Y}}}(w)\), with respect to the action
of its covering group. We begin by recalling a central question
in ``geometric Galois module theory'' that motivates our study.
\medskip

    Let \(G\) be a finite group acting on a proper
    \(\bar{k}\)-scheme \(Y\), and let \(\mathcal F\) be a \(G\)-equivariant
    coherent sheaf.  Its equivariant Euler characteristic is
    \begin{equation}
            \chi_G(Y,\mathcal F)
    :=\sum_i(-1)^i[H^i(Y,\mathcal F)]
    \in G_0(\bar{k}[G]) \label{eq:equivariant-euler-characteristic}
    \end{equation}
    When \(p\nmid |G|\), Maschke's theorem identifies the Grothendieck group of finitely generated \(\overline{k}[G]\)-modules
    \(G_0(\bar{k}[G])\) with the Grothendieck group of finitely generated projective \(\overline{k}[G]\)-modules \(K_0(\bar{k}[G])\).  We call this alternating class the virtual Galois module.
    %
    % \medskip
    %
    For a \(p\)-regular element \(\tau\in G\), we define its Lefschetz number \(L_{\mathrm{Br}}(\tau,Y,\mathcal F)\) to be
    \[
    L_{\mathrm{Br}}(\tau,Y,\mathcal F)
    :=\sum_i(-1)^i\operatorname{BrTr}
    \bigl(\tau\mid H^i(Y,\mathcal F)\bigr),
    \]
    where \( \operatorname{BrTr} \) is the Brauer trace.
    % We note that \(L_{\mathrm{Br}}(\tau,Y,\mathcal F)\) is the value at \(\tau\) of the Brauer character of
    % \(\chi_G(Y,\mathcal F)\).
    %; the coefficient ring and Teichm\"uller-lift
    %conventions are fixed in Section~\ref{sec:galois-modules-brauer}.
    
    \medskip
    
    One would like to compute the class
    \(\chi_G(Y,\mathcal F)\) in \(K_0(\bar{k}[G])\), or equivalently to determine all of its Brauer-character values.  A basic question is whether this class is \emph{virtually free}, that is, whether
    \[
    \chi_G(Y,\mathcal F)
    \in
    \mathbb Z\cdot[\bar{k}[G]]
    \subseteq K_0(\bar{k}[G]).
    \]
    
    \medskip
    
    Since the Brauer character of the regular representation \(\bar{k}[G]\) vanishes at every non-identity element of \(G\), if \(\chi_G(Y,\mathcal F)\) is virtually free, then \(L_{\mathrm{Br}}(\tau,Y,\mathcal F)=0\) for every non-identity \(p\)-regular element \(\tau\). Thus a single nonzero Lefschetz number at a non-identity element is enough to show that the equivariant Euler characteristic is not virtually free, which is what is done in this paper for some specific compactified Deligne-Lusztig covers.  
    % To determine the complete virtual representation, however, one must compute the Lefschetz numbers for all elements of \(G\), or for a collection of elements sufficient to recover all character multiplicities.
    
    \medskip
    
    The above discussed ``virtual freeness question'' has its origins in the classical \textit{normal integral basis problem} which we state below.  Let \(N/K\) be a tamely ramified Galois extension of number fields with Galois group \(G\). E. Noether proved that \(\mathcal O_N\) is a locally free \(\mathcal O_K[G]\)-module of rank one. It therefore determines a stable class
% \vspace{-0.5em}
\[
[\mathcal O_N]^{\mathrm{stab}}
:=
[\mathcal O_N]-[\mathcal O_K[G]]
\in
\operatorname{Cl}(\mathcal O_K[G])
:=
K_0(\mathcal O_K[G])/
\langle[\mathcal O_K[G]]\rangle.
% \vspace{-0.5em}
\]
If \(N/K\) admits a normal integral basis, then
\(\mathcal O_N\cong\mathcal O_K[G]\), and hence
\([\mathcal O_N]^{\mathrm{stab}}=0\). Conversely, the vanishing of
\([\mathcal O_N]^{\mathrm{stab}}\) means that \(\mathcal O_N\) is stably free, i.e., there is an integer \(m\geq 0\) such that
% \vspace{-0.5em}
\[
\mathcal O_N\oplus\mathcal O_K[G]^m
\cong
\mathcal O_K[G]^{m+1}.
% \vspace{-0.5em}
\]
 Thus, \([\mathcal O_N]^{\mathrm{stab}}\) is the obstruction to a stable, or virtual, normal integral basis. An additional cancellation result is required to obtain an actual normal integral basis.  Fr\"ohlich gave a character-theoretic description of the corresponding class over \(\mathbb Z[G]\) in terms of Artin root numbers, and his conjecture was proved by Taylor \cite{taylorFrohlichConjecture1981}.  

    \medskip
    
    Over a field, a particularly close geometric result is Nakajima's
    theorem for finite étale covers.  Let
    \(f\colon Z\longrightarrow W \) be a finite étale \(G\)-torsor between projective varieties over a
    field \(k\), and let \(\mathcal F\) be a coherent sheaf on \(W\).
    Nakajima proves that the derived complex
    \(R\Gamma(Z,f^*\mathcal F)
    \)
    is represented by a bounded complex \(L^\bullet\) of finitely
    generated free \(k[G]\)-modules~\cite[Theorem~1]{nakajimaGaloisModule1984}.
    Consequently, we have that the equivariant Euler characteristic
    \(    \chi_G(Z,f^*\mathcal F)
     =
     \sum_j(-1)^j
     \operatorname{rank}_{k[G]}(L^j)[k[G]]
    \) is an integral multiple of
    the regular representation class.
    % \medskip
    % (Equivalently, one may begin with a coherent \(G\)-equivariant sheaf
    % \(\mathcal E\) on \(Z\): since \(f\) is a \(G\)-torsor, faithfully
    % flat descent gives a coherent sheaf
    % \(\mathcal F=(f_*\mathcal E)^G\) on \(W\) and a canonical
    % isomorphism \(\mathcal E\simeq f^*\mathcal F\)). 
    Furthermore, he proves that if  
    \(  H^i(Z,f^*\mathcal F)=0\)  for every \(i\neq i_0,
    \)
    then \(H^{i_0}(Z,f^*\mathcal F)\) is itself a free
    \(k[G]\)-module~\cite[Theorem~2]{nakajimaGaloisModule1984}.
    
    \medskip
    
    In subsequent work on tamely ramified covers of varieties, Nakajima proves instead that the coherent cohomology complex can in general be represented by a bounded complex of finitely generated projective \(k[G]\)-modules \cite{nakajimaGaloisModule1986}.  
    %This guarantees perfectness, but it does not by itself imply that the resulting Euler class is a multiple of \([k[G]]\). 
    Thus, tameness guarantees that the derived cohomology is represented by a perfect \(k[G]\)-complex and hence defines an Euler class in \(K_0(k[G])\), but it does not imply virtual freeness.
    % Thus the passage from an \'{e}tale torsor to a ramified cover is exactly the passage from a ``free-complex theorem'' to a ``projective-complex theorem''.
    
    \medskip
    
Chinburg extended the study of Galois module structure
to tame covers of schemes by constructing invariants
from equivariant Euler characteristics of truncated
de Rham complexes
\cite{chinburgGaloisStructureRham1994b}.
For projective covers over finite fields, under
regularity and normal-crossings hypotheses, he proved
that an invariant, after restriction of scalars to
\(\mathbb F_p\), determines and is determined by the
\(p\)-adic absolute values of the associated
epsilon-constants. He also expressed its image in
\(\mathrm{Cl}(\mathbb Z[G])\) as the sum of a symplectic
root-number class and a ramification class, giving a
geometric counterpart to Fr\"ohlich's conjecture
\cite[Theorems~0.1--0.2]{chinburgGaloisStructureRham1994b}.
In subsequent joint work, 
Chinburg, Erez, Pappas, and Taylor expressed, under suitable hypotheses, the projective Euler characteristic of a
suitable de Rham complex for a tame cover over
\(\mathbb Z\) in terms of epsilon-constants of
Hasse--Weil--Artin \(L\)-functions, giving a
higher-dimensional counterpart to Fr\"ohlich's conjecture \cite{chinburgEConstantsGaloisStructure1997}.

  % A related question concerns coherent cohomology over
%\(\mathbb Z\). 
  
    % \medskip
    
    % In our setting, the proper map
    % \[
    % \overline{ \mathbf{Y}}(w)\longrightarrow
    % \overline{ \mathbf{X}}(w)
    % \]
    % is ramified along the boundary.  Nakajima's free-complex theorem therefore does not apply to the compactification.  Non-identity elements of the deck group \( T^{wF}\) may have fixed points on the boundary, and their Lefschetz numbers need not vanish.  
    % In the situation considered in this paper, the group \(G\) is
    % \( T^{wF}\), the variety is the normal compactification
    % \(\overline{ \mathbf{Y}}(w)\), and the equivariant sheaf is its structure
    % sheaf.  The boundary ramification of the map \(\overline{ \mathbf{Y}}(w)\to
    % \overline{ \mathbf{X}}(w)\) is precisely what makes this question of virtual freeness in the above defined sense non-trivial.

    \medskip
    
    In the situation considered in this paper, the boundary ramification of
    \(\overline{ \mathbf{Y}}(w)\to\overline{ \mathbf{X}}(w)\) is consequently the source of the non-vanishing of the reduced equivariant Euler class, and hence of the failure of virtual freeness in the sense defined above. We use the Lefschetz--Riemann--Roch theorem to answer the virtual freeness problem algebro-geometrically. 
    We first state the Lefschetz--Riemann--Roch theorem for surfaces below.
    
    \begin{theorem}(Lefschetz--Riemann--Roch)
        For a smooth projective surface \(Z\) and a finite-order automorphism \(h\)
    of order prime to \(p\), whose fixed locus consists of isolated points and curves,
    \begin{equation}\label{eq:intro-lefschetz-rr}
    \begin{aligned}
    L_{\mathrm{Br}}(h,Z, \mathcal{O}_Z)
    ={}&
    \sum_{P\subset Z^h}
    \frac{1}
    {(1-[\lambda_{P,1}])(1-[\lambda_{P,2}])}\\
    &+
    \sum_{D\subset Z^h}
    \left(
    \frac{1-g_D}{1-[\nu_D]^{-1}}
    -
    \frac{[\nu_D]^{-1}D^2}
    {(1-[\nu_D]^{-1})^2}
    \right).
    \end{aligned}
    \end{equation}
    \end{theorem}

    Here, \(P\) runs over the isolated fixed points, \(D\) runs over the
    fixed curves, \(\lambda_{P,1},\lambda_{P,2}\) are the cotangent
    eigenvalues at \(P\), \(\nu_D\) is the normal eigenvalue along \(D\),
    and the brackets \([\cdot]\) denote Teichm\"uller lifts.  
    %Thus a global alternating
    %trace on coherent cohomology is reduced to fixed-locus, character, and
   % intersection data. 
   A more general statement of this theorem and its background is given
    in Section~\ref{sec:lefschetz-rr}.
    
    \medskip
    
    % The normal compactification \(\overline{ \mathbf{Y}}(w)\) may be
    % singular, whereas the displayed formula is most convenient on a smooth
    % surface.  In each of the two cases \(A_2\) and \({}^2A_2\) considered here, we fix an
    % equivariant resolution of an irreducible component of
    % \(\overline{ \mathbf{Y}}(w)\).  Its singularities are rational, and the equivariant Leray spectral sequence
    % identifies the coherent cohomology of the component with that of its
    % resolution.  We may therefore compute the same Brauer--Lefschetz number
    % on the resolved surface.  A nonzero value at a nonidentity element does
    % not determine the full Brauer character, but it suffices to rule out
    % virtual freeness.

    % \medskip

For Deligne--Lusztig varieties, related results on Galois
module structure are known for the Drinfeld curve and its
smooth projective compactification. These results concern
the action of \(\mathrm{SL}_2(\mathbb F_q)\), whereas we study
the action of the torus covering group.
Laurent and K\"ock~\cite{laurentCanonicalRepresentationDrinfeld2024a} determine the canonical representation of
\(\mathrm{SL}_2(\mathbb F_q)\) on the space of holomorphic
differentials of the smooth projective Drinfeld curve
\(C\colon XY^q-X^qY=Z^{q+1}\).
For every prime power \(q=p^r\), they obtain the decomposition
\[
H^0(C,\Omega_C)
\simeq
\bigoplus_{j=0}^{q-2}
\operatorname{Sym}^j(\bar{k}^{\,2})
\]
into indecomposable modules, where \(\bar{k}^{\,2}\) carries
the standard representation. In particular, the canonical
representation is semisimple if and only if \(q=p\).

\medskip

Other related results include that of K\"ock and Marchment~\cite{marchmentGaloisStructureSpaces2026}, who study the action of \(\mathrm{SL}_2(\mathbb F_q)\) on the
spaces \(H^0(C,\Omega_C^{\otimes m})\), \(m\geq2\), of holomorphic
polydifferentials on the smooth projective Drinfeld curve
\(C\colon XY^q-X^qY=Z^{q+1}\), in odd characteristic \(p\).
For \(q=p\), they determine the complete decomposition into
indecomposable modules. For arbitrary powers \(q\) of \(p\),
they construct an explicit basis and obtain a partial
module decomposition.

\medskip

    Our main results are stated below, following which we explain the proof methology in Section~\ref{sec:proof-methodology}.
    % , and the
    % cohomological comparison with an equivariant resolution is established
    % in Section~\ref{sec:equivariant-resolution-comparison}.
    % Boldface symbols denote the general Deligne--Lusztig construction that appears in \cite{bonnafeCompactificationVarietesDeligneLusztig2009}. 
    After
    specializing to the (split) \(A_2\) and (non-split) \({}^2A_2\) cases, and fixing the Coxeter element \(w = s_1 s_2 \in W \), we write \(\mathbf{X},\mathbf{Y},\overline {\mathbf{X}},\overline {\mathbf{Y}}\) for the associated Deligne--Lusztig varieties and \( \Gamma \) for the covering group \( {T}^{wF}\).
    
    \subsection{Main results}\label{sec:intro-main-results}
    
    This paper answers the \textit{virtual freeness} problem discussed above in the negative for Deligne-Lusztig surfaces for \( \mathbf{G}=\mathrm{GL}_3\) of both split and non-split type.  The word \emph{split} refers to the
    standard \(\mathbb F_p\)-structure on \( \mathbf{G}=\mathrm{GL}_3\) given by
    \[
    F_{\mathrm{sp}}(g)=g^{(p)},
    \qquad
     \mathbf{G}^{F_{\mathrm{sp}}}=\mathrm{GL}_3(\mathbb F_p),
    \]
    for which the Frobenius acts trivially on the Dynkin diagram of type
    \(A_2\).  The \emph{non-split}, or unitary, structure is defined by
    \[
    F_{\mathrm{un}}(g)=J_+(g^{(p)})^{-t}J_+^{-1},
    \qquad
    J_+=\operatorname{antidiag}(1,1,1).
    \]
    Here \(F_{\mathrm{un}}^2\) is the \(p^2\)-Frobenius,
    \( \mathbf{G}^{F_{\mathrm{un}}}=\mathrm U_3(\mathbb F_p)\), and the
    induced automorphism on the Dynkin diagram interchanges the two simple reflections.
    Thus the two cases have the same group \(\mathrm{GL}_3\) after base
    change to \(\bar{k}\), but different Frobenius structures: the first is
    the untwisted type \(A_2\) and the second is the twisted type
    \({}^{2}A_2\). We fix the Weyl group element \(w\) in both the cases to be a Coxeter element \( w = s_1 s_2\).
    
    % In Section~\ref{part:split}, we discuss the split case with \(F=F_{\mathrm{sp}}\) and \(w=(123)\). And In Section~\ref{part:unitary}, \(F=F_{\mathrm{un}}\).  In each case we
    % write explicit equations for the compactified
    % cover, resolve its cyclic quotient singularities torically, determine the
    % fixed locus, and apply Theorem~\ref{thm:lefschetz-rr}.  The two main
    % calculations are as follows.
    
    \begin{theorem}[Split case]\label{thm:intro-split-brauer}
    % For the right \( T^{wF}\)-action, use the diagonal realization
    % \[
    %  T^{wF}
    % =
    % \left\{
    % \diag(a,a^p,a^{p^2}):a\in\mathbb F_{p^3}^{\times}
    % \right\}
    % \simeq\mathbb F_{p^3}^{\times}.
    % \]
    Choose \(\mu\in\mathbb F_{p^3}^{\times}\) of order \(p^3-1\), and set
    \[
    \tau:=\diag(\mu,\mu^p,\mu^{p^2}),
    \qquad
    N:=p^2+p+1,
    \qquad
    \sigma:=\tau^{p-1},
    \qquad
    \bar\zeta:=\mu^{1-p}.
    \]
    Then \(\tau\) generates \( T^{wF}\), while \(\sigma\) generates
    the stabilizer of each of the \((p-1)\)-many irreducible components \( \overline{\mathbf{Y}}_\ell\) of the compactified cover \(\overline{\mathbf{Y}}\),
    %and \(\bar\zeta\) is a primitive \(N\)-th root.  
    % \(\bar\zeta\) is the scalar by
    % which \(\sigma\) acts on the local covering coordinate.  
    Put
    \(\widehat\zeta=[\bar\zeta]\in W(\mathbb F_{p^3})^\times\).  Write
    \(\overline {\mathbf{Y}}=\bigsqcup_{\ell\in\mathbb F_p^\times}
    \overline {\mathbf{Y}}_\ell\).  For every such component,
    \begin{align}
    L_{\mathrm{Br}}(\sigma,\overline {\mathbf{Y}}_{\ell}, \mathcal{O}_{\overline {\mathbf{Y}}_{\ell}})
    ={}&N(p+1)\left(
    \frac{\widehat\zeta(\widehat\zeta^{p-1}-1)}
    {(1-\widehat\zeta)^2(1-\widehat\zeta^p)}
    \right)\nonumber\\
    &+N\frac{1+(p-1)\widehat\zeta}
    {(1-\widehat\zeta)^2}
    +N\frac{\widehat\zeta^{2p}}{(1-\widehat\zeta^p)^2}.
    \label{eq:intro-split-brauer}
    \end{align}
    %Moreover, \(\sigma\) preserves all \(p-1\) components and \(\tau\)
    %identifies their Brauer characters.  Hence 
    And, the Lefschetz number of the entire normal compactification is given by
    \[
    L_{\mathrm{Br}}(\sigma,\overline {\mathbf{Y}}(w))
    =(p-1)L_{\mathrm{Br}}(\sigma,\overline {\mathbf{Y}}_\ell(w)).
    \]
    Furthermore, both the Lefschetz number in \eqref{eq:intro-split-brauer} and the value
    for the entire normal compactification are nonzero.
    \end{theorem}

    \begin{theorem}[Non-split unitary case]\label{thm:intro-unitary-brauer}
    Choose \(\alpha\in\mathbb F_{p^2}^{\times}\) of order \(p^2-1\), and
    set
    \[
    \tau:=\diag(\alpha^{p-1},\alpha,\alpha^{-p}),
    \qquad
    \bar\xi:=\alpha^p,
    \qquad
    \widehat\xi:=[\bar\xi]\in W(\mathbb F_{p^2})^\times.
    \]
    Then, \( T^{wF}
    =
    \left\{
    \operatorname{diag}(\zeta,a,a^{-p})
    :
    \zeta^{p+1}=1,\quad a^{p^2-1}=1
    \right\} \simeq \mu_{p+1} \times \mu_{p^2-1} \).
    The normal compactification \(\overline {\mathbf{Y}}\) has exactly \(p+1\)
    irreducible components \(\overline{\mathbf{Y}}_{\ell}\), and \(\tau\) generates the stabilizer \(H\) of each of those components. Fix one of them and denote it by
    \(\overline {\mathbf{Y}}_{\ell}\).  
    % For the right \( T^{wF}\)-action,
    % use the standard diagonal realization
    % \[
    %  T^{wF}
    % =
    % \left\{
    % \diag(\zeta,a,a^{-p}):
    % \zeta^{p+1}=1,\ a^{p^2-1}=1
    % \right\}.
    % \]
    % The setwise stabilizer of \(\overline {\mathbf{Y}}^\circ(w)\) is
    % \[
    % H=
    % \left\{
    % \diag(a^{p-1},a,a^{-p}):a^{p^2-1}=1
    % \right\}
    % \simeq\mu_{p^2-1}.
    % \]
    Set
    \[
    c_0=\frac{(p+1)(p-2)}2,
    \qquad
    c_u=\frac{(p+1)(3p-2)}2,
    \qquad
    c_v=\frac{(p+1)(3p-4)}2.
    \]
    These are integers for every prime \(p\).  For this chosen component,
    \begin{align}
    L_{\mathrm{Br}}(\tau,\overline {\mathbf{Y}}_{\ell})
    ={}&
    \frac{c_u\widehat\xi-c_0}
    {(1-\widehat\xi)^2}
    -\frac{(p^3+1)\widehat\xi}
    {(1-\widehat\xi)(1-\widehat\xi^p)}\nonumber\\
    &+\frac{\widehat\xi^p(c_v-c_0\widehat\xi^p)}
    {(1-\widehat\xi^p)^2}.
    \label{eq:intro-unitary-brauer}
    \end{align}
    Furthermore, both the Lefschetz number in
    \eqref{eq:intro-unitary-brauer} and the Lefschetz number for the entire normal
    compactification are nonzero.
    \end{theorem}

    \subsection{Proof methodology}\label{sec:proof-methodology}
    
    We fix the Deligne--Lusztig datum as discussed above restricting ourselves to the rank two case, \( {\mathbf{G}} = \mathrm{GL}_3\).
    The goal is to check whether the equivariant Euler characteristic \(\chi(\overline{\mathbf{Y}}, \mathcal{O}_{\overline{\mathbf{Y}}})\) of the normal compactification \(\overline{\mathbf{Y}}\) is a virtual multiple of the regular
    representation of its finite covering group \(\Gamma\).  To do so, in both the cases, we explicitly show that the Lefschetz number at a specific non-identity element of the covering group is non-zero, and employ the same methods discussed below.  
    
    \medskip
    
    The group \( \Gamma \) acts transitively on the irreducible components of \(\overline{\mathbf{Y}}\).  We therefore fix one normal
    component \(Z:=\overline{\mathbf{Y}}_{\ell}\), let \(\Gamma_0\subset \Gamma\) be its setwise
    stabilizer, and first compute the Lefschetz number of a generator \(h\) of \(\Gamma_0\)
    on \(Z\).  In the split case \(|\Gamma_0|=p^2+p+1\), and in the non-split
    case \(|\Gamma_0|=p^2-1\). 
    
    \medskip

    % Bonnafe-Rouquier \cite[Theorem 1.2]{bonnafeCompactificationVarietesDeligneLusztig2009} show that the singular locus of the variety \(\overline{\mathbf{Y}}\) is restricted to codimension \(2\). But, in our case, the associated Deligne--Lusztig varieties are surfaces, and 
    The singular locus of \(\overline{\mathbf{Y}}\) is a disconnected set of points. We explicitly construct the variety \(\overline{\mathbf{Y}}\) as the normalization of \(\overline{\mathbf{X}}\) in \(\mathbf{Y}\),  by first analyzing the geometry of \(\overline{\mathbf{X}}\). We would like to point out that Bonnafe and Rouquier \cite{bonnafeCompactificationVarietesDeligneLusztig2009} construct \(\overline{\mathbf{Y}}\) as a quotient of a smooth quasi-projective scheme \(\tilde{\mathbf{Y}}\) by \(\mathbb{G}_m^2\), but we do not employ those methods to construct \(\overline{\mathbf{Y}}\), and we work with smooth projective desingularizations of \(\overline{\mathbf{Y}}\) instead, for the applications considered in this paper.

    \medskip
    
   At a boundary point of \(\overline {\mathbf{X}}\), we construct an auxiliary
\'etale-local cover \(\mathbf{Y}'\) whose normalization gives a toric model
for the singularities of \(\overline {\mathbf{Y}}_{\ell}\). Using \'etale slice theorems, we verify that the resulting singularity
types match with those arising in the compactifications constructed by
Bonnaf\'e and Rouquier
\cite{bonnafeCompactificationVarietesDeligneLusztig2009}.
    
    \medskip
    
     Since it is convenient to apply the Lefschetz-Riemann-Roch theorem for non-singular varieties, we construct explicit \(\Gamma_0\)-equivariant resolution of singularities \(\widetilde Z\) of \(\overline{\mathbf{Y}}_{\ell}\).  We also determine the fixed curves,
    isolated fixed points, and their normal and tangent characters for the action of the generator \( \sigma \) of \(\Gamma_0\). The genera or self-intersections of the strict transforms of the boundary curves are later obtained from the incidence geometry of the compactified base \(\overline{\mathbf{X}}\). The results from toric geometry that we use are collected in Section~\ref{sec:toric-surface-preliminaries}.
    % Lefschetz--Riemann--Roch is most directly applied on a smooth projective
    % surface, while \(Z\) has cyclic quotient singularities.  
    For the equivariant minimal resolution
    \(f\colon\widetilde Z\longrightarrow Z,
    \)
    the quotient singularities are rational, so
    \[
    f_*\mathcal O_{\widetilde Z}=\mathcal O_Z, \text{ and }
    \quad
    R^jf_*\mathcal O_{\widetilde Z}=0\quad(j>0).
    \]
    The Leray spectral sequence
    \[
    E_2^{i,j}
    =H^i\!\left(Z,R^jf_*\mathcal O_{\widetilde Z}\right)
    \Longrightarrow
    H^{i+j}(\widetilde Z,\mathcal O_{\widetilde Z})
    \]
    therefore has only its \(j=0\) row and degenerates.  Because \(f\) is
    equivariant, its edge maps give isomorphisms of
    \(\bar{k}[\langle h\rangle]\)-modules
    \[
    H^i(Z,\mathcal O_Z)
    \simeq
    H^i(\widetilde Z,\mathcal O_{\widetilde Z})
    \qquad(i\geq0).
    \]
    Consequently, the resolution has the same equivariant coherent
    cohomology and hence the same Lefschetz numbers:
    \(
    L_{\mathrm{Br}}(h,Z)
    =L_{\mathrm{Br}}(h,\widetilde Z).
    \)

    \medskip
    
    We finally insert the required data about the fixed loci into the Lefschetz--Riemann--Roch formula of
    Section~\ref{sec:lefschetz-rr} to compute the Lefschetz numbers.  We show that these numbers are non-zero, thereby ruling out virtual freeness, first for the chosen
    component under \(\Gamma_0\) and, hence also for the entire
    normal compactification \(\overline{\mathbf{Y}}\).

    % \medskip

\begin{remark}
     As the singularities of \(\overline{\mathbf{Y}}(w)\) are at worst abelian quotient singularities, we would like to point out that our methods can be used to study the geometry and the Galois module structure of varieties \(\overline{\mathbf{Y}}(w)\) in other group types and general elements \(w \in W\) for which \(\overline{\mathbf{Y}}(w)\) is a surface.
\end{remark}

    \subsection{Structure of the paper}\label{sec:paper-plan}
    
   In Section~\ref{sec:preliminaries}, we fix notation, recall
Deligne--Lusztig varieties and their compactifications, and
introduce virtual Galois modules and Brauer traces. We also
state the Lefschetz--Riemann--Roch theorem, recall the results
on normalization and resolution of toric surfaces needed
below, and describe the Bonnaf\'e--Rouquier slice construction.

In Section~\ref{part:split}, we study the split case. We give
explicit equations, determine the local normalizations and
toric resolutions, and compute the fixed loci and intersection
numbers. In Section~\ref{sec:split-brauer-calculation}, we use
these calculations to prove Theorem~\ref{thm:intro-split-brauer}.
In Section~\ref{sec:split-nonvanishing}, we establish the
non-vanishing of the resulting expression and deduce the
Galois-module consequence.

In Section~\ref{part:unitary}, we study the unitary case and carry out similar calculations as in the previous section. In
Section~\ref{sec:unitary-brauer-calculation}, we prove
Theorem~\ref{thm:intro-unitary-brauer}, and in
Section~\ref{sec:unitary-nonvanishing}, we establish the
non-vanishing of the Lefschetz numbers and deduce the
corresponding Galois-module obstruction.

\section*{Acknowledgements}

The author thanks George Pappas for helpful discussions, insightful advice, and for valuable comments on an earlier version of this article. %The author was partially supported by Simons Foundation grant SFI-MPS-TSM-00013296.

%\section*{Disclosure of AI assistance}

The author used ChatGPT to help identify relevant references, check typographical errors, and improve the readability of some sentences. %The author verified the references and reviewed all suggested revisions, and takes full responsibility for the content of this article. 
    
    \section{Preliminaries}
    \label{sec:preliminaries}
    
    We collect the conventions and geometric facts used throughout the paper.
    Our presentation follows the original construction of Deligne and Lusztig
    \cite[\S\S1 and~9]{deligneRepresentationsReductiveGroups1976} and the modern treatment of Bonnafe-Rouquier \cite{bonnafeCompactificationVarietesDeligneLusztig2009} and G\"ortz-Schr\"oer \cite[\S1]{gortzDeligneLusztigVarietiesWhose2026}.
    Recall that \(k=\mathbb F_p\) and
    \(\bar{k}=\overline{\mathbb F}_p\).
    
    \subsection{Reductive groups, Frobenius, and Weyl-group data}
    
    Let \( \mathbf{G}\) be a connected reductive group over \(\bar{k}\).  Let
    \(F: \mathbf{G}\to \mathbf{G}\) be an isogeny such that \(F^d\) is a Frobenius
    endomorphism for some \(d\geq1\).  The fixed-point group
    \[
     \mathbf{G}^F=\{g\in \mathbf{G}(\bar{k}):F(g)=g\}
    \]
    is a finite group of Lie type.  By the Lang--Steinberg theorem, one may
    choose an \(F\)-stable Borel subgroup \( B\), an \(F\)-stable maximal
    torus \( T\subseteq B\), and the unipotent radical
    \( U\) of \( B\).  We use the standard conventions of
    \cite{digneMichelRepresentationsFiniteGroups2020,malleTestermanLinearAlgebraicGroups2011};
    for the Lang map and the fixed-point theory, see
    \cite{langAlgebraicGroupsFinite1956,steinbergEndomorphismsLinearAlgebraic1968}.
    
    \medskip
    
    Let
    \(
    W=N_{ \mathbf{G}}( T)/ T
    \)
    be the Weyl group.  
    The choice of \( B\) determines the set \(S\) of
    simple reflections, the length function \(\ell:W\to\mathbb Z_{\geq0}\), and
    the Bruhat order \(\leq\).  For \(w\in W\), we choose a representative
    \(\dot w\in N_{ \mathbf{G}}( T)\).  The induced endomorphism of the
    torus is denoted
    \[
    wF: T\longrightarrow T,
    \qquad
    wF(t)=\dot wF(t)\dot w^{-1},
    \]
    and its finite group of fixed points is \( T^{wF}\).  The actions of
    \(F\) on the character and cocharacter lattices are adjoint for the natural
    pairing, and \(F\) permutes the simple roots and simple reflections; see
    \cite{digneMichelRepresentationsFiniteGroups2020,malleTestermanLinearAlgebraicGroups2011}.
    
    \subsection{Relative position and Deligne--Lusztig varieties}
    
    Write \(\mathcal B=\operatorname{Bor}( \mathbf{G})\) for the variety of Borel
    subgroups of \( \mathbf{G}\).  The choice of the Borel subgroup \( B\)
    gives an identification
    \[
     \mathbf{G}/ B\xrightarrow{\sim}\mathcal B,
    \qquad
    g B\longmapsto g B g^{-1}.
    \]
    For every algebraically closed
    extension \(\bar{k}/k\), the Bruhat decomposition gives a disjoint union
    \[
    \mathcal B(\bar{k})\times\mathcal B(\bar{k})
    =
    \coprod_{w\in W}\mathcal O(w)(\bar{k}),
    \]
    where the diagonal \( \mathbf{G}\)-orbit indexed by \(w\) is
    \[
    \mathcal O(w)
    =
    \{(g_1 B,g_2 B):
    g_1^{-1}g_2\in B\dot w B\}.
    \]
    Consequently there is a well-defined relative-position map on geometric
    points
    \begin{equation}
    \operatorname{inv}\colon
    \mathcal B(\bar{k})\times\mathcal B(\bar{k})\longrightarrow W,
    \qquad
    \operatorname{inv}(g_1 B,g_2 B)=w
    \Longleftrightarrow
    g_1^{-1}g_2\in B\dot w B.
    \label{eq:relative-position-map}
    \end{equation}
    It is independent of the representatives \(g_1,g_2\).  Scheme-theoretically,
    \(\mathcal O(w)\) is the fiber over \(w\) of the quotient-stack map
    \[
    \mathcal B\times\mathcal B
    \longrightarrow
    [ \mathbf{G}\backslash(\mathcal B\times\mathcal B)]
    \simeq[ B\backslash \mathbf{G}/ B].
    \]
    Thus two Borel subgroups \(\mathcal B_1,\mathcal B_2\) are said to be
    \emph{in relative position \(w\)}, or equivalently that \(\mathcal B_2\) has
    relative position \(w\) with respect to \(\mathcal B_1\), if
    \[
    \operatorname{inv}(\mathcal B_1,\mathcal B_2)=w.
    \]
    Following the notation of Bonnaf\'e--Rouquier
    \cite[\S1.1]{bonnafeCompactificationVarietesDeligneLusztig2009}, we may also
    write this ordered relation as
    \[
    \mathcal B_1\xrightarrow{\,w\,}\mathcal B_2.
    \]
    The order is significant:
    \(\mathcal B_1\xrightarrow{w}\mathcal B_2\) if and only if
    \(\mathcal B_2\xrightarrow{w^{-1}}\mathcal B_1\).  Scheme-theoretically,
    the first relation means precisely that the ordered pair lies in the locally
    closed subscheme \(\mathcal O(w)\).  Bruhat closure gives
    \begin{equation}
    \overline{\mathcal O(w)}
    =
    \bigcup_{v\leq w}\mathcal O(v).
    \label{eq:prelim-orbit-closure}
    \end{equation}
    If \(w=w_1w_2\) and
    \(\ell(w)=\ell(w_1)+\ell(w_2)\), the intermediate Borel subgroup in a chain
    of relative positions \(w_1,w_2\) is unique.  Equivalently,
    \[
    \mathcal O(w_1)\times_{\mathcal B}\mathcal O(w_2)
    \xrightarrow{\sim}\mathcal O(w).
    \]
    These facts are recalled in \cite[\S1]{deligneRepresentationsReductiveGroups1976}
    and \cite[\S1.3]{gortzDeligneLusztigVarietiesWhose2026}.
    
    \subsubsection[Relative position in type A]{Relative position in type \(A\)}
    
    We only consider the root data of type \(A\) in this paper.  For
    \( \mathbf{G}=\mathrm{GL}_n\), with \( B\) upper triangular and
    \( T\) diagonal, one has \(W\simeq S_n\).  In this case, we identify
    \(\mathcal B\) with the full flag variety \(\mathrm{Fl}\); its
    \(\bar{k}\)-points are flags
    \[
    \mathcal F=(0=V_0\subset V_1\subset\cdots\subset V_n=K^n),
    \qquad \dim V_i=i.
    \]
    For \(w\in S_n\), put
    \[
    r_w(i,j)
    :=
    \#\{\nu:1\leq\nu\leq j,\ w(\nu)\leq i\}.
    \]
    If \(\mathcal F=(V_i)\) and \(\mathcal F'=(V'_j)\), then
    \begin{equation}
    \operatorname{inv}(\mathcal F,\mathcal F')=w
    \quad\Longleftrightarrow\quad
    \dim(V_i\cap V'_j)=r_w(i,j)
    \quad\text{for all }i,j.
    \label{eq:type-A-relative-position-ranks}
    \end{equation}
    Moreover, the closure condition is
    \begin{equation}
    \operatorname{inv}(\mathcal F,\mathcal F')\leq w
    \quad\Longleftrightarrow\quad
    \dim(V_i\cap V'_j)\geq r_w(i,j)
    \quad\text{for all }i,j.
    \label{eq:type-A-Bruhat-ranks}
    \end{equation}
    Thus, under this \(\mathrm{GL}_n\)-specific identification, relative position
    records all incidence dimensions between the two flags, while the Bruhat
    order records their possible degenerations; see
    \cite[\S\S1.3--1.4]{gortzDeligneLusztigVarietiesWhose2026}.

    \medskip

There are two Frobenius structures of type \(A_{n-1}\) relevant here.  In
    the split case,
    \[
    F_{\mathrm{sp}}(g)=g^{(p)},
    \qquad
     \mathbf{G}^{F_{\mathrm{sp}}}=\mathrm{GL}_n(\mathbb F_p),
    \]
    and \(F_{\mathrm{sp}}\) acts trivially on the Dynkin diagram.  In the
    non-split case, put
    \[
    J_+=\operatorname{antidiag}(1,\ldots,1),
    \qquad
    F_{\mathrm{un}}(g)
    =J_+(g^{(p)})^{-t}J_+^{-1}.
    \]
    Then \(F_{\mathrm{un}}^2\) is the \(p^2\)-Frobenius,
    \( \mathbf{G}^{F_{\mathrm{un}}}=\mathrm U_n(\mathbb F_p)\), and the induced
    diagram involution is
    \(
    s_i\longmapsto s_{n-i}.
    \)
    These are the untwisted type \(A_{n-1}\) and twisted type
    \({}^{2}A_{n-1}\), respectively.  For \(n=3\), they give exactly the
    split and non-split type-\(A_2\) surfaces treated in this paper.  
    %We also fix the
    %Coxeter element \(w=s_1s_2\) of length two.
    
    \subsubsection{Deligne--Lusztig strata}
    
    The graph morphism of \(F\) is
    \[
    \Gamma_F\colon\mathcal B\longrightarrow\mathcal B\times\mathcal B,
    \qquad
    \mathcal C\longmapsto(\mathcal C,F(\mathcal C)).
    \]
    The Deligne--Lusztig variety attached to \(w\) is the inverse image of the
    relative-position stratum \(\mathcal O(w)\):
    \begin{align}
     \mathbf{X}(w)
    &=
    \{g B\in\mathcal B:
    g^{-1}F(g)\in B\dot w B\}\nonumber\\
    &=
    \{\mathcal C\in\mathcal B:
    \operatorname{inv}(\mathcal C,F(\mathcal C))=w\}\nonumber\\
    &=
    \{\mathcal C\in\mathcal B:
    \mathcal C\xrightarrow{\,w\,}F(\mathcal C)\}.
    \label{eq:prelim-Xw}
    \end{align}
    Equivalently, it is the transverse intersection of \(\mathcal O(w)\) with
    the graph of \(F\).  Since every ordered pair of Borel subgroups has a unique
    relative position, the variety of Borel subgroups is partitioned into
    Deligne--Lusztig strata
    \[
    \mathcal B=\coprod_{w\in W} \mathbf{X}(w).
    \]
    The closed stratum \( \mathbf{X}(1)=\mathcal B^F\) is the finite scheme of
    \(F\)-stable Borel subgroups, whereas the stratum for the longest element
    is open.  Each \( \mathbf{X}(w)\) is stable under \( \mathbf{G}^F\), because
    simultaneous left translation by an \(F\)-fixed element does not change
    relative position.
    In particular, the Coxeter element \( w = s_1s_2\) produces the surfaces studied
    in this paper.
    
    \begin{proposition}[Basic geometry of \( \mathbf{X}(w)\)]
    \label{prop:prelim-basic-X}
    The scheme \( \mathbf{X}(w)\) is smooth and equidimensional of dimension
    \(\ell(w)\), and \( \mathbf{G}^F\) acts on it by left translation.  Its
    schematic closure in \(\mathcal B\) is normal and Cohen--Macaulay, and its
    underlying stratification is
    \[
    \overline{ \mathbf{X}(w)}^{\mathrm{naive}}
    =
    \bigcup_{v\leq w} \mathbf{X}(v).
    \]
    \end{proposition}
    
    \begin{proof}
    This follows by comparing the Lang morphism with the Bruhat cell
    \( B\dot w B/ B\); see
    \cite[Lemma~1.10]{gortzDeligneLusztigVarietiesWhose2026}, whose proof uses
    \cite{langAlgebraicGroupsFinite1956,steinbergEndomorphismsLinearAlgebraic1968}.
    \end{proof}
    
    \subsubsection*{The covering variety}
    
    Let \( U=R_u( B)\), and choose a representative
    \(\dot w\in N_{ \mathbf{G}}( T)\) of \(w\).  The corresponding
    Deligne--Lusztig covering variety is
    \begin{equation}
     \mathbf{Y}(w)
    =
    \{g U\in \mathbf{G}/ U:
    g^{-1}F(g)\in U\dot w U\}.
    \label{eq:prelim-Yw}
    \end{equation}
    Write
    \[
    wF(t)=\dot wF(t)\dot w^{-1}
    \qquad(t\in T),
    \]
    and let \( T^{wF}\) be the finite group of fixed points of this
    twisted Frobenius \(wF\).  Left multiplication by \( \mathbf{G}^F\) preserves
    \( \mathbf{Y}(w)\).  There is also a right action
    \[
    (g U)\cdot t=gt U
    \qquad(t\in T^{wF})
    \]
    because the condition \(wF(t)=t\) gives
    \((gt)^{-1}F(gt)\in U\dot w U\).  The left and right
    actions commute.
    
    The quotient map \( \mathbf{G}/ U\to \mathbf{G}/ B\) restricts
    to
    \begin{equation}
    \pi_w: \mathbf{Y}(w)\longrightarrow \mathbf{X}(w),
    \qquad
     \mathbf{Y}(w)/ T^{wF}\xrightarrow{\sim} \mathbf{X}(w).
    \label{eq:prelim-DL-cover}
    \end{equation}
    Indeed, changing a lift of a point of \( \mathbf{X}(w)\) from
    \( \mathbf{G}/ B\) to \( \mathbf{G}/ U\) amounts to right
    multiplication by an element of \( T= B/ U\), and the
    Frobenius condition restricts the possible changes exactly to
    \( T^{wF}\).  Thus every geometric fiber of \(\pi_w\) is a simply
    transitive \( T^{wF}\)-set.  In particular, \(\pi_w\) is a finite
     \'{e}tale right \( T^{wF}\)-torsor  %This construction, including the commuting
    % \( \mathbf{G}^F\)- and \( T^{wF}\)-actions, is the covering introduced
    % in 
    \cite[\S1]{deligneRepresentationsReductiveGroups1976}.
    % ; the sequence
    % version used for compactifications is recalled in
    % \cite[\S\S1.1--1.3]{bonnafeCompactificationVarietesDeligneLusztig2009}.
    
    % For the smooth compactification associated with a chosen reduced expression
    % \( w\), let \(\overline{ \mathbf{X}}( w)\) be the variety
    % constructed in the next section.  We define
    % \(\overline{ \mathbf{Y}}( w)\) to be the normalization of
    % \(\overline{ \mathbf{X}}( w)\) in
    % \({\mathbf{Y}}(w)\).  The action of \( T^{wF}\) extends uniquely to
    % this normalization, and the finite map
    % \[
    % \overline\pi_w:\overline{ \mathbf{Y}}(w)
    % \longrightarrow\overline{ \mathbf{X}}(w)
    % \]
    % restricts to \(\pi_w\) over the open stratum.  Hence it is \'{e}tale over
    % \( \mathbf{X}(w)\), while any ramification is supported on the boundary.
    % Since \( T^{wF}\) is the group of rational points of a torus, its
    % order is prime to \(p\), and the boundary ramification is therefore tame.  The
    % normalization need not be smooth, and its  singularities are
    % precisely the ones analyzed later in this chapter.
    
    \subsection{The smooth compactification and its boundary}
    \label{sec:smooth-compactification-boundary}
    
    Fix a reduced expression
    \(
    w=s_1\cdots s_r.
    \)
    Let \(\mathcal B=\operatorname{Bor}( \mathbf{G})\).  Following
    Deligne--Lusztig \cite[\S\S9.1--9.2]{deligneRepresentationsReductiveGroups1976},
    consider the Bott--Samelson--Demazure--Hansen variety
    \begin{equation}
    \overline{\mathcal O}(w)
    =
    \left\{
    (\mathcal B_0,\ldots,\mathcal B_r):
    \begin{array}{l}
    \operatorname{inv}(\mathcal B_{i-1},\mathcal B_i)
    \in\{1,s_i\}\quad(1\leq i\leq r)
    \end{array}
    \right\}.
    \label{eq:prelim-BSDH-orbit}
    \end{equation}
    Here each \(\mathcal B_i\) is a Borel subgroup.  The condition in
    \eqref{eq:prelim-BSDH-orbit} means that
    \((\mathcal B_{i-1},\mathcal B_i)\) lies in the closure
    \(\overline{\mathcal O(s_i)}=\mathcal O(1)\cup\mathcal O(s_i)\)
    of the Bruhat orbit corresponding to \(s_i\).  Projection to
    \(\mathcal B_0\) exhibits \(\overline{\mathcal O}(w)\) as an
    iterated \(\mathbb P^1\)-bundle over \(\mathcal B\).  In particular, it
    is smooth and projective.  Its endpoint morphism is
    \[
    \Gamma_w\colon
    \overline{\mathcal O}(w)\longrightarrow
    \mathcal B\times\mathcal B,
    \qquad
    (\mathcal B_0,\ldots,\mathcal B_r)\longmapsto
    (\mathcal B_0,\mathcal B_r).
    \]
    Let \(\Gamma_F\subset\mathcal B\times\mathcal B\) be the graph of
    \(F\).  The compactification used by Deligne and Lusztig is the fiber
    product
    \begin{equation}
    \overline{ \mathbf{X}}(w)
    =
    \overline{\mathcal O}(w)
    \times_{\mathcal B\times\mathcal B}\Gamma_F.
    \label{eq:prelim-BSDH}
    \end{equation}
    Equivalently, it is the closed subvariety of
    \eqref{eq:prelim-BSDH-orbit} defined by
    \(\mathcal B_r=F(\mathcal B_0)\).  
    % This fiber-product description is
    % the point at which the transversality argument of
    % \cite[Lemma~9.11]{deligneRepresentationsReductiveGroups1976} enters.
    
    \medskip
    
    For a subset \(I\subseteq\{1,\ldots,r\}\), put
    \[
    x_{I,i}=
    \begin{cases}
    1,&i\in I,\\
    s_i,&i\notin I,
    \end{cases}
    \]
    and define the locally closed stratum
    \[
     \mathbf{X}_I
    =
    \left\{
    (\mathcal B_0,\ldots,\mathcal B_r)
    \in\overline{ \mathbf{X}}(w):
    \operatorname{inv}(\mathcal B_{i-1},\mathcal B_i)=x_{I,i}
    \text{ for every }i
    \right\}.
    \]
    One then has the disjoint stratification
    \begin{equation}
    \overline{ \mathbf{X}}(w)
    =\bigsqcup_{I\subseteq\{1,\ldots,r\}}
     \mathbf{X}_I,
    \qquad
     \mathbf{X}(w)= \mathbf{X}_{\varnothing}.
    \label{eq:prelim-BSDH-stratification}
    \end{equation}
    For \(1\leq i\leq r\), put
    \begin{equation}
    D_i
    =
    \left\{
    \operatorname{inv}(\mathcal B_{i-1},\mathcal B_i)=1
    \right\}
    =
    \bigsqcup_{\substack{I\subseteq\{1,\ldots,r\}\\
    i\in I}}
     \mathbf{X}_I.
    \label{eq:prelim-boundary-divisor}
    \end{equation}
    The divisor \(D_i\) is the boundary condition obtained by replacing
    the \(i\)-th reflection in the chosen reduced expression for \(w\) by
    the identity.  It need not
    be connected after imposing \(\mathcal B_r=F(\mathcal B_0)\).
    
    \begin{proposition}[Smooth compactification]
    \label{prop:prelim-smooth-compactification}
    The variety \(\overline{ \mathbf{X}}(w)\) is smooth, projective, and
    of pure dimension \(r=\ell(w)\), and \( \mathbf{X}(w)\) is a dense open
    subvariety.  Its boundary
    \[
    \partial\overline{ \mathbf{X}}(w)
    =\overline{ \mathbf{X}}(w)\setminus \mathbf{X}(w)
    =D_1\cup\cdots\cup D_r
    \]
    is a simple normal-crossings divisor.  More precisely, for
    \(I\subseteq\{1,\ldots,r\}\), every nonempty intersection
    \[
    D_I:=\bigcap_{i\in I}D_i
    \]
    is smooth of codimension \(|I|\) in
    \(\overline{ \mathbf{X}}(w)\).
    
    Projection to the first Borel subgroup defines a proper morphism
    \[
    \rho_w\colon
    \overline{ \mathbf{X}}(w)
    \longrightarrow
    \mathcal B,
    \qquad
    (\mathcal B_0,\ldots,\mathcal B_r)\longmapsto\mathcal B_0.
    \]
    It restricts to the usual identification of the open stratum with
    \( \mathbf{X}(w)\), and its image is the closure of \( \mathbf{X}(w)\) in
    \(\mathcal B\).  Hence it gives a proper birational morphism from
    \(\overline{ \mathbf{X}}(w)\) to that naive closure.
    \end{proposition}
    
    \begin{proof}
    The boundary divisors of
    \(\overline{\mathcal O}(w)\) and all of their intersections are
    smooth.  Deligne and Lusztig prove that \(\Gamma_F\) is transverse to
    \(\Gamma_w\) and to the restriction of
    \(\Gamma_w\) to every such intersection
    \cite[Lemma~9.11]{deligneRepresentationsReductiveGroups1976}.
    Taking the fiber product with \(\Gamma_F\) therefore preserves
    smoothness and the normal-crossings boundary.  The dimension is \(r\), and projectivity follows from the
closed embedding
\(\overline {\mathbf{X}}(w)\hookrightarrow\mathcal B^{r+1}\).  On the locus where every relative position is
    \(s_i\), the chain is uniquely determined by its first Borel subgroup;
    this identifies the open stratum with \( \mathbf{X}(w)\).  The final
    assertion follows by properness and density of this open stratum.  See
    also the formulation in
    \cite[\S1.2]{bonnafeCompactificationVarietesDeligneLusztig2009}.
    \end{proof}
    
    For the length-two reduced expression \(w=s_1s_2\), the deepest stratum is
    \[
     \mathbf{X}_{\{1,2\}}
    =D_1\cap D_2
    =\{(\mathcal B,\mathcal B,\mathcal B):F(\mathcal B)=\mathcal B\}.
    \]
    Thus its points are the \(F\)-fixed Borel subgroups.  
    % At every such point
    % there is an \'{e}tale neighborhood with regular parameters \(u,v\) for
    % which
    % \[
    % D_1=(u=0),
    % \qquad
    % D_2=(v=0).
    % \]
    % There is nevertheless a source-level description of the ramification at
    % the generic point of each \(D_i\).  
    Let \(\alpha_i^\vee\) be the co-root
    corresponding to \(s_i\), and put
    \[
    \beta_i=s_1\cdots s_{i-1}(\alpha_i^\vee)\in Y( T).
    \]
     \cite[Lemma~9.13]{deligneRepresentationsReductiveGroups1976} or equivalently \cite[Corollary~1.4]{bonnafeCompactificationVarietesDeligneLusztig2009} says that the right \( T^{wF}\)-torsor over the open
    stratum has, along \(D_i\), the same ramification as the pullback by
    \(\beta_i\colon\mathbb G_m\to T\) of the Lang covering
    \[
    L_{wF}\colon T\longrightarrow T,
    \qquad
    t\longmapsto t^{-1}wF(t),
    \]
    when that pullback is extended across \(0\). This determines the tame
    local monodromy along each boundary component.  
    %At a boundary crossing the
    % two monodromy directions
    % must be combined and the resulting finite algebra normalized; that is the
    % role of the Bonnaf\'e--Rouquier construction recalled in
    % Section~\ref{sec:br-slice-singularities}.
    
    \subsection{A description in terms of flags}
    
    For \( \mathbf{G}=\mathrm{GL}_3\), recall that the variety \(\mathcal{B}\) of Borel subgroups is the full flag variety, and its rational point is a complete flag
    \((L\subset H)\), where \(L\) is a line and \(H\) is a plane in \(k^3\).
    With the upper-triangular Borel convention and
    \(S=\{s_1,s_2\}\), the Bruhat conditions for two flags
    \((L\subset H)\) and \((L'\subset H')\) are
    \begin{center}
    \small
    \begin{tabular}{c|ccccc}
    \(w\) & \(1\) & \(s_1\) & \(s_2\) & \(s_1s_2\) & \(s_2s_1\) \\
    \hline
    \(\operatorname{inv}((L,H),(L',H'))\leq w\)
    & \(L=L',\ H=H'\)
    & \(H=H'\)
    & \(L=L'\)
    & \(L'\subset H\)
    & \(L\subset H'\)
    \end{tabular}
    \end{center}
    See \cite[Lemma~1.5]{gortzDeligneLusztigVarietiesWhose2026}.  In particular,
    for \(w=s_1s_2\), we compare the two flags \((L,H)\) and \((F(L), F(H)) \) and see that the naive compactification is cut out by the condition
    \(F(L)\subseteq H\).  
    % The two boundary conditions correspond to the two
    % equalities in the chain model \eqref{eq:prelim-BSDH}, and their intersections
    % are the \(F\)-fixed complete flags.
    
    \subsection{Galois modules and Brauer traces}
    \label{sec:galois-modules-brauer}

% \medskip

In our setting, the open Deligne--Lusztig cover
\(
\mathbf{Y}(w)\longrightarrow \mathbf{X}(w)
\)
is a finite \'etale right \( T^{wF}\)-torsor.
The right \( T^{wF}\)-action extends to the normal
compactification \(\overline {\mathbf{Y}}(w)\), and the resulting
finite morphism
\(
\overline {\mathbf{Y}}(w)\longrightarrow\overline {\mathbf{X}}(w)
\)
is the quotient by \( T^{wF}\) is ramified along the
boundary. 
Set
\[
\Gamma= T^{wF},
\qquad
Z=\overline {\mathbf{Y}}(w).
\]
The order of \(\Gamma\), and hence that of every component
stabilizer, is prime to \(p\). Maschke's theorem therefore
implies that \(\bar{k}[\Gamma]\) is semisimple and identifies
\(G_0(\bar{k}[\Gamma])\) with \(K_0(\bar{k}[\Gamma])\).
For a \(\Gamma\)-equivariant coherent sheaf \(\mathcal E\)
on \(Z\), we may thus write
\[
\chi_\Gamma(Z,\mathcal E)
=
\sum_i(-1)^i[H^i(Z,\mathcal E)]
\in K_0(\bar{k}[\Gamma])
\]
This is the equivariant Euler characteristic introduced
in \eqref{eq:equivariant-euler-characteristic}, and it
agrees with the projective Euler characteristic of
\(R\Gamma(Z,\mathcal E)\). The class is
\emph{virtually free} if
\[
\chi_\Gamma(Z,\mathcal E)
\in\mathbb Z\cdot[\bar{k}[\Gamma]]
\]
% Our main results concerns \(\mathcal E=\mathcal O_Z\) and the terminology ``Galois module'' refers to the action
% of the covering group \(\Gamma\) on the coherent cohomology
% of the structure sheaf.
% % 
% % \medskip
% 
Since \(Z\) is normal, its irreducible components are
disjoint. If \(\Gamma\) acts transitively on these components,
\(Z_0\) is one component, and \(\Gamma_0\) is its stabilizer,
then
\[
H^i(Z,\mathcal O_Z)
\simeq
\operatorname{Ind}_{\Gamma_0}^{\Gamma}
H^i(Z_0,\mathcal O_{Z_0})
\]
for every \(i\). Consequently,
\begin{equation}
\chi_\Gamma(Z,\mathcal O_Z)
=
\operatorname{Ind}_{\Gamma_0}^{\Gamma}
\chi_{\Gamma_0}(Z_0,\mathcal O_{Z_0})
\label{eq:component-induction-euler-class}
\end{equation}

% \medskip

Note that as \(\Gamma\) is given by the subgroup of rational points of a torus, every element of \(\Gamma\) is \(p\)-regular, and Brauer
characters determine the classes in \(K_0(\bar{k}[\Gamma])\).
Let
\(
[\,\cdot\,]\colon
\bar{k}^{\times}\longrightarrow W(\bar{k})^{\times}
\)
denote the Teichm\"uller lift, where \(W(\bar{k})\) is
the ring of \(p\)-typical Witt vectors.
If \(\tau\in\Gamma\) acts on a finite-dimensional
\(\bar{k}\)-representation \(V\) with eigenvalues
\(\lambda_1,\ldots,\lambda_m\), its Brauer trace is
\[
\operatorname{BrTr}(\tau\mid V)
=
\sum_{j=1}^{m}[\lambda_j]
\]
The eigenvalues are roots of unity of order prime to \(p\),
and their Teichm\"uller lifts are the corresponding roots
of unity in \(W(\bar{k})\).

\medskip

Recall that the Brauer--Lefschetz number is defined by
\begin{equation}
\begin{aligned}
L_{\mathrm{Br}}(\tau,Z,\mathcal E)
&:=
\sum_i(-1)^i
\operatorname{BrTr}\bigl(\tau\mid H^i(Z,\mathcal E)\bigr),\\
L_{\mathrm{Br}}(\tau,Z)
&:=
L_{\mathrm{Br}}(\tau,Z,\mathcal O_Z).
\end{aligned}
\label{eq:def-brauer-lefschetz}
\end{equation}
Thus the function
\(\tau\mapsto L_{\mathrm{Br}}(\tau,Z,\mathcal E)\)
is precisely the Brauer character of
\(\chi_\Gamma(Z,\mathcal E)\), and its values determine
this virtual module.

\medskip

If
\(\chi_\Gamma(Z,\mathcal E)=m[\bar{k}[\Gamma]]\)
for some \(m\in\mathbb Z\), its value at every
nonidentity element is zero, because the regular
representation has character zero away from the identity.
Consequently,
\begin{equation}
\tau\neq1
\quad\text{and}\quad
L_{\mathrm{Br}}(\tau,Z,\mathcal E)\neq0
\quad\Longrightarrow\quad
\chi_\Gamma(Z,\mathcal E)
\notin\mathbb Z\cdot[\bar{k}[\Gamma]].
\label{eq:nonzero-lefschetz-free-obstruction}
\end{equation}
% In this case, \(R\Gamma(Z,\mathcal E)\) cannot be
% represented by a bounded complex of finitely generated
% free \(\bar{k}[\Gamma]\)-modules.

% For \(a\neq1\), the fixed locus \(Z^a\) is contained
% in \(Z\setminus \mathbf{Y}(w)\), since the action on \(\mathbf{Y}(w)\)
% is free. The fixed-point calculations below therefore
% concern the boundary of the compactification.
% Determining the full virtual module requires enough
% character values to recover all irreducible multiplicities.
% The calculations below use only the criterion above:
% one nonzero value at a nonidentity element suffices
% to disprove virtual freeness.

% \medskip

\begin{remark}\label{lem:cyclotomic-comparison}
In our main results in Theorems \ref{thm:intro-split-brauer} and \ref{thm:intro-unitary-brauer}, the Brauer--Lefschetz expression
has the form \(R(\widehat\omega)\), where
\(R(T)\in\mathbb Q(T)\), \(M\) is prime to \(p\), and
\(\widehat\omega\) is the Teichm\"uller lift of a primitive
\(M\)-th root of unity in \(\bar{k}\).
Write \(R(T)=P(T)/Q(T)\), with
\(P(T),Q(T)\in\mathbb Q[T]\), and assume that
\(Q(\widehat\omega)\neq0\).
Set \(\zeta_M=e^{2\pi i/M}\). The assignment
\(
\zeta_M\longmapsto\widehat\omega
\)
defines a field embedding
\(
\mathbb Q(\zeta_M)
\hookrightarrow
\operatorname{Frac}W(\bar{k}).
\)
It follows that \(Q(\zeta_M)\neq0\) and
\(
R(\widehat\omega)\neq0
\Longleftrightarrow
R(\zeta_M)\neq0.
\)
It therefore suffices to show $R(\zeta_M)\neq0$ in $\mathbb C$.
\end{remark}

    \subsection{The Lefschetz--Riemann--Roch Theorem}\label{sec:lefschetz-rr}
    
    Lefschetz--Riemann--Roch expresses an alternating trace on coherent
    cohomology as a sum of contributions supported on the fixed locus.  
    % It
    % is the algebraic geometric analogues of the analytic fixed-point
    % formulas of Atiyah--Bott and Atiyah--Segal
    % \cite{atiyahLefschetzFixedPoint1968,atiyahIndexEllipticOperators1968}.  
    Donovan
    established the algebraic formula for periodic automorphisms of smooth
    projective varieties over an algebraically closed field
    \cite[Theorem~5.4 and Corollary~5.5]{donovanLefschetzRiemannRochFormula1969}.
    Baum, Fulton, and Quart subsequently developed a
    Lefschetz--Riemann--Roch transformation for singular varieties
    \cite{baumLefschetzriemannrochSingularVarieties1979}.  Thomason later
    formulated the theorem and its coherent trace formula in equivariant
    algebraic \(K\)-theory
    %, in which localization concentrates the relevant
    %equivariant class on the fixed subscheme
    \cite{thomasonLefschetzRiemannRochTheoremCoherent1986}. We now state the theorem for smooth projective varieties.
    
    \medskip
    
  Let \(Z\) be a smooth projective variety over \(\bar{k}\), let \(G\)
be a finite group acting on \(Z\), and let \(\mathcal E\) be a
\(G\)-equivariant vector bundle. Fix an element \(h\in G\) whose
order is prime to \(p\). The fixed subscheme \(Z^h\) is then smooth.
% Indeed, let \(z\in Z^h\) be a closed point, and denote by
% \(\mathfrak m_z\) the maximal ideal of the local ring
% \(\mathcal O_{Z,z}\). Averaging over \(\langle h\rangle\) gives an
% \(h\)-equivariant \(\bar{k}\)-linear section of
% \[
% \mathfrak m_z\longrightarrow\mathfrak m_z/\mathfrak m_z^2.
% \]
% An eigenbasis of the cotangent space
% \(\mathfrak m_z/\mathfrak m_z^2\) therefore lifts to a regular
% system of parameters \(t_1,\ldots,t_n\) satisfying
% \(h^*t_i=\lambda_i t_i\). After reordering, suppose that
% \(\lambda_i=1\) precisely for \(1\leq i\leq r\).
% In the completed local ring, the ideal of \(Z^h\) is generated
% by \(t_{r+1},\ldots,t_n\), so
% \[
% \widehat{\mathcal O}_{Z^h,z}
% \simeq \bar{k}[[t_1,\ldots,t_r]].
% \]
% Thus \(Z^h\) is smooth. 
Write its decomposition into connected
components as \(Z^h=\coprod_F F\).
    
    \begin{theorem}[Lefschetz--Riemann--Roch]
    \label{thm:lefschetz-rr}
    With the preceding notation,
    \begin{equation}\label{eq:lrr-general-fixed-component}
    L_{\mathrm{Br}}(h,Z,\mathcal E)
    =
    \sum_{F\subset Z^h}
    \int_F
    \frac{
    \operatorname{ch}_h(\mathcal E|_F)\operatorname{Td}(T_F)
    }{
    \operatorname{ch}_h\!\left(\lambda_{-1}(N_{F/Z}^{\vee})\right)
    },
    \end{equation}
    \end{theorem}
    
   % We now explain the terms in
 %   \eqref{eq:lrr-general-fixed-component}.  
 where for a connected component
    \(F\subset Z^h\), its normal bundle is
    \[
    N_{F/Z}=T_Z|_F/T_F,
    \]
    and its dual \(N_{F/Z}^{\vee}\) is the conormal bundle.  The class in
    the denominator is
    \begin{equation}\label{eq:lambda-minus-one-normal-bundle}
    \lambda_{-1}(N_{F/Z}^{\vee})
    :=
    \sum_{j=0}^{\operatorname{codim}(F,Z)}
    (-1)^j[\Lambda^jN_{F/Z}^{\vee}].
    \end{equation}
    % Here \(\Lambda^jN_{F/Z}^{\vee}\) is the \(j\)-th exterior power with the induced \(h\)-action in
    % equivariant \(K\)-theory.  Thus \(\lambda_{-1}\) is the alternating
    % total exterior-power class of the conormal bundle.  
    Put \(c=\operatorname{codim}(F,Z)\).  For \(x\in F\), let
    \(\rho_1,\ldots,\rho_c\) be the eigenvalues of \(h\) on the normal
    space
    \[
    N_{F/Z,x}=T_xZ/T_xF.
    \]
    The eigenvalues on the conormal space \(N_{F/Z,x}^{\vee}\) are then
    \(\rho_1^{-1},\ldots,\rho_c^{-1}\), and
    \[
    \operatorname{BrTr}\!\left(
    h\mid\lambda_{-1}(N_{F/Z,x}^{\vee})
    \right)
    =
    \prod_{i=1}^c(1-[\rho_i]^{-1}).
    \]
    % Since \(F\) is connected, this multiset of normal eigenvalues is
    % independent of \(x\).  
    % None of the \(\rho_i\) is one, otherwise the
    % normal space would contain an additional tangent direction to the fixed
    % locus.  
    % The displayed product is the fibrewise character, equivalently
    % the degree-zero term of
    % \(\operatorname{ch}_h(\lambda_{-1}(N_{F/Z}^{\vee}))\). The full class
    % also contains the Chern-class data of the normal eigenbundles.  Its
    % nonzero degree-zero term makes the denominator in
    % \eqref{eq:lrr-general-fixed-component} invertible in the localized
    % coefficient ring.  The bundle \(\mathcal E|_F\) is the restriction of
    % \(\mathcal E\) to \(F\).  
%   On the fixed locus \(Z^h\) the automorphism \(h\) acts fibrewise on
% \(E|_{X^h}\) with order prime to \(p\), so \(E|_{X^h}\) splits as a
% direct sum \(\bigoplus_{\zeta}E_\zeta\) of \(h\)-eigenbundles. The
% equivariant Chern character is the eigenvalue-weighted sum
% \[
% \operatorname{ch}_h(E)
% =\sum_{\zeta}\widehat\zeta\,\operatorname{ch}(E_\zeta)
% \in CH^*(X^h)\otimes\operatorname{Frac}W(\bar k).\]
 % \(\operatorname{Td}(T_F)\) is the Todd class of the
 %    tangent bundle of \(F\), and \(\int_F\) means taking the degree of the
 %    component of codimension \(\dim F\) in the resulting Chow class.  For
 %    the structure sheaf used in this paper,
 %    \(\operatorname{ch}_h(\mathcal O_Z|_F)=1\).
    
 %    \medskip
    
    We use the specialization of the theorem to \(\mathcal E=\mathcal O_Z\)
    when \(Z\) is a surface in which case its fixed locus is a disjoint union of
    isolated points \(P\) and smooth curves \(D\).  If
    \(\lambda_1,\lambda_2\) are the eigenvalues of \(h\) on the conormal
    space \(T_P^\vee Z=\mathfrak m_P/\mathfrak m_P^2\), then the contribution
    of \(P\) is
    \begin{equation}\label{eq:brauer-point}
    a_{\mathrm{Br}}(P)
    =
    \frac{1}{(1-[\lambda_1])(1-[\lambda_2])}.
    \end{equation}
    If \(D\) has genus \(g_D\), self-intersection \(D^2\), and normal
    eigenvalue \(\nu\), then its contribution is
    \begin{equation}\label{eq:brauer-curve}
    b_{\mathrm{Br}}(D)
    =
    \frac{1-g_D}{1-[\nu]^{-1}}
    -
    \frac{[\nu]^{-1}D^2}{(1-[\nu]^{-1})^2}
    =
    \frac{[\nu](1-g_D)}{[\nu]-1}
    -
    \frac{[\nu]D^2}{([\nu]-1)^2}.
    \end{equation}
    Consequently,
    \[
    L_{\mathrm{Br}}(h,Z)
    =
    \sum_P a_{\mathrm{Br}}(P)+\sum_D b_{\mathrm{Br}}(D).
    \]
    Here \(1-g_D=\chi(D,\mathcal O_D)\), while \(D^2\) is the degree of the
    normal line bundle \(N_{D/Z}\).  Thus these two terms are precisely the
    degree-one parts of the Todd class and the inverse conormal factor in
    \eqref{eq:lrr-general-fixed-component}.  All denominators are units
    because the displayed eigenvalues are nontrivial roots of unity of order
    prime to \(p\).
    
    \medskip
    
    In this paper, the above theorem is applied to the equivariant minimal
    resolution \(\widetilde {\mathbf{Y}}'\) of a component of the normal
    compactification.  The required input is therefore concrete: the
    isolated fixed points and their local eigenvalues, the fixed curves and
    their normal eigenvalues, and the genus and self-intersections of those curves.  
    %The toric calculations determine the resolution and the local
    % eigenvalues, while the global boundary geometry determines the incidence
    % numbers, genera, and intersection numbers.  The equivariant Leray
    % argument of Section~\ref{sec:equivariant-resolution-comparison}
    % identifies the coherent cohomology of \(Z\) with that of
    % \(\widetilde Z\), so the value computed on the smooth surface is the
    % Brauer--Lefschetz number of the original normal surface.
    
    \subsection{Toric geometry preliminaries}
    \label{sec:toric-surface-preliminaries}
    
    We recall the results from toric geometry used to normalize and resolve the local surface
    models in both the split and the non-split cases.  Our conventions follow
    \cite[Chapters~1--4 and~\S10.2]{coxToricVarieties2011a}.
    
    Let \(N\simeq\mathbb Z^2\) be a lattice, let
    \(M=\operatorname{Hom}(N,\mathbb Z)\) be its dual, and put
    \(N_{\mathbb R}=N\otimes_{\mathbb Z}\mathbb R\).  A strongly convex
    rational polyhedral cone \(\sigma\subset N_{\mathbb R}\) determines the
    normal affine toric surface
    \begin{equation}\label{eq:toric-affine-cone}
    U_\sigma
    =
    \operatorname{Spec}k[S_\sigma],
    \qquad
    S_\sigma:=\sigma^\vee\cap M.
    \end{equation}
    Conversely, suppose that we have a semigroup
    algebra \(k[S]\), where \(S\subset M\) is finitely generated and spans
    \(M\) as a group.  If \(C=\operatorname{Cone}(S)\subset M_{\mathbb R}\),
    then its normalization is
    \begin{equation}\label{eq:toric-semigroup-normalization}
    \operatorname{Spec}k[C\cap M].
    \end{equation}
    Normalization can also be described in terms of semigroups, see~\cite[Theorem~1.3.5]{coxToricVarieties2011a}.
Let \(S\) be a finitely generated submonoid of a lattice, put
\(M_S=\mathbb ZS\), and consider
\(V=\operatorname{Spec}\bar{k}[S]\).
The saturation of \(S\) in \(M_S\) is
\[
    S^{\mathrm{sat}}
    :=
    \{m\in M_S : nm\in S
      \text{ for some }n\in\mathbb Z_{>0}\}
    =
    \operatorname{Cone}(S)\cap M_S.
\]
The integral closure of \(\bar{k}[S]\) in its field of fractions
is \(\bar{k}[S^{\mathrm{sat}}]\). Hence the inclusion
\(S\subseteq S^{\mathrm{sat}}\) induces the normalization map
\[
    \operatorname{Spec}\bar{k}[S^{\mathrm{sat}}]
    \longrightarrow
    \operatorname{Spec}\bar{k}[S],
\]
which is finite, birational, and equivariant for the common
dense torus \(\operatorname{Spec}\bar{k}[M_S]\).
In particular, \(V\) is normal if and only if
\(S=S^{\mathrm{sat}}\); see
\cite[Theorem~1.3.5 and Proposition~1.3.8]{coxToricVarieties2011a}.
Thus, to normalize an affine toric model, one replaces its
defining semigroup by all lattice points of its cone in
the lattice generated by that semigroup.
    % Equivalently, normalization replaces \(S\) by its saturation.  This is
    % the passage from each binomial surface equation below to its normal
    % affine toric model; see
    
   The orbit--cone correspondence identifies the torus orbits of
\(U_\sigma\) with the faces of \(\sigma\): a face \(\tau\)
corresponds to an orbit of dimension \(2-\dim\tau\).
In particular, the zero face corresponds to the dense torus,
the rays correspond to one-dimensional orbits, and \(\sigma\)
itself corresponds to the unique closed orbit, a point denoted
by \(o_\sigma\).

Let \(u_0,u_{r+1}\) be the primitive generators of the two rays
of \(\sigma\). The surface \(U_\sigma\) is smooth if and only if
these generators form a basis of \(N\), or equivalently,
\begin{equation}\label{eq:toric-smoothness-determinant}
    |\det(u_0,u_{r+1})|=1,
\end{equation}
where the determinant is computed with respect to a lattice
basis of \(N\). More generally, choose primitive lattice vectors
\(u_1,\ldots,u_r\) in the interior of \(\sigma\), ordered from
\(u_0\) to \(u_{r+1}\), and let \(\Sigma\) be the fan whose maximal
cones are
\[
    \sigma_i=\operatorname{Cone}(u_{i-1},u_i),
    \qquad 1\leq i\leq r+1.
\]
The subdivision induces a proper birational toric morphism
\[
    \pi_\Sigma\colon X_\Sigma\longrightarrow U_\sigma,
\]
which is an isomorphism away from \(o_\sigma\).
If every consecutive pair satisfies
\[
    |\det(u_{i-1},u_i)|=1,
\]
then \(X_\Sigma\) is smooth, and \(\pi_\Sigma\) is a resolution
of singularities. Each maximal cone \(\sigma_i\) gives an
affine chart \(U_{\sigma_i}\simeq\mathbb A^2\).

For the singularities occurring in this paper, choose a basis
\(e_1,e_2\) of \(N\) in which
\begin{equation}\label{eq:toric-cox-normal-form}
    \sigma=\operatorname{Cone}(e_2,de_1-ae_2),
    \qquad
    0<a<d,
    \qquad
    \gcd(a,d)=1.
\end{equation}
The negative Hirzebruch--Jung continued fraction
\begin{equation}\label{eq:toric-hj-general}
    \frac da
    =[[b_1,\ldots,b_r]]
    :=
    b_1-\cfrac{1}{b_2-\cfrac{1}{\ddots-\cfrac{1}{b_r}}},
    \qquad b_i\geq2,
\end{equation}
determines a regular subdivision as follows. Starting with
\(u_0=e_2\) and \(u_1=e_1\), define the remaining vectors by
\begin{equation}\label{eq:toric-wall-relations}
    u_{i-1}+u_{i+1}=b_i u_i
    \qquad(1\leq i\leq r).
\end{equation}
Then \(u_{r+1}=de_1-ae_2\), the vectors
\(u_1,\ldots,u_r\) are primitive and lie in the interior of
\(\sigma\), and every consecutive pair forms a basis of \(N\).
Thus they give a smooth toric resolution as above; see
\cite[Theorems~10.2.3 and~10.2.5]{coxToricVarieties2011a}.

Write \(\rho_i=\mathbb R_{\geq0}u_i\). The exceptional divisors
\[
    E_i=V(\rho_i)\subset X_\Sigma,
    \qquad 1\leq i\leq r,
\]
are isomorphic to \(\mathbb P^1\) and form a chain.
Consecutive curves meet transversely at one point, and their
intersection numbers are
\begin{equation}\label{eq:toric-exceptional-intersections}
    E_i^2=-b_i,
    \qquad
    E_i\cdot E_{i+1}=1,
    \qquad
    E_i\cdot E_j=0\quad(|i-j|>1);
\end{equation}
see \cite[Theorem~10.4.4]{coxToricVarieties2011a}.
Since \(b_i\geq2\), no exceptional curve has self-intersection
\(-1\), so this is the minimal resolution.

The same fan describes the torus action on the resolution.
On the chart associated with
\(\sigma_i=\operatorname{Cone}(u_{i-1},u_i)\), let
\(m_i,n_i\in M\) be the basis dual to \(u_{i-1},u_i\).
The corresponding characters
\[
    x_i=\chi^{m_i},
    \qquad
    y_i=\chi^{n_i}
\]
are coordinates on \(U_{\sigma_i}\simeq\mathbb A^2\).
For a torus element \(t\), denote translation by \(t\) by
\(r_t\). Its pullback on character functions is
\[
    r_t^*(\chi^m)=\chi^m(t)\chi^m.
\]
Thus, on this chart,
\[
    r_t^*x_i=\chi^{m_i}(t)x_i,
    \qquad
    r_t^*y_i=\chi^{n_i}(t)y_i.
\]
In particular, the fixed subscheme is defined by
\[
    \bigl(\chi^{m_i}(t)-1\bigr)x_i=0,
    \qquad
    \bigl(\chi^{n_i}(t)-1\bigr)y_i=0.
\]
These equations determine the fixed points and fixed curves
directly from the character values.

The fan also determines the orders of characters along the
torus-invariant divisors. For \(m\in M\), the character
\(\chi^m\), regarded as a rational function on \(X_\Sigma\),
has divisor
\begin{equation}\label{eq:toric-character-divisor}
    \operatorname{div}_{X_\Sigma}(\chi^m)
    =
    \sum_{\rho\in\Sigma(1)}
    \langle m,u_\rho\rangle D_\rho,
\end{equation}
where \(\Sigma(1)\) is the set of rays of \(\Sigma\),
\(u_\rho\) is the primitive generator of \(\rho\), and
\(D_\rho=V(\rho)\) is the corresponding torus-invariant prime
divisor; see
\cite[Propositions~4.1.1--4.1.2]{coxToricVarieties2011a}.
The coefficient \(\langle m,u_\rho\rangle\) is the order of
vanishing, or the negative of the order of pole, of
\(\chi^m\) along \(D_\rho\). We use these orders in the
pullback and ramification calculations below.

\medskip

The exceptional curves \(E_i\) are complete, and their
self-intersections are determined by the local fan.
By contrast, the divisors associated with the two extremal
rays are noncomplete in \(X_\Sigma\). They are the strict
transforms of the two torus-invariant boundary divisors
of \(U_\sigma\), and describe only local portions of the
corresponding curves on the global surface. Their global
self-intersections therefore cannot be recovered from this
local fan alone. We compute those numbers using the global
boundary divisors and the projection formula in
Sections~\ref{sec:split-intersections}
and~\ref{sec:unitary-intersections}.
    % In the split case, the cone has parameters
    % \((d,a)=(p^2+p+1,p^2)\).  In the non-split case, it has parameters
    % \((d,a)=(p^2-1,p)\).  The case chapters compute the corresponding
    % continued fractions, rays, affine charts, and stabilizer characters.
    
    \subsection{The Bonnaf\'e--Rouquier normal compactification}
    \label{sec:br-slice-singularities}
    
    We first recall the main theorems and the precise
    ingredients from \cite[\S2.4]{bonnafeCompactificationVarietesDeligneLusztig2009}
    that enter the local calculations.
    
%     \medskip

%     We first recall the compactification theorem and the ingredients from
% \cite[\S2.4]{bonnafeCompactificationVarietesDeligneLusztig2009}
% that enter the local calculation. We begin by describing the boundary
% strata and fixing the notation.

\medskip

\noindent\textbf{Notation and boundary strata.}
% Bonnaf\'e--Rouquier distinguish the word
% \( w=(s_1,\ldots,s_r)\) from its product
% \(w=s_1\cdots s_r\) in the Weyl group. We fix a reduced expression
% \(
% w=s_1\cdots s_r
% \)
% and write \(\overline{ \mathbf{X}}(w)\) and
% \(\overline{ \mathbf{Y}}(w)\) for the associated compactifications.

% \medskip

Let \(\mathcal B\) be the variety of Borel subgroups of \( \mathbf{G}\),
and write \(\operatorname{pos}(B,B')\in W\) for the relative position
of two Borel subgroups. The compactification \(\overline{\mathbf{X}}(w)\) of \( \mathbf{X}(w)\)
has the following description:
\[
\overline{ \mathbf{X}}(w)
=
\left\{
(B_0,\ldots,B_r)\in\mathcal B^{r+1}
\;\middle|\;
\begin{array}{l}
B_r=F(B_0),\\
\operatorname{pos}(B_{i-1},B_i)\in\{1,s_i\}
\quad(1\leq i\leq r)
\end{array}
\right\}.
\]
The open locus on which all relative positions are \(s_i\)
is identified with \( \mathbf{X}(w)\) by projection to \(B_0\).
Allowing relative position \(1\), equivalently
\(B_{i-1}=B_i\), adds the boundary.

\medskip
 For words \( x=(x_1,\ldots,x_r)\) and \( w =(s_1,\ldots,s_r) \), we say \( x \preccurlyeq w\)
if \( x_i \preccurlyeq s_i \) in the Bruhat order. For \( x\preccurlyeq w \), we define
    \(I_{{x}}=\{i:x_i=1\}\subseteq\{1,\ldots,r\}\). 
For \(I\subseteq\{1,\ldots,r\}\), put
\[
x_i=
\begin{cases}
1,&i\in I,\\
s_i,&i\notin I,
\end{cases}
\qquad
x_I=x_1\cdots x_r,
\]
and define the locally closed stratum
\[
 \mathbf{X}_I
=
\left\{
(B_0,\ldots,B_r)\in\overline{ \mathbf{X}}(w)
\;\middle|\;
\operatorname{pos}(B_{i-1},B_i)=x_i
\quad(1\leq i\leq r)
\right\}.
\]
Thus
\[
\overline{ \mathbf{X}}(w)
=
\coprod_{I\subseteq\{1,\ldots,r\}} \mathbf{X}_I,
\qquad
 \mathbf{X}_{\varnothing}\simeq \mathbf{X}(w),
\]
and the nonempty subsets \(I\) index the boundary strata.

\medskip

For type \(A_{n-1}\), the simple reflection \(s_i\) is the
transposition \((i\ i+1)\). In the rank-two case considered in this paper,
\[
w=s_1s_2=(12)(23)=(123)
\]
Here every subexpression
is reduced after its identity terms are omitted, and
\[
 \mathbf{X}_{\varnothing}\simeq \mathbf{X}(s_1s_2),
\qquad
 \mathbf{X}_{\{1\}}\simeq \mathbf{X}(s_2),
\qquad
 \mathbf{X}_{\{2\}}\simeq \mathbf{X}(s_1),
\qquad
 \mathbf{X}_{\{1,2\}}\simeq \mathbf{X}(1).
\]

% Using fixed representatives of the simple reflections, we denote
% by \( \mathbf{Y}_I\to \mathbf{X}_I\) the Deligne--Lusztig covering
% associated with the sequence \((x_1,\ldots,x_r)\). It is a finite
% \'etale torsor under \( T^{x_IF}\).
% Let
% \[
% \overline\pi_w:
% \overline{ \mathbf{Y}}(w)\longrightarrow
% \overline{ \mathbf{X}}(w)
% \]
% be the finite morphism obtained by normalization. The strata
% upstairs are
% \[
% \bigl(\overline\pi_w^{-1}( \mathbf{X}_I)\bigr)_{\mathrm{red}}.
% \]
% These are generally quotients of the covering varieties
% \( \mathbf{Y}_I\), rather than the covering varieties themselves;
% the precise quotient is given in the theorem below.
    
    % \noindent\textbf{Translation of notation.}
    % Bonnaf\'e--Rouquier write
    % \( w=(s_1,\ldots,s_r)\) for a word in simple reflections and
    % write \(w=s_1\cdots s_r\) for its product in the Weyl group.  In this
    % paper we retain only the latter notation. That is, after choosing a reduced
    % expression, we write
    % \[
    % w=s_1\cdots s_r
    % \]
    % and denote the associated compactifications by boldfont symbols
    % \(\overline{ \mathbf{X}}(w)\) and \(\overline{ \mathbf{Y}}(w)\).  For type \(A_{n-1}\), the simple
    % reflection \(s_i\) is the transposition \((i\ i+1)\).  In the
    % rank-two case used below, with our multiplication convention,
    % \[
    % w=s_1s_2=(12)(23)=(123).
    % \]
   % We write \( \mathbf{X}_I\) and
   %  \( \mathbf{Y}_I\) for the corresponding stratum and covering variety, and put
   %  \(x_I=x_1\cdots x_r\) when the Weyl-group product is required.
    
    \subsubsection{The compactification \(\overline{\mathbf{Y}}(w)\) of Bonnaf\'e--Rouquier}
    
  Let \( \mathbf{G}\) be a connected reductive group over \(\bar{k}\), and
    let \(F: \mathbf{G}\to \mathbf{G}\) be an isogeny some positive power of which
    is a Frobenius endomorphism.  Fix an \(F\)-stable Borel subgroup
    \( B\), an \(F\)-stable maximal torus
    \( T\subseteq B\), and write \( U=R_u( B)\)
    for the unipotent radical.  Let
    \[
    W=N_{ \mathbf{G}}( T)/ T
    \]
    be the Weyl group.  The choice of \( B\) determines a set
    \(\Delta\) of simple roots and the corresponding set
    \(S=\{s_\alpha:\alpha\in\Delta\}\) of simple reflections.  For
    \(\alpha\in\Delta\), write \(\alpha^\vee\in Y( T)\) for the
    associated coroot, where
    \[
    Y( T)=\operatorname{Hom}(\mathbb G_m, T)
    \]
    is the cocharacter lattice, and fix a representative
    \(\dot s_\alpha\in N_{ \mathbf{G}}( T)\) of \(s_\alpha\).  The
    endomorphism induced by \(F\) on \(W\), \( T\), and
    \(Y( T)\) is again denoted by \(F\).  Thus, for \(v\in W\), the
    notation \(vF\) denotes the corresponding twisted endomorphism; on
    \( T\), it is
    \[
    t\longmapsto \dot vF(t)\dot v^{-1},
    \]
    where \(\dot v\in N_{ \mathbf{G}}( T)\) represents \(v\).  In
    particular, \( T^{vF}\) denotes its finite fixed-point group.
    
    Fix the reduced expression \(w=s_1\cdots s_r\).  For
    \(I\subseteq\{1,\ldots,r\}\), let \(x_i=1\) for \(i\in I\) and
    \(x_i=s_i\) for \(i\notin I\), and put
    \[
    Y_{w,I}
    =
    \sum_{i\in I}
    \mathbb Zs_1\cdots s_{i-1}(\alpha_i^\vee)
    \subseteq Y( T).
    \]
    Choose \(d\geq1\) and a power \(q\) of \(p\) such that
    \((vF)^d(t)=t^q\) for every \(v\in W\) and \(t\in T\), and fix a
    primitive \((q-1)\)-st root \(\epsilon\).  The norm map used below is
    \[
    N_w:Y( T)\longrightarrow T^{wF},
    \qquad
    N_w(\lambda)
    =
    \prod_{j=0}^{d-1}(wF)^j\bigl(\lambda(\epsilon)\bigr).
    \]
    It is surjective and induces
    \[
    Y( T)/(wF-1)Y( T)
    \xrightarrow{\sim} T^{wF}.
    \]
    For \(x_I=x_1\cdots x_r\), the notation \(N_{x_I}\) has the analogous
    meaning.
    
    % The symbol \(q\) is retained only in this general construction: it is an
    % auxiliary power of \(p\) determined by the common iterate \(d\), and need
    % not equal \(p\).  All explicit split and non-split calculations below are
    % written directly in terms of \(p\).
    
    \begin{theorem}[Bonnaf\'e--Rouquier, Theorem~1.2]
    \label{thm:bonnafe-rouquier}
    Let \(\overline{ \mathbf{Y}}(w)\) be the normalization of the compactification
    \(\overline{ \mathbf{X}}(w)\) in \( \mathbf{Y}(w)\), and let
    \(\overline\pi_w\) be the resulting finite morphism.
    Then the following hold.
    \begin{enumerate}
    \item[\textup{(a)}] The variety \(\overline{ \mathbf{Y}}(w)\) is projective,
    normal, and rationally smooth, and
    \[
    \operatorname{Sing}\overline{ \mathbf{Y}}(w)
    \subseteq
    \overline\pi_w^{-1}\!\left(
    \bigcup_{\substack{I\subseteq\{1,\ldots,r\}\\
    |I|\geq2}}
     \mathbf{X}_I
    \right).
    \]
    
    \item[\textup{(b)}] The action of \( T^{wF}\) on \( \mathbf{Y}(w)\)
    extends to \(\overline{ \mathbf{Y}}(w)\), and
    \[
    \overline{ \mathbf{Y}}(w)/ T^{wF}
    \xrightarrow{\sim}
    \overline{ \mathbf{X}}(w).
    \]
    
    \item[\textup{(c)}] For \(I\subseteq\{1,\ldots,r\}\), the stabilizer in
    \( T^{wF}\) of every point of
    \(\overline\pi_w^{-1}( \mathbf{X}_I)\) is
    \(N_w(Y_{w,I})\).
    
    \item[\textup{(d)}] For \(I\subseteq\{1,\ldots,r\}\), there is a canonical
    morphism
    \[
    i_I: \mathbf{Y}_I
    \longrightarrow
    \overline\pi_w^{-1}( \mathbf{X}_I)
    \]
    over \( \mathbf{X}_I\), and it induces an isomorphism
    \[
     \mathbf{Y}_I/N_{x_I}(Y_{w,I})
    \xrightarrow{\sim}
    \overline\pi_w^{-1}( \mathbf{X}_I).
    \]
    \end{enumerate}
    \end{theorem}
    
    Part~\textup{(c)} identifies the inertia group along each boundary
    stratum. The finite stabilizer arising from the auxiliary torus quotient
    is described in Proposition~\ref{prop:br-boundary-stabilizers}; it is this
    group that acts on the normal slice.
    
    \subsubsection{The auxiliary construction and its stabilizers}
    
    For every \(i\), Bonnaf\'e--Rouquier choose the unique primitive
    cocharacter \(\lambda_i\in Y( T)\) and unique positive integer
    \(m_i\) satisfying
    \begin{equation}
    \lambda_i-wF(\lambda_i)
    =m_i s_1\cdots s_{i-1}(\alpha_i^\vee),
    \qquad
    Y( T)/\mathbb Z\lambda_i\text{ is torsion-free}.
    \label{eq:br-cocharacter-equation}
    \end{equation}
    The defining equality implies \(m_i\mid(q-1)\); in particular,
    \(p\nmid m_i\) \cite[Remark~2.4]{bonnafeCompactificationVarietesDeligneLusztig2009}.
    
    Let \(\varphi_{\alpha_i}: G_{\alpha_i} U\to\mathbb A^1\)
    be the regular function used in the construction; its zero locus is
    \( B\cap G_{\alpha_i} U\), and it takes the value
    one on \( U\dot s_i U\).  The auxiliary variety consists
    of the tuples
    \[
    (g_1 U,\ldots,g_{r+1} U;\xi_1,\ldots,\xi_r)
    \]
    such that, for every \(i\),
    \begin{equation}
    g_i^{-1}g_{i+1}\in G_{\alpha_i} U,
    \qquad
    \varphi_{\alpha_i}(g_i^{-1}g_{i+1})=\xi_i^{m_i},
    \qquad
    g_{r+1} U=F(g_1) U.
    \label{eq:br-auxiliary-variety}
    \end{equation}
    We denote it by \(\widetilde {\mathbf{Y}}_{\mathrm{BR}}(w)\).  This
    Bonnaf\'e--Rouquier auxiliary variety is distinct from both the toric surfaces used below and the minimal resolutions of the
    normal components \(\overline {\mathbf{Y}}_\ell(w)\).
    
    \begin{proposition}[Bonnaf\'e--Rouquier, (2.5), Proposition~2.6 and \S2.5]
    \label{prop:br-auxiliary-quotient}
    The variety \(\widetilde {\mathbf{Y}}_{\mathrm{BR}}(w)\) is smooth and pure
    of dimension \(2r\).  It carries a right action of
    \( T^{wF}\times(\mathbb G_m)^r\), and the natural morphism
    \(\widetilde\pi_w:\widetilde {\mathbf{Y}}_{\mathrm{BR}}(w)
    \to\overline{ \mathbf{X}}(w)\) induces
    \[
    \widetilde {\mathbf{Y}}_{\mathrm{BR}}(w)
    /\bigl( T^{wF}\times(\mathbb G_m)^r\bigr)
    \xrightarrow{\sim}
    \overline{ \mathbf{X}}(w).
    \]
    Moreover,
    \[
    \overline{ \mathbf{Y}}(w)
    =\widetilde {\mathbf{Y}}_{\mathrm{BR}}(w)/(\mathbb G_m)^r,
    \]
    and the locus on which every \(\xi_i\) is nonzero is the open part
    corresponding to \( \mathbf{Y}(w)\).
    \end{proposition}
    
    For \(z=(z_1,\ldots,z_r)\in(\mathbb G_m)^r\), set
    \[
    \Lambda_1(z)=\lambda_1(z_1)\cdots\lambda_r(z_r),
    \qquad
    \Lambda_{i+1}(z)
    =s_i\bigl(\Lambda_i(z)\bigr)\alpha_i^\vee(z_i^{m_i}).
    \]
    The auxiliary torus acts by
    \[
    (g_1 U,\ldots,g_{r+1} U;\xi_1,\ldots,\xi_r)*z
    =
    (g_1\Lambda_1(z) U,\ldots,
    g_{r+1}\Lambda_{r+1}(z) U;
    z_1\xi_1,\ldots,z_r\xi_r).
    \]
    For \(I\subseteq\{1,\ldots,r\}\), define
    \[
    \widetilde {\mathbf{Y}}_I(w)
    =
    \{(g;\xi)\in\widetilde {\mathbf{Y}}_{\mathrm{BR}}(w):
    \xi_i=0\Longleftrightarrow i\in I\}
    \]
    and
    \[
    H_I
    =
    \{z\in(\mathbb G_m)^r:
    \Lambda_1(z)=\cdots=\Lambda_{r+1}(z)=1,
    \ z_i=1\text{ for }i\notin I\}.
    \]
    
    \begin{proposition}[Bonnaf\'e--Rouquier, Proposition~2.7]
    \label{prop:br-boundary-stabilizers}
    For every \(I\subseteq\{1,\ldots,r\}\), the stabilizer in
    \((\mathbb G_m)^r\) of any point of \(\widetilde {\mathbf{Y}}_I(w)\) is
    exactly \(H_I\).  The group \(H_I\) is finite and is contained in
    \(H_{\{1,\ldots,r\}}\).  If \(|I|\leq1\), then \(H_I=1\).  Furthermore,
    for every \(I\subseteq\{1,\ldots,r\}\),
    \[
    \widetilde\pi_w^{-1}( \mathbf{X}_I)
    =\widetilde {\mathbf{Y}}_I(w).
    \]
    \end{proposition}
    
    For each \(I\subseteq\{1,\ldots,r\}\), the construction in the proof
    further gives a
    \((\mathbb G_m)^r\)-equivariant isomorphism
    \[
     \mathbf{Y}_I\times
    \bigl((\mathbb G_m)^r/H_I\bigr)
    \xrightarrow{\sim}
    \widetilde {\mathbf{Y}}_I(w).
    \]
    This identifies the tangent directions along a boundary stratum with the
    Deligne--Lusztig stratum directions together with the auxiliary-torus orbit
    directions.
    
    Thus only strata with at least two vanishing boundary coordinates can
    produce singularities in the quotient, exactly as asserted in
    Theorem~\ref{thm:bonnafe-rouquier}\textup{(a)}.
    
    \subsubsection{The \'{e}tale-slice in rank two}
    
    We first recall the form of Luna's \'{e}tale slice theorem that will be
    used.  If a smooth group \(G_0\) acts on a smooth variety \(V\), the
    tangent space to the orbit \(G_0\cdot v\) accounts for the directions
    along which the quotient removes the \(G_0\)-action.  The remaining
    first-order directions form the \emph{normal representation}
    \[
    N_v:=T_vV/T_v(G_0\cdot v)
    \]
    of the stabilizer \((G_0)_v\).  Luna's theorem asserts that, when this
    stabilizer is linearly reductive, the quotient is \'{e}tale-locally
    modeled on the quotient of \(N_v\) by \((G_0)_v\).
    
    \begin{theorem}[Luna's \'{e}tale slice theorem, quotient form]
    \label{thm:luna-etale-slice-quotient}
    Let \(V\) be a quasi-separated algebraic space locally of finite type over
    the algebraically closed field \(k\), and let a smooth affine group scheme
    \(G_0\) act on \(V\).  Fix \(v\in V(k)\), put \(H=(G_0)_v\), and suppose
    that \(H\) is linearly reductive.  Then there are an affine \(H\)-scheme
    \(S\), an \(H\)-fixed point \(s\in S(k)\), and an \(H\)-equivariant
    unramified morphism
    \[
    (S,s)\longrightarrow(V,v)
    \]
    such that the induced morphism from the contracted product
    \[
    G_0\times^H S\longrightarrow V
    \]
    is \'{e}tale.
    
  If, in addition, \(V\) admits a good quotient \(\pi\colon V\to V/\!/G_0\)
and the orbit \(G_0\cdot v\) is closed in \(V\), the slice can be chosen so
that \(S/\!/H\to V/\!/G_0\) is \'{e}tale and
\[
  G_0\times^H S\simeq V\times_{V/\!/G_0}(S/\!/H).
\]
  Suppose now that \(V\) admits a good quotient, that \(G_0\cdot v\) is
closed, and that \(V\) is smooth at \(v\). The \'{e}tale morphisms
\(V/\!/G_0\leftarrow S/\!/H\rightarrow N_v/\!/H\) identify the completed
local ring of \(V/\!/G_0\) at \(\bar v=\pi(v)\) with that of \(N_v/\!/H\)
at~\(0\). If moreover \(H\) is finite of order prime to
\(p=\operatorname{char}k\), then \(H\) is linearly reductive, and
completing \(k[N_v]^H\) at the origin gives \(k[[N_v^\vee]]^H\). Hence there
is a non-canonical isomorphism
\begin{equation}
  \widehat{\mathcal O}_{V/\!/G_0,\bar v}\simeq k[[N_v^\vee]]^H .
  \label{eq:luna-completed-quotient}
\end{equation}
In particular, the \'{e}tale-local singularity type of \(V/\!/G_0\) at
\(\bar v\) is determined by the representation of \(H\) on \(N_v\).
    \end{theorem}
    
    For more details, see~\cite[Theorem~4.5 and Remark~4.6]{alperLunaEtaleSlice2020}. 
    % When \(H\)
    % is finite of order prime to \(p\), it is linearly reductive and invariants
    % commute with completion.  The theorem therefore gives a non-canonical
    % isomorphism
    % \begin{equation}
    % \widehat{\mathcal O}_{V/\!/G_0,\overline v}
    % \simeq
    % \bar{k}[[N_v^\vee]]^H,
    % \label{eq:luna-completed-quotient}
    % \end{equation}
    % where \(\overline v\) is the image of \(v\) in the quotient.  Thus both
    % the \'{e}tale singularity type and the completed local ring are determined
    % by the stabilizer representation on \(N_v\).

    \medskip
    
    We now specialize this statement to
    \( \mathbf{G}=\mathrm{GL}_3\), \(r=2\), and the reduced expression
    \(
    w=s_1s_2.
    \) Let
    \[
    Y( T)=\mathbb Ze_1\oplus\mathbb Ze_2\oplus\mathbb Ze_3,
    \qquad
    \alpha_1^\vee=e_1-e_2,
    \qquad
    \alpha_2^\vee=e_2-e_3,
    \]
    where \(s_1\) interchanges \(e_1,e_2\) and \(s_2\) interchanges
    \(e_2,e_3\).  Set
    \[
    V=\widetilde {\mathbf{Y}}_{\mathrm{BR}}(w),
    \qquad
    G_0=(\mathbb G_m)^2.
    \]
    By Proposition~\ref{prop:br-auxiliary-quotient}, \(V\) is smooth of
    dimension four and its geometric quotient by \(G_0\) is
    \(\overline {\mathbf{Y}}(w)\).  On the invariant affine neighborhoods occurring in
    the slice theorem this is a good quotient, so its quotient clause applies.
    
    Choose a point \(\widetilde P\) of the deepest boundary stratum
    \(\xi_1=\xi_2=0\).  Then the stabilizer \(H_{\{1,2\}}\) is given by
    \begin{equation}
    K_{12}
    =
    \left\{
    (z_1,z_2)\in(\mathbb G_m)^2:
    z_1^{m_1}=z_2^{m_2}=1,
    \ \lambda_1(z_1)\lambda_2(z_2)=1
    \right\}.
    \label{eq:br-deep-stabilizer}
    \end{equation}
    We write \(K_{12}\) to distinguish this auxiliary slice stabilizer from
    the component stabilizers inside \( T^{wF}\).  By
    Proposition~\ref{prop:br-boundary-stabilizers}, the stabilizers on the
    other strata are trivial.  Since \(m_i\mid(q-1)\), the finite group
    \(K_{12}\) has order prime to \(p\), so
    Theorem~\ref{thm:luna-etale-slice-quotient} applies.
    
    The \(G_0\)-orbit through \(\widetilde P\) has dimension two, and hence
    the normal representation
    \[
    N_{\widetilde P}
    =
    T_{\widetilde P}V/
    T_{\widetilde P}(G_0\cdot\widetilde P)
    \]
    has dimension two.  The two boundary parameters are transverse to the
    deepest stratum.  Indeed, for \(I=\{1,2\}\), the factor
    \( \mathbf{Y}_{\{1,2\}}\) in the preceding stratum description is
    zero-dimensional, so the tangent space to
    \(\widetilde {\mathbf{Y}}_{\{1,2\}}(w)\) at \(\widetilde P\) is the tangent space to
    the \(G_0\)-orbit.  The morphism \((\xi_1,\xi_2):V\to\mathbb A^2\) is smooth at
    \(\widetilde P\): in the proof of
    \cite[Proposition~2.3]{bonnafeCompactificationVarietesDeligneLusztig2009},
    the successive fibres have coordinates \((x,\xi)\) on
    \(\mathbb A^2\setminus\{0\}\), and fixing the \(\xi_i\) gives smooth
    fibres. The graph-transversality lemma
    \cite[Lemma~2.5]{bonnafeCompactificationVarietesDeligneLusztig2009}
    then gives smooth fibres of dimension two on \(V\). Thus the classes
    of \(\xi_1,\xi_2\) form a basis of \(N_{\widetilde P}^{\vee}\).  From the action in
    Proposition~\ref{prop:br-auxiliary-quotient}, one reads
    \begin{equation}
    (z_1,z_2)\colon(\xi_1,\xi_2)
    \longmapsto
    (z_1\xi_1,z_2\xi_2).
    \label{eq:br-normal-slice-weights}
    \end{equation}
    Consequently the two characters of \(K_{12}\) obtained by restricting
    the coordinate characters \(z_1,z_2\) are exactly the two slice weights.
    
    Suppose \(K_{12}\simeq\mu_d\) and a chosen generator is embedded as
    \[
    \eta\longmapsto(\eta^a,\eta^b).
    \]
    Then, for \(P\) the image of \(\widetilde P\) in \(\overline {\mathbf{Y}}(w)\),
    \eqref{eq:luna-completed-quotient} becomes
    \[
    \widehat{\mathcal O}_{\overline {\mathbf{Y}}(w),P}
    \simeq
    \bar{k}[[u,v]]^{\mu_d},
    \qquad
    \eta(u,v)=(\eta^a u,\eta^b v).
    \]
    When \(a\) is invertible modulo \(d\), replacing the generator by
    \(\eta^{a^{-1}}\) puts the action in the standard cyclic quotient form
    \[
    \frac1d(1,c),
    \qquad
    c\equiv a^{-1}b\pmod d.
    \]
    
    This description is evaluated in
    Section~\ref{sec:br-slice-split} for the split Frobenius and in
    Section~\ref{sec:br-slice-unitary} for the unitary Frobenius.  In each
    case the cocharacters \(\lambda_i\) determine \(K_{12}\), while
    \eqref{eq:br-normal-slice-weights} determines its two weights.  
    % The resulting cyclic quotient type is then compared with the direct toric
    % normalization and determines the local resolution fan.  The slice theorem
    % is used only for this \'{e}tale and completed local identification; the
    % global boundary intersections and the normal eigenvalues needed for
    % Lefschetz--Riemann--Roch are computed separately on the explicit surface
    % models in Sections~\ref{part:split} and~\ref{part:unitary}.

    \subsection{Equivariant comparison with the normal compactification}
    \label{sec:equivariant-resolution-comparison}
    
    \begin{proposition}[Cohomology under resolution of rational surface singularities]
    \label{ratsingprop}
    Let \(X\) be a normal projective surface with rational singularities, and
    let
    \[
    f\colon Y\longrightarrow X
    \]
    be a projective resolution of singularities.  Then the natural pullback
    maps are isomorphisms
    \begin{equation}
    H^i(X,\mathcal O_X)\xrightarrow{\ \sim\ }H^i(Y,\mathcal O_Y)
    \qquad(i\geq0).
    \label{eq:rational-singularity-cohomology}
    \end{equation}
    If a finite group \(A\) acts on \(X\), the action lifts to \(Y\), and
    \(f\) is \(A\)-equivariant, then these are isomorphisms of
    \(\bar{k}[A]\)-modules.
    \end{proposition}
    
    \begin{proof}
    Because \(X\) is normal and \(f\) is proper and birational, Zariski's
    Main Theorem gives
    \[
    f_*\mathcal O_Y=\mathcal O_X.
    \]
    The assumption that \(X\) has rational singularities gives
    \[
    R^jf_*\mathcal O_Y=0
    \qquad(j>0).
    \]
    Thus the Leray spectral sequence
    \[
    E_2^{i,j}=H^i\!\left(X,R^jf_*\mathcal O_Y\right)
    \Longrightarrow H^{i+j}(Y,\mathcal O_Y)
    \]
    has only its \(j=0\) row.  It therefore degenerates and its edge maps
    give
    \[
    H^i(Y,\mathcal O_Y)
    \cong H^i\!\left(X,f_*\mathcal O_Y\right)
    =H^i(X,\mathcal O_X)
    \]
    for every \(i\geq0\).  When \(f\) is \(A\)-equivariant, the derived
    pushforward, the Leray spectral sequence, and its edge maps are all
    \(A\)-equivariant.  The resulting isomorphisms therefore respect the
    \(\bar{k}[A]\)-module structures.
    \end{proof}

\begin{lemma}[Rational singularities]
\label{lem:tame-cyclic-quotients-rational}
Let \(X\) be a normal integral variety over a perfect field \(K\).
Suppose that there is a family of \'{e}tale morphisms
\(u_\alpha\colon U_\alpha\to X\) whose images cover \(X\), such that each
\(U_\alpha\) is a tame quotient in the sense of
\cite[Definition~4.2.6]{chatzistamatiouHigherDirectImages2011}, i.e.\
\(U_\alpha\) is integral and normal and admits a finite surjective morphism
from a smooth integral \(K\)-scheme whose degree is invertible in \(K\).
Then every resolution \(g\colon\widetilde X\to X\), that is, every proper
birational morphism with \(\widetilde X\) smooth and integral, satisfies
\[
g_*\mathcal O_{\widetilde X}=\mathcal O_X,
\qquad
R^ig_*\mathcal O_{\widetilde X}=0
\quad(i>0).
\]
In particular, \(X\) has rational singularities in the cohomological
sense used in Proposition~\ref{ratsingprop}.
\end{lemma}

\begin{proof}
Since \(X\) is integral and \(g\) is proper and birational, there is a dense
open subset \(V\subset X\) such that \(g^{-1}(V)\to V\) is an isomorphism.
Let \(u\colon U\to X\) be one of the morphisms \(u_\alpha\) and form the
cartesian square
\[
\begin{tikzcd}
\widetilde U:=\widetilde X\times_XU \arrow[r,"u'"] \arrow[d,"g_U"']
  & \widetilde X \arrow[d,"g"]\\
U \arrow[r,"u"'] & X.
\end{tikzcd}
\]
The morphism \(g_U\) is proper by base change, and \(\widetilde U\) is
smooth over \(K\) because \(u'\) is \'{e}tale and \(\widetilde X\) is smooth
over \(K\).

We claim that \(\widetilde U\) is integral and that \(g_U\) is birational.
Since \(u\) is open and \(X\) is irreducible, the nonempty open set \(u(U)\)
meets \(V\); hence \(U_V:=u^{-1}(V)\) is a nonempty, hence dense and
irreducible, open subset of \(U\), and \(g_U\) restricts to an isomorphism
\(\widetilde U_V:=g_U^{-1}(U_V)\xrightarrow{\sim}U_V\).  As \(\widetilde U\)
is smooth, it is the disjoint union of finitely many connected components,
each of which is open and integral.  For such a component \(W\), the
restriction \(u'|_W\) is \'{e}tale, so \(u'(W)\) is a nonempty open subset
of the irreducible scheme \(\widetilde X\) and therefore meets the dense
open subset \(g^{-1}(V)\).  Thus \(W\) meets
\(u'^{-1}(g^{-1}(V))=\widetilde U_V\).  Since \(\widetilde U_V\simeq U_V\) is
irreducible and the components are pairwise disjoint and open,
\(\widetilde U\) has exactly one component.  Hence \(\widetilde U\) is
integral, \(\widetilde U_V\) is dense in \(\widetilde U\), and \(g_U\) is
birational.

Consequently \(\widetilde U\) is a tame quotient (it is smooth, hence normal,
and integral; take the identity, of degree \(1\)), and \(U\) is a tame
quotient by hypothesis.  By
\cite[Corollary~4.3.12]{chatzistamatiouHigherDirectImages2011}, applied to
the proper birational morphism \(g_U\),
\[
R(g_U)_*\mathcal O_{\widetilde U}\simeq\mathcal O_U.
\]
Since \(u\) is flat and \(g\) is proper, flat base change gives
\[
u^*R^ig_*\mathcal O_{\widetilde X}
\simeq R^i(g_U)_*\mathcal O_{\widetilde U},
\]
which vanishes for \(i>0\).

Now fix \(i>0\) and put \(\mathcal F:=R^ig_*\mathcal O_{\widetilde X}\), a
coherent \(\mathcal O_X\)-module.  Given \(x\in X\), choose \(\alpha\) and
\(y\in U_\alpha\) with \(u_\alpha(y)=x\).  The local homomorphism
\(\mathcal O_{X,x}\to\mathcal O_{U_\alpha,y}\) is flat, hence faithfully
flat, and
\(\mathcal F_x\otimes_{\mathcal O_{X,x}}\mathcal O_{U_\alpha,y}
=(u_\alpha^*\mathcal F)_y=0\).  Therefore \(\mathcal F_x=0\) for every
\(x\in X\), i.e.\ \(R^ig_*\mathcal O_{\widetilde X}=0\) for \(i>0\).

Finally, \(g_*\mathcal O_{\widetilde X}\) is a coherent
\(\mathcal O_X\)-algebra, since \(g\) is proper, and it is contained in the
function field \(K(\widetilde X)=K(X)\), since \(g\) is birational and
\(\widetilde X\) is integral.  Its local sections are therefore rational
functions that are integral over \(\mathcal O_X\), and normality of \(X\)
gives \(g_*\mathcal O_{\widetilde X}=\mathcal O_X\).
\end{proof}

Let \(f\colon\widetilde Z\to Z\) denote either of the minimal resolutions
\[
\begin{aligned}
\widetilde{\overline {\mathbf{Y}}}_\ell(w)&\longrightarrow\overline {\mathbf{Y}}_\ell(w)
&&\text{in the split case},\\
\widetilde{\overline {\mathbf{Y}}}_{\delta_0}(w)&\longrightarrow\overline {\mathbf{Y}}_{\delta_0}(w)
&&\text{in the non-split case, with }\delta_0\in\mu_{p+1}\text{ fixed}.
\end{aligned}
\]
% NOTE: the argument below uses that Z is an integral proper surface over
% \bar{k}; add a reference to where this is established if it is not
% already stated nearby.
Recall that \(Z\) is an integral proper surface over \(\bar k\).
By the quotient clause of
Theorem~\ref{thm:luna-etale-slice-quotient}, every singular point
\(z\in Z\) has an \'{e}tale neighborhood
\[
u_z\colon U_z=S_z/\!/\mu_d\longrightarrow Z,
\]
where \(S_z\) is affine and
\[
d=N=p^2+p+1
\quad\text{in the split case},
\qquad
d=N_{\mathrm{un}}=p^2-1
\quad\text{in the non-split case}.
\]
% NOTE: check that the point of S_z over z is mu_d-fixed in
% Theorem~\ref{thm:luna-etale-slice-quotient}; this is what makes the
% connected component through it mu_d-stable.
After replacing \(S_z\) by a \(\mu_d\)-stable open neighborhood of the
point over \(z\) contained in its smooth locus, and then by the connected
component through that point (which is \(\mu_d\)-stable because the point
is \(\mu_d\)-fixed), we may assume that \(S_z\) is smooth, connected and
affine, hence integral.  Since \(p\nmid d\), the group scheme \(\mu_d\) is
the constant group \(\mathbb Z/d\) over \(\bar k\).  The quotient
morphism \(S_z\to U_z\) is therefore finite and surjective, \(U_z\) is
integral and normal, and the degree of \(S_z\to U_z\) equals the order of
the image of \(\mu_d\) in \(\operatorname{Aut}(S_z)\), which divides \(d\).
Since \(N\equiv1\) and \(N_{\mathrm{un}}\equiv-1\pmod p\), this degree is
invertible in \(\bar k\), so \(U_z\) is a tame quotient in the sense of
\cite[Definition~4.2.6]{chatzistamatiouHigherDirectImages2011}.  The smooth
locus \(Z^{\mathrm{sm}}\subset Z\) is a tame quotient via the identity, and
the open immersion \(Z^{\mathrm{sm}}\hookrightarrow Z\) is \'{e}tale.  Hence
the morphisms \(u_z\), together with \(Z^{\mathrm{sm}}\hookrightarrow Z\),
form an \'{e}tale cover of \(Z\) by tame quotients.  In particular \(Z\) is
normal, because normality is local in the \'{e}tale topology.
Lemma~\ref{lem:tame-cyclic-quotients-rational} therefore shows that \(Z\)
has rational singularities:
\begin{equation}\label{eq:common-derived-pushforward}
f_*\mathcal O_{\widetilde Z}=\mathcal O_Z,
\qquad
R^jf_*\mathcal O_{\widetilde Z}=0
\quad(j>0).
\end{equation}

Let \(h\) be the element under consideration, namely \(h=\sigma\) in the
split case and \(h=\tau\) in the non-split case; it acts on \(Z\) by
\(\bar k\)-automorphisms.  Since \(Z\) is a normal surface, its minimal
resolution is unique up to unique isomorphism.  Hence the automorphism
\(h\) of \(Z\) lifts uniquely to an automorphism of \(\widetilde Z\) with
\(f\circ h=h\circ f\), so \(f\) is \(\langle h\rangle\)-equivariant.
Proposition~\ref{ratsingprop}, applied to the equivariant resolution
\(f\) and combined with \eqref{eq:common-derived-pushforward}, gives
canonical isomorphisms of \(\bar k[\langle h\rangle]\)-modules
\begin{equation}\label{eq:equivariant-cohomology-comparison}
f^*\colon H^i(Z,\mathcal O_Z)
\xrightarrow{\ \sim\ }
H^i(\widetilde Z,\mathcal O_{\widetilde Z})
\qquad(i\geq0).
\end{equation}
(Explicitly, \eqref{eq:common-derived-pushforward} and the Leray spectral
sequence show that \(f^*\) is an isomorphism, and it commutes with the
action of \(h\) because \(f\circ h=h\circ f\).)
Since \(Z\) and \(\widetilde Z\) are proper surfaces over \(\bar k\), these
modules are finite-dimensional and vanish for \(i>2\).  Since the order of
\(h\) is prime to \(p\), their Brauer characters are defined at \(h\), and
taking alternating Brauer traces of
\eqref{eq:equivariant-cohomology-comparison} gives
\begin{equation}\label{eq:common-brauer-resolution-invariance}
\boxed{
L_{\mathrm{Br}}(h,Z)
=
L_{\mathrm{Br}}(h,\widetilde Z).
}
\end{equation}
Thus
\[
L_{\mathrm{Br}}(\sigma,\overline {\mathbf{Y}}_\ell(w))
=
L_{\mathrm{Br}}(\sigma,\widetilde{\overline {\mathbf{Y}}}_\ell(w))
\]
in the split case, and
\[
L_{\mathrm{Br}}(\tau,\overline {\mathbf{Y}}_{\delta_0}(w))
=
L_{\mathrm{Br}}(\tau,\widetilde{\overline {\mathbf{Y}}}_{\delta_0}(w))
\]
in the non-split case.  These equalities will be applied to the
fixed-locus calculations in
Sections~\ref{sec:split-brauer-calculation}
and~\ref{sec:unitary-brauer-calculation}, respectively, to obtain the
formulas for the normal compactifications.
    
    \section{Split Case}\label{part:split}
    
    \subsection{Split setup and main result}\label{sec:split-setup}
    Recall that \(k=\mathbb F_p\) and \(\bar{k}=\overline{\mathbb F}_p\).
    We take \( \mathbf{G}=\mathrm{GL}_3\) over
    \(\bar{k}\), with split Frobenius \(F\), upper triangular Borel
    subgroup \( B\), and diagonal torus
    \[
     T
    =\{\diag(t_1,t_2,t_3):t_i\in\bar{k}^\times\}
    \simeq(\mathbb G_m)^3.
    \]
    
    Its Weyl group is \(S_3\).
    All geometric fixed-locus and coherent-cohomology calculations are understood
    after base change to \(\bar{k}\). For the Coxeter element \(w=(123)\), choose
    the lift that cyclically permutes the diagonal entries.  Then
    \begin{equation}
    wF(t_1,t_2,t_3)=(t_3^p,t_1^p,t_2^p),
    \qquad
    (wF)^3(t_1,t_2,t_3)=(t_1^{p^3},t_2^{p^3},t_3^{p^3}).
    \label{eq:split-twisted-frobenius}
    \end{equation}
    Consequently,
    \[
     T^{wF}
    =\{\diag(a,a^p,a^{p^2}):a\in\mathbb F_{p^3}^{\times}\}
    \simeq\mathbb F_{p^3}^{\times},
    \]
    so the finite group acting on the right is cyclic of order \(p^3-1\).
    Equivalently, \(\mathbf{Y}(w)\to \mathbf{X}(w)\) is a right torsor under this cyclic group.
    Fix a primitive
    \((p^3-1)\)-st root \(\mu\in\mathbb F_{p^3}^{\times}\) and let
    \[
    \tau:=\diag(\mu,\mu^p,\mu^{p^2})
    \]
    be the corresponding generator.
    
    By Proposition~\ref{split:Ynotirrprop}, both \(\mathbf{Y}(w)\) and its normal
    compactification \(\overline {\mathbf{Y}}(w)\) are disjoint unions of \(p-1\)
    irreducible components.  The generator \(\tau\) permutes them
    transitively, whereas its \((p-1)\)-st power
    \[
    \sigma:=\tau^{p-1}
    \]
    preserves every component and has order
    \(N=p^2+p+1\).  If \(\mu\) is the primitive
    \((p^3-1)\)-st root used above, then
    \(\bar\zeta:=\mu^{1-p}\) is a primitive \(N\)-th root.  Thus the
    Teichm\"uller parameter relevant to the action on a chosen component is
    \[
    \widehat\zeta=[\bar\zeta]=[\mu^{1-p}]
    \in W(\mathbb F_{p^3})^\times.
    \]
    The Brauer--Lefschetz
    calculation for \(\sigma\) is stated in
    Theorem~\ref{thm:intro-split-brauer} and proved in
    Section~\ref{sec:split-brauer-calculation}.
    
    \subsection{Explicit description of the varieties}\label{sec:split-varieties}

    \bigskip
    
    Let \(V\) be a three-dimensional vector space over \(k=\mathbb F_p\),
    where \(p\) is prime.  Let \(\Omega_V\) denote the Drinfeld half-space,
    the complement of all \(k\)-rational lines in
    \(\mathbb P(V)=\mathbb P^2\).
    
    \begin{proposition} \label{split:Dellusdeflemma}
    The variety \(\Omega_V\) is isomorphic to the Deligne--Lusztig variety
    \(\mathbf{X}:=\mathbf{X}(w)\) corresponding to the Coxeter element \(w=(123)\in S_3\)
    for \( \mathbf{G}=\mathrm{GL}_3\).
    \end{proposition}
    
    \begin{proof}
    See \cite[\S2.1]{deligneRepresentationsReductiveGroups1976}.
    \end{proof}
    
    To describe this embedding explicitly, fix homogeneous coordinates
    \(x_0,x_1,x_2\) on \(\mathbb P(V)\simeq\mathbb P^2\) and dual
    homogeneous coordinates \(y_0,y_1,y_2\) on
    \(\mathbb P(V^*)\simeq\mathbb P^2\).  We recall the definition of Moore determinant.
    
    \medskip
    
    \begin{definition}[\textbf{The Moore determinant}]
    For indeterminates \(w_1,\ldots,w_n\), define the Moore determinant
    \(\Delta_p(w_1,\ldots,w_n)\) by
    \[
    \begin{vmatrix}
    w_1 & w_2 & \cdots & w_n \\
    w_1^p & w_2^p & \cdots & w_n^p \\
    \vdots & \vdots & \ddots & \vdots \\
    w_1^{p^{n-1}} & w_2^{p^{n-1}} & \cdots & w_n^{p^{n-1}}
    \end{vmatrix}
    \]
    viewed as an element of \(\mathbb F_p[w_1,\ldots,w_n]\).
    \end{definition}
    
    \medskip
    
    \begin{lemma}
    The Moore determinant is the product of all normalized linear forms over
    \(\mathbb F_p\), where the last nonzero coefficient is normalized to
    one; explicitly,
    \[
    \Delta_p(w_1,\ldots,w_n)
    =
    \prod_{i=1}^n
    \prod_{(a_1,\ldots,a_{i-1})\in\mathbb F_p^{i-1}}
    \left(a_1w_1+\cdots+a_{i-1}w_{i-1}+w_i\right).
    \]
    This identity holds in \(\mathbb F_p[w_1,\ldots,w_n]\).
    \end{lemma}
    
    \begin{proof}
    See \cite[Corollary~1.3.7]{gossBasicStructuresFunction1996}.
    \end{proof}
    
    \medskip
    
    \begin{definition}
    Let \(\Pi\in k[X,Y,Z]\) be the product of the linear forms indexed by
    the nonzero vectors of \(k^3\):
    
    \begin{equation}
    \Pi(X,Y,Z)
    :=
    \prod_{(a_0,a_1,a_2)\in k^3\setminus\{0\}}
    (a_0X+a_1Y+a_2Z).
    \end{equation}
    
    For each point of \(\mathbb P^2(k)\), choose the unique representative
    whose last nonzero coordinate is one, and define
    \begin{equation}
    \Theta(X,Y,Z)
    :=
    \prod_{[a_0:a_1:a_2]\in\mathbb P^2(k)}
    (a_0X+a_1Y+a_2Z),
    \end{equation}
    using these representatives.  With this normalization, the Moore-product
    formula gives the identity
    \begin{equation}
    \Theta(X,Y,Z)=\Delta_p(X,Y,Z).
    \label{eq:split-theta-equals-moore}
    \end{equation}
    \end{definition}
    
    \medskip
    
    \begin{remark}
    The chosen representatives, whose last nonzero coordinate is one, make
    \(\Theta\) a specified polynomial.  If arbitrary representatives are
    used, then \(\Theta\) changes by a scalar in \(k^\times\); hence
    \(\Theta^{p-1}\), \(D_+(\Theta)\), and the localization at \(\Theta\)
    do not change, although the labels of the components \(\mathbf{Y}_\ell\) may be
    permuted.
    
    With the stated normalization, group the factors of \(\Pi\) by their
    projective classes.  If \(L\) is the normalized linear form for one
    class, then
    \[
    \prod_{c\in k^\times}cL
    =
    \left(\prod_{c\in k^\times}c\right)L^{p-1}
    =-L^{p-1}
    \]
    in \(k[X,Y,Z]\).
    Since \(|\mathbb P^2(k)|=p^2+p+1\) is odd, it follows that
    \begin{equation}
    \Pi=-\Theta^{p-1}=(-1)^p\Theta^{p-1}
    \qquad\text{in }k[X,Y,Z].
    \end{equation}
    \end{remark}
    
    The preceding description gives the following affine presentation.
    
    \begin{lemma}
    Put
    \[
    t=\frac{x_1}{x_0},\qquad v=\frac{x_2}{x_0},
    \qquad \theta:=\Theta(1,t,v).
    \]
    Then \(\mathbf{X}=D_+(\Theta)\) lies in the chart \(x_0\ne0\), and
    \[
    \mathbf{X}\simeq\operatorname{Spec}k[t,v,\theta^{-1}].
    \]
    \end{lemma}
    
    \bigskip
    
    We next describe the affine cover \(\mathbf{Y}\).  By
    \cite[\S2.2]{deligneRepresentationsReductiveGroups1976}, it is defined
    in \(\mathbb A_k^3\), with coordinates \(X_0,X_1,X_2\), by
    \[
    \Delta_p(X_0,X_1,X_2)^{p-1}=1.
    \]
    Since \(X_0\neq0\) on this hypersurface, set
    \[
    (t,v,z)=(X_1/X_0,X_2/X_0,1/X_0),
    \qquad (X_0,X_1,X_2)=(1/z,t/z,v/z).
    \]
    Homogeneity gives \(\Delta_p(X_0,X_1,X_2)=X_0^N\theta(t,v)\), and hence
    \begin{equation} \label{split:Ydefeqn}
    z^{p^3-1}=\theta(t,v)^{p-1}.
    \end{equation}
    The right action of \(\tau\) multiplies the first column by \(\mu\),
    so \(\tau^*z=\mu^{-1}z\) and \(\sigma^*z=\mu^{1-p}z=\bar\zeta z\).
    Because \(\mathbb F_p^\times\) is the set of roots of
    \(W^{p-1}-1\),
    \[
    W^{p-1}-1=\prod_{\ell\in\mathbb F_p^\times}(W-\ell).
    \]
    
    Writing \(N=p^2+p+1\), the defining polynomial factors exactly as
    \begin{equation}
    z^{p^3-1}-\theta^{p-1}
    =\prod_{\ell\in\mathbb F_p^\times}(z^N-\ell\theta).
    \label{eq:split-kummer-factorization}
    \end{equation}
    
    \medskip
    
    \begin{proposition} \label{split:Ynotirrprop}
    For \( \mathbf{G}=\mathrm{GL}_3\), the group \( T^{wF}\) acting on
    the right is cyclic of order \(p^3-1\).  The affine scheme
    \[
    \mathbf{Y}
    :=
    \operatorname{Spec}
    \frac{k[t,v,\theta^{-1},z]}{(z^{p^3-1}-\theta^{p-1})}
    \]
    is a finite \'{e}tale right \( T^{wF}\)-torsor over \(\mathbf{X}\), and it is the
    disjoint union of the \(p-1\) integral schemes \(\mathbf{Y}_\ell\):
    \begin{equation}
    \mathbf{Y}=\bigsqcup_{\ell\in\mathbb F_p^\times}\mathbf{Y}_\ell,
    \end{equation}
    where
    \begin{equation}
    \mathbf{Y}_\ell
    :=
    \operatorname{Spec}
    \frac{k[t,v,\theta^{-1},z]}{(z^N-\ell\theta)}.
    \end{equation}
    \end{proposition}
    
    \begin{proof}
    The factors in \eqref{eq:split-kummer-factorization} are pairwise
    comaximal over \(k[t,v,\theta^{-1}]\): for \(\ell\ne\ell'\), their
    difference is \((\ell'-\ell)\theta\), which is a unit.  The Chinese
    remainder theorem therefore gives the displayed disjoint union.
    
    It remains to prove that each factor is integral.  Let \(K=k(t,v)\),
    fix \(\ell\in\mathbb F_p^\times\), and choose any nonconstant linear
    factor \(L\) of \(\theta\).  Its discrete valuation satisfies
    \[
    \nu_L(\ell\theta)=1.
    \]
    If \(\alpha^N=\ell\theta\) and \(\widetilde\nu\) is an extension of
    \(\nu_L\) to \(K(\alpha)\), with ramification index \(e\), then
    \[
    N\widetilde\nu(\alpha)
    =\widetilde\nu(\ell\theta)=e.
    \]
    Hence \(N\mid e\).  Since
    \[
    e\leq[K(\alpha):K]\leq N,
    \]
    all three quantities equal \(N\).  Thus \(z^N-\ell\theta\) is
    irreducible over \(K\), and Gauss's lemma shows that its quotient over
    the localization \(k[t,v,\theta^{-1}]\) is a domain. The same valuation
    argument over \(\bar{k}(t,v)\) proves that each \(\mathbf{Y}_\ell\) is geometrically integral. Finally,
    \(p\nmid N\) and \(z\) is invertible on this quotient, so its derivative
    \(Nz^{N-1}\) is invertible.  Each \(\mathbf{Y}_\ell\to \mathbf{X}\) is therefore finite
    \'{e}tale of degree \(N\), and their disjoint union has degree
    \((p-1)N=p^3-1\).
    \end{proof}
    
    \bigskip
    
    These presentations fit into the following commutative diagram.
    
    \[\begin{tikzcd}
    	{\mathbf{Y} = \mathrm{Spec} \Big( \frac{k[t,v,\theta^{-1},z]}{(z^{p^3-1}-\theta^{p-1})} \Big) = \bigsqcup_{\ell \in \mathbb F_p^{\times}}\mathbf{Y}_{\ell} } & {\overline{\mathbf{Y}} = \bigsqcup_{\ell \in \mathbb F_p^{\times}}  \overline{\mathbf{Y}}_{\ell} } \\
    	\\
    	{\mathbf{X} = \mathrm{Spec}(k[t,v,\theta^{-1}])} & {\overline{\mathbf{X}} = \operatorname{V}(f_0, f_1)= \mathrm{Bl}_Z \mathbb{P}^2 \subset \mathbb{P}^2 \times  \mathbb{P}^2 } \\
    	\\
    	{\genfrac{}{}{0pt}{}{\mathbb{P}^2}{[1;t;v]}} & {\mathbb{P}^2}
    	\arrow[hook, from=1-1, to=1-2]
    	\arrow["{\mathbb{Z}/(p^3-1)\mathbb{Z}}", from=1-1, to=3-1]
    	\arrow["{\mathbb{Z}/(p^3-1)\mathbb{Z}}", from=1-2, to=3-2]
    	\arrow[hook, from=3-1, to=3-2]
    	\arrow[hook', from=3-1, to=5-1]
    	\arrow[from=3-2, to=5-2]
    	\arrow["\cong", no head, from=5-1, to=5-2]
    \end{tikzcd}\]
    
    \begin{remark}
    The schemes \(\mathbf{Y}\) and \(\overline {\mathbf{Y}}\) are disjoint unions of \(p-1\)
    irreducible components.  The generator \(\tau\) of \( T^{wF}\) cyclically
    permutes them: it sends \(\mathbf{Y}_\ell\) to \(\mathbf{Y}_{\mu^{-N}\ell}\), and
    \(\mu^{-N}\) generates \(\mathbb F_p^\times\). For \(p>2\), its alternating
    Brauer trace on the cohomology of the whole disjoint union is zero.
    For \(p=2\), there is one component and \(\tau=\sigma\).  We therefore fix a component
    \(\overline {\mathbf{Y}}_\ell\) and compute
    \(L_{\mathrm{Br}}(\sigma,\overline {\mathbf{Y}}_\ell)\) for its stabilizer generator
    \(\sigma=\tau^{p-1}\).
    \end{remark}
    
 \medskip
    
    We use Langer's description
    \cite{langerBirationalGeometryCompactifications2019} of the smooth
    compactification \(\overline {\mathbf{X}}\) constructed in
    \cite[\S9.10]{deligneRepresentationsReductiveGroups1976}.  Let
    \(x_0,x_1,x_2\) and \(y_0,y_1,y_2\) be the homogeneous coordinates on
    \(\mathbb P(V)=\mathbb P^2\) and \(\mathbb P(V^*)=\mathbb P^2\),
    respectively.
    
    \medskip
    
    \begin{proposition}
    The surface \(\overline {\mathbf{X}}\) is isomorphic to the complete intersection
    \[
    \overline {\mathbf{X}}=\operatorname{V}(f_0,f_1)
    \subset \P{V}\times\P{V^*}
    =\mathbb P^2\times\mathbb P^2,
    \]
    where
    \[
    f_0=x_0y_0+x_1y_1+x_2y_2,
    \qquad
    f_1=x_0^py_0+x_1^py_1+x_2^py_2.
    \]
    \end{proposition}
    
    \begin{proof}
    See \cite[\S5, Proposition~5.2]{langerBirationalGeometryCompactifications2019}.
    \end{proof}
    
    \bigskip
    
    This is the wonderful compactification of the hyperplane arrangement in
    the sense of De Concini--Procesi
    \cite{deconciniWonderfulModelsSubspace1995}.  It is the blow-up of
    \(\mathbb P^2\) at all \(k\)-rational points
    \cite[\S5]{langerBirationalGeometryCompactifications2019}.
    
    \bigskip
    
    Recall that \(V=\mathbb F_p^3\).  Let \(\pi_1\) and \(\pi_2\) be the
    two projections from \(\overline {\mathbf{X}}\subset\mathbb P^2\times\mathbb P^2\).
    The morphism \(\pi_1\colon\overline {\mathbf{X}}\to\mathbb P^2\) is the blow-up at
    the finite set \(\mathbb P^2(\mathbb F_p)\).
    
    \bigskip
    
    We denote by \( \mathcal{L}(V) \) the set of all proper \( k \)-linear subspaces of \( \mathbb{P}(V) \). Every \( k \)-linear subspace \( L \in \mathcal{L}(V) \) is of the form \( \mathbb{P}(V/W) \) for some \( k \)-linear subspace \( W \subset V \). For \( L = \mathbb{P}(V/W) \subseteq \mathbb{P}(V) \) we set 
    \[
    L^{\perp} = \mathbb{P}(W^{*}) \subseteq \mathbb{P}(V^{*}). 
    \]

    If \( \eta_{L} \) denotes the generic point of \( L \), then we set
    \[
    E_{L} := \overline{\pi_{1}^{-1}(\eta_{L})} \subset \overline{\mathbf{X}}.
    \]
    
    \medskip
    
    We use the following description from
    \cite[\S5, Lemma~5.1]{langerBirationalGeometryCompactifications2019}.
    
    \medskip
    
    \begin{proposition}
        For any \( L \in \mathcal{L}(V) \) we have
    \[
    E_L = L \times_k L^{\perp} \quad \text{and} \quad \pi_{1}^{-1}(L) = \bigcup_{\substack{L' \in \mathcal{L}(V) \\ L' \subset L}} E_{L'}.
    \]
    In particular, each \(E_L\) is a smooth curve over \(k\), and
    \(\pi_1^{-1}(L)\) is of pure dimension \(1\).
    \end{proposition}
    Therefore,
    \(
    \pi_{1}^{-1}(\mathbf{X})
    =\overline {\mathbf{X}}\setminus
    \bigcup_{L\in\mathcal L(V)}E_L.
    \)
    
    \begin{proposition}[The split boundary and its incidence count]
    \label{prop:split-boundary-count}
    Put \(N=p^2+p+1\).  The boundary
    \(\overline {\mathbf{X}}\setminus \mathbf{X}\) has \(2N\) irreducible components:
    \[
    \overline {\mathbf{X}}\setminus \mathbf{X}
    =
    \bigcup_{P\in\mathbb P^2(\mathbb F_p)}E_P
    \ \cup\!
    \bigcup_{L\in(\mathbb P^2)^\vee(\mathbb F_p)}\widetilde L.
    \]
    Here \(E_P\) is the exceptional curve over the rational point \(P\),
    and \(\widetilde L\) is the strict transform of the rational line
    \(L\).  Their self-intersections are
    \[
    E_P^2=-1,
    \qquad
    \widetilde L^{2}=1-(p+1)=-p.
    \]
    Moreover, \(E_P\) and \(\widetilde L\) meet transversely exactly when
    \(P\in L\).  There are \(N(p+1)\) such incidences.
    \end{proposition}
    
    \begin{proof}
    There are \(N\) rational points and, by projective duality, \(N\)
    rational lines in \(\mathbb P^2_{\mathbb F_p}\).  Blowing up a point
    produces an exceptional curve of self-intersection \(-1\).  A rational
    line has self-intersection \(1\) in \(\mathbb P^2\) and contains exactly
    \(p+1\) rational points. Blowing up those points therefore changes its
    self-intersection to \(1-(p+1)=-p\).  The final assertion follows from the fact that every one of the
    \(N\) points lies on \((p+1)\)-many rational lines.
    \end{proof}
    
    \subsection{The Bonnaf\'e--Rouquier slice in the split case}
    \label{sec:br-slice-split}
    
    We now evaluate the slice mechanism of
    Section~\ref{sec:br-slice-singularities} for the split Frobenius.  On
    \(Y( T)\), the map \(F\) is multiplication by \(p\).  Put
    \[
    N=p^2+p+1.
    \]
    The solutions of \eqref{eq:br-cocharacter-equation} are
    \begin{equation}
    \lambda_1=(p+1,-1,-p),
    \qquad
    \lambda_2=(1,p,-p-1),
    \qquad
    m_1=m_2=N.
    \label{eq:br-split-lambdas}
    \end{equation}
    Indeed, since
    \[
    w(e_1)=e_2,
    \qquad
    w(e_2)=e_3,
    \qquad
    w(e_3)=e_1,
    \]
    we have
    \[
    \lambda_1-pw(\lambda_1)=N(e_1-e_2)
    \]
    and
    \[
    \lambda_2-pw(\lambda_2)=N(e_1-e_3)
    =N s_1(\alpha_2^\vee).
    \]
    Both \(\lambda_i\) are primitive, so the torsion-free conditions in
    \eqref{eq:br-cocharacter-equation} are satisfied.
    
    For \((z_1,z_2)\in(\mathbb G_m)^2\),
    \[
    \lambda_1(z_1)\lambda_2(z_2)
    =
    \diag\!\left(
    z_1^{p+1}z_2,
    z_1^{-1}z_2^p,
    z_1^{-p}z_2^{-p-1}
    \right).
    \]
    Substitution in \eqref{eq:br-deep-stabilizer} gives
    \begin{equation}
    K_{12}
    =
    \left\{
    (\zeta,\zeta^{-(p+1)}):\zeta^N=1
    \right\}
    \simeq\mu_N.
    \label{eq:br-split-stabilizer}
    \end{equation}
    Consequently, the action on the normal slice is
    \[
    \zeta(x,y)
    =
    (\zeta x,\zeta^{-(p+1)}y)
    =
    (\zeta x,\zeta^{p^2}y),
    \]
    because \(-(p+1)\equiv p^2\pmod N\).  Luna's theorem therefore gives
    the cyclic quotient singularity
    \begin{equation}
    \boxed{
    \frac1{p^2+p+1}(1,p^2)
    }.
    \label{eq:br-split-quotient-type}
    \end{equation}
    
    This agrees with the normalized toric cone computed below.  That cone is
    \[
    \Cone\bigl((1,0),(p+1,N)\bigr).
    \]
    With \(e_2=(1,0)\) and \(e_1=(1,1)\), its second primitive ray is
    \[
    (p+1,N)=Ne_1-p^2e_2.
    \]
    Thus its cyclic quotient parameters are also \((N,p^2)\).  Equivalently,
    since \(p^3\equiv1\pmod N\), the quotient may be written as
    \(\frac1N(1,p)\); this reverses the exceptional chain.  Finally,
    \[
    \frac{N}{p^2}
    =
    [[\underbrace{2,\ldots,2}_{p-1\text{ terms}},p+2]],
    \]
    which is the continued fraction appearing in the split resolution.

    \subsection[Local study of the normalization]{\texorpdfstring{Local study of the normalization \(\overline{\mathbf{Y}}_{\ell}\)}{Local study of the normalization}}\label{sec:split-normalization}

    The global surface \(\overline {\mathbf{Y}}_\ell(w)\) is the normalization of
    \(\overline {\mathbf{X}}(w)\) in the function field of a fixed component
    \(\mathbf{Y}_\ell(w)\).  We now compute its completed local rings over the
    zero-dimensional boundary strata.  These local calculations are used to determine
    the type of singularities that occur in \(\overline{\mathbf{Y}}_{\ell}(w)\).
    Consider the affine chart
    \[
    U_{10}=\{x_1\neq0,\ y_0\neq0\}
    \subset\mathbb P^2\times\mathbb P^2.
    \]
    The six charts \(U_{ij}=\{x_i\neq0,\ y_j\neq0\}\), with
    \(i\neq j\), cover \(\overline {\mathbf{X}}\). On \(U_{10}\), put
    \[
    t=\frac{x_0}{x_1},\qquad v=\frac{x_2}{x_1},\qquad
    r=\frac{y_1}{y_0},\qquad w=\frac{y_2}{y_0}.
    \]
    The equations of \(\overline {\mathbf{X}}\) become
    \[
    t+r+vw=0,
    \qquad
    t^p+r+v^pw=0.
    \]
    Eliminating \(r\) gives
    \begin{equation}\label{eq:split-U10-ring}
    R=k[U_{10}\cap\overline {\mathbf{X}}]
    =
    \frac{k[t,v,w]}{(t-t^p+vw-v^pw)}.
    \end{equation}
    
    % The homogeneous arrangement polynomial must be trivialized in the
    % \(x_1\neq0\) chart.  Accordingly, throughout the following local
    % calculation we write
    % \begin{equation}\label{eq:split-theta10}
    % \theta_{10}:=\Theta(t,1,v).
    % \end{equation}
    % This is the correct chart representative of \(\Theta\); the expression
    % \(\Theta(1,t,v)\) belongs to the different chart \(x_0\neq0\).

    % \bigskip

    The Kummer presentation of Proposition~\ref{split:Ynotirrprop} is
    written in the chart \(x_0\neq0\): there
    \(\theta=\Theta(1,x_1/x_0,x_2/x_0)\), and the Kummer generator is
    \(z=1/X_0\).  Near the boundary points studied below we work instead in
    the chart \(x_1\neq0\), so both the defining polynomial and the
    Kummer generator have to be changed.  With
    \(t=x_0/x_1\) and \(v=x_2/x_1\) as above, put
    \begin{equation}\label{eq:split-theta10}
    \theta_{10}:=\Theta(t,1,v).
    \end{equation}
    This is the chart representative of \(\Theta\) for \(x_1\neq0\); the
    expression \(\Theta(1,t,v)\) belongs to the chart \(x_0\neq0\).  Since
    \(\Theta\) is homogeneous of degree \(N=p^2+p+1\),
    \[
    \Theta(x_0,x_1,x_2)
    =x_0^{N}\,\Theta\!\Bigl(1,\frac{x_1}{x_0},\frac{x_2}{x_0}\Bigr)
    =x_1^{N}\,\Theta\!\Bigl(\frac{x_0}{x_1},1,\frac{x_2}{x_1}\Bigr),
    \]
    and hence, on the overlap of the two charts,
    \begin{equation}\label{eq:split-theta-chart-change}
    \theta=\Bigl(\frac{x_1}{x_0}\Bigr)^{N}\theta_{10}=t^{-N}\theta_{10}.
    \end{equation}
    Accordingly, we replace the Kummer generator \(z=1/X_0\) by
    \begin{equation}\label{eq:split-z10}
    z_{10}:=zt=\frac{1}{X_1}.
    \end{equation}
    The linear forms \(x_0\) and \(x_1\) are factors of \(\Theta\), so
    \(t\) is a unit on \(\mathbf{X}\) and \(z_{10}\) is a unit on \(\mathbf{Y}_\ell\).  By
    \eqref{eq:split-theta-chart-change}, the defining equation
    \(z^N=\ell\theta\) of \(\mathbf{Y}_\ell\) becomes
    \begin{equation}\label{eq:split-kummer-U10}
    z_{10}^{N}=\ell\,\theta_{10}.
    \end{equation}
    Equivalently, \eqref{eq:split-kummer-U10} is obtained by repeating the
    proof of Proposition~\ref{split:Ynotirrprop} with \(X_1\) in place of
    \(X_0\), using \(\Delta_p(X_0,X_1,X_2)=X_1^{N}\,
    \theta_{10}(X_0/X_1,X_2/X_1)\).  Since \(t\) is pulled back from
    \(\overline {\mathbf{X}}\), it is fixed by \(\sigma\), and therefore
    \[
    \sigma^*z_{10}=\bar\zeta\,z_{10},
    \]
    exactly as for \(z\).  All local calculations below use the generator
    \(z_{10}\) and the equation \eqref{eq:split-kummer-U10}. We proceed to construct the normalization of the local Kummer extension
    \[
    \frac{R[z_{10}]}{(z_{10}^{N}-\ell\theta_{10})},
    \qquad \ell\in\mathbb F_p^\times.
    \]

    % We proceed to construct the normalization of the local Kummer extension
    % \[
    % \frac{R[z]}{(z^{p^2+p+1}-\ell\theta_{10})},
    % \qquad \ell\in\mathbb F_p^\times.
    % \]
    
    % \bigskip

    \medskip

    \begin{lemma}\label{lem:split-local-theta-factorization}
    Let
    \[
    P=([0:1:0],[1:0:0])\in\overline {\mathbf{X}}(k).
    \]
    In the chart \(U_{10}\), the point \(P\) is the origin
    \(t=v=r=w=0\).  The elements
    \[
    u:=w,
    \qquad
    v:=\frac{x_2}{x_1}
    \]
    form a regular system of parameters at \(P\), and there is a unit
    \(\varepsilon\in\widehat{\mathcal O}_{\overline {\mathbf{X}},P}^{\times}\) such
    that
    \begin{equation}\label{eq:split-local-theta-factorization}
    \theta_{10}=\varepsilon v^{p+1}u.
    \end{equation}
    More explicitly, if
    \[
    E(t,v):=\prod_{a,b\in\mathbb F_p}(1+at+bv)
    \]
    and
    \[
    \beta:=-\frac{1-v^{p-1}}{1-t^{p-1}},
    \]
    then \(t=vw\beta\) in the completed local ring and one may take
    \begin{equation}\label{eq:split-explicit-local-unit}
    \varepsilon
    =E(t,v)(1-v^{p-1})(1+w^{p-1}\beta^p).
    \end{equation}
    The same factorization, up to multiplication of \(\varepsilon\) by a
    nonzero scalar, holds at every \(k\)-rational point of \(\overline {\mathbf{X}}\).
    \end{lemma}
    
    \begin{proof}
    Write
    \[
    f=t-t^p+vw-v^pw.
    \]
    At the origin, \(\partial f/\partial t=1\).  Thus the formal implicit
    function theorem applied to \eqref{eq:split-U10-ring} gives
    \[
    \widehat{\mathcal O}_{\overline {\mathbf{X}},P}\simeq k[[v,w]],
    \]
    so \(u=w\) and \(v=x_2/x_1\) are indeed a regular system of
    parameters.
    
    We next factor the defining polynomial in the correct chart.  From
    the definition
    \[
    \theta_{10}
    =
    \prod_{[a_0:a_1:a_2]\in\mathbb P^2(\mathbb F_p)}
    (a_0t+a_1+a_2v),
    \]
    the factors that vanish at \(P\) are precisely those with \(a_1=0\).
    They are
    \[
    t
    \quad\text{and}\quad
    at+v\qquad(a\in\mathbb F_p).
    \]
    Their product is \(-v(t^p-tv^{p-1})\). The factors with \(a_1\neq0\)
    have product \(-E(t,v)\): renormalizing them to \(a_1=1\) contributes
    \((\prod_{c\in\mathbb F_p^\times}c)^p=-1\).
    The polynomial \(E(t,v)\) is a unit because \(E(0,0)=1\).  Using
    \[
    \prod_{c\in\mathbb F_p}(T+cQ)=T^p-TQ^{p-1},
    \]
    we obtain the exact identity
    \begin{equation}\label{eq:split-theta-first-factorization}
    \theta_{10}
    =E(t,v)v(t^p-tv^{p-1}).
    \end{equation}
    
    The equation \(f=0\) gives
    \[
    t^p-t=vw(1-v^{p-1}),
    \]
    and hence
    \begin{equation}\label{eq:split-theta-middle-factorization}
    t^p-tv^{p-1}
    =(1-v^{p-1})(t+vw).
    \end{equation}
    It also gives
    \[
    t(1-t^{p-1})=-vw(1-v^{p-1}).
    \]
    Both parenthesized factors are units at \(P\), so
    \[
    t=vw\beta,
    \qquad
    \beta=-\frac{1-v^{p-1}}{1-t^{p-1}}\in
    \widehat{\mathcal O}_{\overline {\mathbf{X}},P}^{\times}.
    \]
    Finally, using \(t+vw=t^p+v^pw\), we find
    \begin{equation}\label{eq:split-t-plus-vw}
    t+vw
    =v^pw(1+w^{p-1}\beta^p).
    \end{equation}
    Substituting \eqref{eq:split-theta-middle-factorization} and
    \eqref{eq:split-t-plus-vw} into
    \eqref{eq:split-theta-first-factorization} yields
    \[
    \theta_{10}
    =
    E(t,v)(1-v^{p-1})(1+w^{p-1}\beta^p)v^{p+1}w
    =\varepsilon v^{p+1}u.
    \]
    Each of the three factors defining \(\varepsilon\) has nonzero constant
    term (in fact, constant term \(1\)), so \(\varepsilon\) is a unit.  This
    proves the claimed factorization at \(P\).
    
    The group \(\operatorname{PGL}_3(\mathbb F_p)\) acts transitively on the
    \(\mathbb F_p\)-rational flags parametrized by \(\overline {\mathbf{X}}(\mathbb
    F_p)\).  It permutes the linear factors of \(\Theta\), changing their
    product only by a nonzero scalar.  Transporting the calculation above
    therefore proves the final assertion.
    \end{proof}
    
    For the later toric notation, set \(s:=u=w\) and retain the coordinate
    \(v=x_2/x_1\).  If
    \(N=p^2+p+1\), then \(p\nmid N\) and
    \(\varepsilon\equiv1\pmod{(u,v)}\).  Hensel's lemma therefore gives a
    unique \(\eta\in1+(u,v)\) with \(\eta^N=\varepsilon\).  Replacing the
    Kummer coordinate by \(z/\eta\) removes the unit.  The remaining
    nonzero component scalar can be normalized to one up to an isomorphism
    of complete local \(k\)-algebras by rescaling a boundary parameter. Thus, we get the following equation:
    \begin{equation}\label{eq:split-local-toric-equation}
    z^N=v^{p+1}s.
    \end{equation}
    
    \subsubsection{Local calculations}
    \label{sec:split-local-to-global}
    
    The wonderful compactification \(\overline {\mathbf{X}}(w)\) is the blow-up of
    \(\mathbb P^2\) at all its \(\mathbb F_p\)-rational points
    \cite{langerBirationalGeometryCompactifications2019}.  Its boundary has
    one exceptional component \(E_P\) for every rational point \(P\) and one
    component \(\widetilde L\), the strict transform, for every rational line
    \(L\).  An \(\mathbb F_p\)-rational point of \(\overline {\mathbf{X}}(w)\) is a
    zero-dimensional boundary stratum corresponding to an incidence
    \(P\in L\): the exceptional curve over \(P\) parametrizes tangent
    directions at \(P\), and its \(\mathbb F_p\)-points are precisely the
    \(p+1\) rational lines through \(P\).  Thus, in an \'{e}tale
    neighborhood of such a point, the boundary is a pair of transverse
    branches
    \[
    E_P=(v=0),
    \qquad
    \widetilde L=(s=0),
    \]
    and the completed cover has the Kummer equation
    \[
    z^N=v^{p+1}s.
    \]
    % A geometric point on only one boundary component has only one local
    % boundary branch, while a point of the open Deligne--Lusztig variety has
    % none.
    
    The following calculation explains why the normal
    eigenvalue is constant within each of the two boundary families.
    
    \begin{lemma}[Normal character along a ramified divisor]
    \label{lem:split-kummer-normal-character}
    Let \(K\subset L\) be the degree-\(N\) extension of function fields
    induced by a chosen component of the cover.  Let \(D\) be a prime
    divisor on the base, let \(t\) be a uniformizer at its generic point,
    and suppose that in \(L\)
    \[
    z^N=\varepsilon t^a,
    \qquad
    \sigma(z)=\bar\zeta z,
    \]
    where \(\varepsilon\) is a unit, \(\sigma\) fixes \(K\),
    \(\bar\zeta\) has order \(N\), and \(\gcd(a,N)=1\).  Then there is a
    unique prime divisor \(D'\) above \(D\), and it is totally ramified.  If
    \(ab\equiv1\pmod N\), then \(\sigma\) acts on a local parameter of
    \(D'\), and hence on \(I_{D'}/I_{D'}^2\), by
    \(\bar\zeta^{\,b}\).  In the convention of
    Theorem~\ref{thm:lefschetz-rr}, the normal eigenvalue is therefore
    \[
    \nu_{D'}=\bar\zeta^{\,-b}.
    \]
    The same eigenvalue occurs along the strict transform of \(D'\) under
    any equivariant resolution that is an isomorphism at its generic point.
    \end{lemma}
    
    \begin{proof}
    Let \(v'\) extend the valuation of \(D\), and let \(e\) be its
    ramification index.  From \(Nv'(z)=ae\) and \(\gcd(a,N)=1\), one
    obtains \(N\mid e\).  Since \(e\leq[L:K]=N\), we have \(e=N\); the
    fundamental inequality then implies that \(D'\) is unique and has
    residue degree one.
    
    Choose \(c\in\mathbb Z\) such that \(ab+Nc=1\).  Then
    \(\pi=z^bt^c\) has valuation one, so it is a local parameter at
    \(D'\), and
    \[
    \sigma(\pi)=\bar\zeta^{\,b}\pi.
    \]
    Reducing modulo \(\pi^2\) gives the claimed character.  An equivariant
    resolution that is an isomorphism at the generic point preserves this
    discrete valuation ring and its character.
    \end{proof}

    In the present cover, every rational line occurs with multiplicity one
    in the boundary divisor, whereas the exceptional divisor over a
    rational point occurs with multiplicity \(p+1\), because exactly \(p+1\)
    rational lines pass through that point.  Consequently, at the generic
    points of the two boundary types one has
    \[
    a_{\widetilde L}=1,
    \qquad
    a_{E_P}=p+1.
    \]
    Lemma~\ref{lem:split-kummer-normal-character} therefore gives
    \[
    \nu_{\widetilde L'}=\bar\zeta^{-1},
    \qquad
    \nu_{\widetilde E'_P}=\bar\zeta^{p},
    \]
    because
    \[
    (-p)(p+1)+N=1.
    \]
    This proves the constancy across all curves in either family.  It also shows that the same
    Kummer multiplicities occur at every incidence.  Along an individual
    fixed curve, the eigenvalue is constant for the intrinsic reason that
    \(\sigma\) acts over the identity on its normal line bundle, hence by an
    invertible regular function and therefore is constant on a connected projective integral curve.
    
    % The self-intersection of a complete curve, in contrast, is not
    % determined by an \'{e}tale neighborhood of one point. 
    % On
    % \(\overline {\mathbf{X}}(w)\), one also has
    % \[
    % E_P^2=-1,
    % \qquad
    % \widetilde L^2=1-(p+1)=-p.
    % \]
    % The self-intersections of their lifts in \(\overline{\mathbf{Y}}(w)\) are obtained later from the
    % global pullback divisors and the projection formula.  Since every point
    % lies on \(p+1\) rational lines, every line contains \(p+1\) rational
    % points, and the same toric chain occurs at each incidence, those global
    % calculations are independent of \(P\) and \(L\) within the corresponding
    % family.
    
    \bigskip

    \subsection{Toric normalization, resolution, and fixed locus}
    \label{sec:split-toric-resolution}
    
    \subsubsection{Interpretation of the local model as a toric surface}
    
    We use the semigroup, normalization, and resolution conventions collected
    in Section~\ref{sec:toric-surface-preliminaries}.  Put
    \(N=p^2+p+1\).
    
    Introduce the affine algebraic surface
    \[
    Y_{\mathscr A}
    :=
    \operatorname{Spec}\frac{k[v,s,z]}{(z^N-v^{p+1}s)}.
    \]
    Its completion at the origin is the scalar-one algebra \(A'_P\) of the
    preceding subsection.  Let
    \[
    \nu_\sigma\colon
    U_\sigma:=Y_{\mathscr A}^{\mathrm{norm}}
    \longrightarrow Y_{\mathscr A}
    \]
    be its normalization, and let \(o_\sigma\in U_\sigma\) be the point
    above the origin.  Compatibility of normalization with completion give
    %and
    % \eqref{eq:split-completed-normalization} 
    \begin{equation}
    \widehat{\mathcal O}_{\overline {\mathbf{Y}},\overline P}
    \xrightarrow{\ \sim\ }
    \widehat{\mathcal O}_{U_\sigma,o_\sigma}.
    \label{eq:split-normal-toric-completion}
    \end{equation}
    
    \bigskip
    We first observe that we have an isomorphism of \(k\)-algebras
    \begin{equation} \label{split:iso of k-alg}
        \frac{k[v,s,z]}{(z^{N} -  v^{p+1} s)  }  \simeq k[X^pY, XY, XY^{p^2}]
    \end{equation}
    given by \( z \mapsto XY, v \mapsto X^pY  \) and \( s \mapsto XY^{p^2} \).
    The map is surjective. Both rings are two-dimensional domains, so its
    kernel is a prime ideal of height zero and hence is zero.
    
    \bigskip
    
    Let \(M\simeq\mathbb Z^2\) be the character lattice of the dense torus,
    put \(N_{\mathrm{lat}}=\operatorname{Hom}(M,\mathbb Z)\), and write
    \(\chi^{(a,b)}=t_1^at_2^b\).  Taking \(t_1=XY=z\) and
    \(t_2=X^pY=v\), the three monomials become
    \begin{equation}
    v=\chi^{(0,1)}=t_2,
    \qquad
    z=\chi^{(1,0)}=t_1,
    \qquad
    s=\chi^{(N,-p-1)}=t_1^Nt_2^{-p-1}.
    \label{split:changeofcordeqn}
    \end{equation}
    Consequently,
    \[
    \mathscr A=\{(0,1),(N,-p-1),(1,0)\}\subset M\simeq\mathbb Z^2,
    \]
    and the corresponding affine toric variety is
    \[
    Y_{\mathscr A}=\operatorname{Spec}k[S],
    \]
    where \(S\) is the semigroup generated by \(\mathscr A\).  This is the
    algebraic local model whose completion gives \(A'_P\); it is not being
    identified with a global surface over \(\overline {\mathbf{X}}(w)\).
    
    \bigskip
    
    By \eqref{eq:toric-semigroup-normalization}, the toric normalization
    introduced above is
    \[
    U_\sigma=\operatorname{Spec}k[S_\sigma],
    \]
    where \(S_\sigma\) is the saturation of \(S\) and \(\sigma\) is the
    dual cone of the cone generated by \(\mathscr A\).  Its completed local
    ring at the torus-fixed point \(o_\sigma\) is identified with the
    completed local ring of \(\overline {\mathbf{Y}}\) at \(\overline P\) by
    \eqref{eq:split-normal-toric-completion}.
    
    \bigskip
    
    \begin{lemma}
        For the cone \( \sigma = \mathrm{Cone}( (1,0) , (p+1, N))  \), the Hilbert basis \( \mathscr{H} \) is \[ \mathscr{H} = \{ u_0 := (1,0)  , \  u_1:= (1,1) , \  u_2 :=  (1,2) , \ \dots \ , \  u_{p}:= (1,p) , \ u_{p+1} := (p+1, N) \} \] 
    \end{lemma}
    
    \begin{proof}
    The consecutive determinants are
    \[
    \det(u_i,u_{i+1})=1\quad(0\leq i<p),
    \qquad
    \det(u_p,u_{p+1})=N-p(p+1)=1.
    \]
    Thus every cone spanned by two consecutive vectors is smooth, and the
    union of these cones is \(\sigma\).  Every lattice point of \(\sigma\)
    therefore lies in one of these smooth cones and is a nonnegative
    integral combination of its two generators.  The vectors
    \(u_0,\ldots,u_p\), whose first coordinate is \(1\), are indecomposable
    in the semigroup \(\sigma\cap N_{\mathrm{lat}}\); the final vector
    \(u_{p+1}\) is the
    primitive generator of the other extremal ray.  Hence the displayed set
    is precisely the Hilbert basis.
    \end{proof}
    
    % \bigskip
    
    Take
    \(
    e_1=(1,1), e_2=(1,0).
    \)
    Then
    \(
    (p+1,N)=Ne_1-p^2e_2,
    \)
    so the cone is in Cox normal form
    \(    \sigma=\operatorname{Cone}(e_2,Ne_1-p^2e_2)
    \)
    with numerical parameters \(d=N\) and \(a=p^2\).  As recalled in
    Section~\ref{sec:toric-surface-preliminaries}, the Hirzebruch--Jung
    expansion of \(d/a\) determines the minimal regular subdivision and
    hence the minimal toric resolution \(\phi:X_\Sigma\to U_\sigma\).
    
    \begin{proposition}
    The Hirzebruch--Jung continued fraction is
    \[
    \frac{N}{p^2}
    =[[\underbrace{2,2,\dots,2}_{p-1\ \text{terms}},p+2]].
    \]
    \end{proposition}
    
    \begin{proof}
    The negative Euclidean algorithm begins with
    \[
    b_1=\Big\lceil \frac{p^2+p+1}{p^2}\Big\rceil
    =2,
    \qquad
    r_1=2p^2-N=p^2-p-1.
    \]
    Put \(r_0=p^2\).  Inductively,
    \[
    r_j=p^2-j(p+1)\quad(0\leq j\leq p-1).
    \]
    Indeed, for \(1\leq j\leq p-2\),
    \[
    \frac{r_{j-1}}{r_j}
    =1+\frac{p+1}{r_j}
    \in(1,2),
    \]
    so the next quotient is two and
    \[
    r_{j+1}=2r_j-r_{j-1}
    =p^2-(j+1)(p+1),
    \]
    as required.  Finally,
    \[
    r_{p-1}=p^2-(p-1)(p+1)
    =1,
    \qquad
    r_{p-2}=p+2.
    \]
    Thus the first \(p-1\) quotients are two and the last quotient is
    \(p+2\).
    \end{proof}
    
    \subsubsection[The minimal resolution]{The minimal resolution \(\phi\)}
    \label{split:minres}
    
    For the Hilbert-basis vectors \(u_0,\ldots,u_{p+1}\) above, put
    \[
    \sigma_i:=\operatorname{Cone}(u_{i-1},u_i)
    \qquad(1\leq i\leq p+1),
    \]
    and let \(\Sigma\) consist of these cones and their faces.  The
    determinant calculation in the preceding lemma shows that each
    \(\sigma_i\) is regular.  Since \(|\Sigma|=\sigma\), the subdivision
    defines a toric resolution
    \[
    \phi\colon X_\Sigma\longrightarrow U_\sigma.
    \]
    This morphism is projective; see
    \cite[\S10.2]{coxToricVarieties2011a}.  The exceptional curves
    \(D_i\), corresponding to the rays generated by \(u_i\) for
    \(1\leq i\leq p\), form a chain of smooth rational curves.  The
    relations
    \[
    u_{i-1}+u_{i+1}=2u_i\quad(1\leq i<p),
    \qquad
    u_{p-1}+u_{p+1}=(p+2)u_p
    \]
    give
    \[
    D_i^2=-2\quad(1\leq i<p),
    \qquad
    D_p^2=-(p+2).
    \]
    Thus \(\phi\) is the minimal resolution associated with the
    Hirzebruch--Jung expansion above, as recalled in
    Section~\ref{sec:toric-surface-preliminaries}.
    The surface \(X_\Sigma\) is covered by the \(p+1\) affine charts
    \(U_{\sigma_i}\), with
    \[
    U_{\sigma_i}\cap U_{\sigma_j}=U_{\sigma_i\cap\sigma_j};
    \]
    see \cite[Ch.~3, (3.1.3)]{coxToricVarieties2011a}.
    
    \medskip
    
    Let
    \(
    \rho_\ell\colon
    \widetilde{\overline {\mathbf{Y}}}_\ell(w)
    \longrightarrow\overline {\mathbf{Y}}_\ell(w)
    \)
    be the minimal resolution of the normal projective surface
    \(\overline {\mathbf{Y}}_\ell(w)\).  Its existence follows from resolution theory
    for normal surfaces.  Write
    \[
    \overline {\mathbf{Y}}:=\overline {\mathbf{Y}}_\ell(w),
    \qquad
    \widetilde{\overline {\mathbf{Y}}}:=\widetilde{\overline {\mathbf{Y}}}_\ell(w),
    \]
    and put
    \[
    E_{\overline P}:=\rho_\ell^{-1}(\overline P),
    \qquad
    E_\Sigma:=\phi^{-1}(o_\sigma).
    \]
    
    We compare the two resolutions on an \'{e}tale neighborhood.  Before
    the coordinate changes used in
    \eqref{eq:split-local-toric-equation}, the local Kummer equation is
    \[
    z_0^N=\ell\varepsilon v^{p+1}u,
    \qquad N=p^2+p+1.
    \]
    The formula in \eqref{eq:split-explicit-local-unit} shows that
    \(\varepsilon\) belongs to
    \(\mathcal O_{\overline {\mathbf{X}},P}^{\times}\): all its denominators are
    units at \(P\).  Thus this equation holds on an affine open
    neighborhood of \(P\).  Since \(p\nmid N\), adjoining a root
    \(\eta\) of \(\eta^N=\varepsilon\) gives a finite \'{e}tale cover
    of this neighborhood.  Choose a point \(P_V\) above \(P\) and an
    affine open neighborhood \(V\) of \(P_V\) in this cover.  On \(V\),
    put
    \[
    z:=z_0/\eta,
    \qquad s:=\ell u.
    \]
    The Kummer equation becomes
    \[
    z^N=v^{p+1}s.
    \]
    Since \(v,u\) are regular parameters at \(P\), after shrinking
    \(V\) the morphism
    \[
    \gamma\colon V\longrightarrow\mathbb A_k^2,
    \qquad \gamma=(v,s),
    \]
    is \'{e}tale.  We may also assume that \(P_V\) is the only point
    of \(V\) mapping to the origin.
    
    Normalization commutes with \'{e}tale base change; see the Stacks
    Project, Tag \href{https://stacks.math.columbia.edu/tag/03GV}{03GV}.
    Consequently, the preceding change of coordinates induces an
    isomorphism
    \[
    W:=\overline {\mathbf{Y}}\times_{\overline {\mathbf{X}}}V
    \xrightarrow{\ \sim\ }
    U_\sigma\times_{\mathbb A_k^2}V.
    \]
    In particular, the projections
    \[
    a\colon W\longrightarrow\overline {\mathbf{Y}},
    \qquad
    b\colon W\longrightarrow U_\sigma
    \]
    are \'{e}tale.  The scheme \(W\) is affine, since it is finite over
    \(V\).  Let \(P_W\) be its unique point above \(P_V\).  Then
    \[
    a(P_W)=\overline P,
    \qquad
    b(P_W)=o_\sigma,
    \qquad
    \kappa(P_W)=k.
    \]
    Since \(b\) is \'{e}tale and \(b^{-1}(o_\sigma)\) has only the
    point \(P_W\), the surface \(W\) is regular away from \(P_W\).
    
    The two base changes
    \[
    \widetilde{\overline {\mathbf{Y}}}\times_{\overline {\mathbf{Y}}}W\longrightarrow W,
    \qquad
    X_\Sigma\times_{U_\sigma}W\longrightarrow W
    \]
    are resolutions of \(W\).  Over \(P_W\), \'{e}tale base change
    preserves the exceptional curves and their self-intersections.
    Hence both resolutions are minimal.  Uniqueness of the minimal
    resolution gives an isomorphism over \(W\),
    \begin{equation}
    \widetilde{\overline {\mathbf{Y}}}\times_{\overline {\mathbf{Y}}}W
    \xrightarrow{\ \sim\ }
    X_\Sigma\times_{U_\sigma}W.
    \label{eq:split-etale-resolution-isomorphism}
    \end{equation}
    Put \(\widetilde W:=
    \widetilde{\overline {\mathbf{Y}}}\times_{\overline {\mathbf{Y}}}W\).  We obtain the
    commutative diagram
    \[
    \begin{tikzcd}[column sep=large,row sep=large]
    \widetilde{\overline {\mathbf{Y}}}
        \arrow[d,"\rho_\ell"']
    & \widetilde W
        \arrow[l]
        \arrow[r]
        \arrow[d]
    & X_\Sigma
        \arrow[d,"\phi"]
    \\
    \overline {\mathbf{Y}}
        \arrow[d,"\overline\pi_\ell"']
    & W
        \arrow[l,"a"']
        \arrow[r,"b"]
        \arrow[d]
    & U_\sigma
        \arrow[d]
    \\
    \overline {\mathbf{X}}
    & V
        \arrow[l]
        \arrow[r,"\gamma"']
    & \mathbb A_k^2,
    \end{tikzcd}
    \]
    in which all four squares are cartesian.
    
    The fibre of \(\widetilde W\to W\) over \(P_W\) identifies with
    both \(E_{\overline P}\) and \(E_\Sigma\).  Thus the exceptional
    curves above \(\overline P\) have the same intersection matrix as
    the curves \(D_1,\ldots,D_p\).  Moreover, the affine open subsets
    \[
    U_{\sigma_i}\times_{U_\sigma}W,
    \qquad 1\leq i\leq p+1,
    \]
    cover \(\widetilde W\), with transition maps obtained by base
    change from those of \(X_\Sigma\).
    
    \medskip
    
    Let \(G_\ell=\langle\sigma\rangle\) be the component stabilizer.
    It acts trivially on \(\overline {\mathbf{X}}\), and on the Kummer coordinate
    it acts by
    \[
    \sigma(z_0)=\bar\zeta z_0,
    \]
    where \(\bar\zeta\) is a primitive \(N\)-th root of unity.
    Give \(V\) the trivial \(G_\ell\)-action.  Since \(\eta\) is a
    function on \(V\), the change of coordinates above gives
    \[
    \sigma(z)=\bar\zeta z,
    \qquad
    \sigma(v)=v,
    \qquad
    \sigma(s)=s.
    \]
    The morphisms \(a\) and \(b\) are therefore
    \(G_\ell\)-equivariant.  Uniqueness of the minimal resolution
    lifts the actions on \(\overline {\mathbf{Y}}\) and \(U_\sigma\) to their
    resolutions and makes
    \eqref{eq:split-etale-resolution-isomorphism} equivariant.  The
    lift on \(X_\Sigma\) is the restriction of its torus action.
    Hence the toric charts compute the stabilizer characters on the
    tangent and normal spaces and the fixed locus near
    \(E_{\overline P}\).  The lower squares in the diagram also
    identify the two boundary branches and preserve their
    ramification indices.
    
    The number and orbits of the global boundary curves, their genera,
    the self-intersections of their strict transforms, and the total
    number of incidences require the global geometry of \(\overline {\mathbf{X}}\).
    Their intersection numbers are computed from pullback divisors and
    the projection formula.  These numbers, together with the local
    characters computed below, enter the Lefschetz--Riemann--Roch formula.

    \subsubsection{Affine charts of the resolution}
    
    We now write down these affine charts, compute the action of the component
    stabilizer \(G_\ell=\langle\sigma\rangle\simeq\mathbb Z/N\mathbb Z\),
    and analyze its fixed locus.
    
    \medskip
    
    The dual monoids must be computed cone by cone.  For the first cone the
    primitive dual generators are \((0,1)\) and \((1,-1)\).  For
    \(2\leq i\leq p\), the dual generators of
    \(\sigma_i=\operatorname{Cone}(u_{i-1},u_i)\) are
    \[
    (i,-1),
    \qquad
    (1-i,1),
    \]
    and for the last cone they are \((-p,1)\) and \((N,-p-1)\).  Each
    displayed pair is the basis of \(M\) dual to the two primitive ray
    generators of the corresponding unimodular cone.  Consequently, the
    charts for the local toric resolution \(X_\Sigma\) are as follows:
    
    \begin{equation*}
    \begin{aligned}
         U_{\sigma_1} & = \mathrm{Spec} \big( k[t_2, t_1 t_2^{-1}] \big) & \rightsquigarrow \text{ \ affine coordinates \ } a_{1,1} , a_{1,2} \\
        U_{\sigma_2} &= \mathrm{Spec} \big( k[t_1^2 t_2^{-1}, t_1^{-1}t_2] \big)  & \rightsquigarrow \text{ \ affine coordinates \ } a_{2,1} , a_{2,2} \\
         U_{\sigma_3} &= \mathrm{Spec} \big( k[t_1^3 t_2^{-1}, t_1^{-2}t_2] \big)   & \rightsquigarrow \text{ \ affine coordinates \ } a_{3,1} , a_{3,2} \\
      &    \vdots  & \vdots \\
         U_{\sigma_i} &= \mathrm{Spec} \big( k[t_1^i t_2^{-1}, t_1^{1-i}t_2] \big)  & \rightsquigarrow \text{ \ affine coordinates \ } a_{i,1} , a_{i,2} \\
        & \vdots & \vdots \\
         U_{\sigma_{p+1}} &= \mathrm{Spec} \big( k[t_1^{-p}t_2, t_1^N t_2^{-p-1}] \big)  & \rightsquigarrow \text{ \ affine coordinates \ } a_{p+1,1} , \ a_{p+1,2}
    \end{aligned}
    \end{equation*}
        
    \bigskip
    
    \subsubsection{Action of the component stabilizer and fixed locus}
    
    The subgroup \(\langle\sigma\rangle\leq T^{wF}\) that stabilizes a
    component is cyclic of order \(N\).  On the space \(
    \operatorname{Spec}k[v,s,z]/(z^N-v^{p+1}s)
    \), \(\sigma\) acts by
    \[
    \sigma(z)=\bar\zeta z,
    \qquad
    \sigma(v)=v,
    \qquad
    \sigma(s)=s,
    \]
    where \(\bar\zeta=\mu^{1-p}\) is a primitive \(N\)-th root.  Under
    \eqref{split:iso of k-alg}, this becomes
    \[
    \sigma(t_1)=\bar\zeta t_1,
    \qquad
    \sigma(t_2)=t_2.
    \]
    The action extends to \(\overline {\mathbf{Y}}_\ell(w)\) and lifts to its minimal
    resolution \(\widetilde{\overline {\mathbf{Y}}}_\ell(w)\).  On the local toric
    charts it has the following
    weights:
    \begin{align*}
    U_{\sigma_1}\colon\quad&
    (a_{1,1},a_{1,2})\longmapsto
    (a_{1,1},\bar\zeta a_{1,2}),\\
    U_{\sigma_i}\colon\quad&
    (a_{i,1},a_{i,2})\longmapsto
    (\bar\zeta^{\,i}a_{i,1},\bar\zeta^{\,1-i}a_{i,2})
    \quad(2\leq i\leq p),\\
    U_{\sigma_{p+1}}\colon\quad&
    (a_{p+1,1},a_{p+1,2})\longmapsto
    (\bar\zeta^{-p}a_{p+1,1},a_{p+1,2}).
    \end{align*}
    % The endpoint charts also identify the two global boundary types.  On
    % \(U_{\sigma_1}\),
    % \[
    % a_{1,1}=v,
    % \qquad
    % a_{1,2}=z/v,
    % \qquad
    % s=a_{1,1}^{p^2}a_{1,2}^N,
    % \]
    % so \(a_{1,2}=0\) is the strict lift of the line-type branch
    % \(s=0\), and its normal character is \(\bar\zeta^{-1}\).  On the last
    % chart,
    % \[
    % a_{p+1,1}=v/z^p,
    % \qquad
    % a_{p+1,2}=s,
    % \qquad
    % v=a_{p+1,2}^{p}a_{p+1,1}^{N},
    % \]
    % so \(a_{p+1,1}=0\) is the strict lift of the exceptional-type branch
    % \(v=0\), and its normal character is \(\bar\zeta^{p}\).

 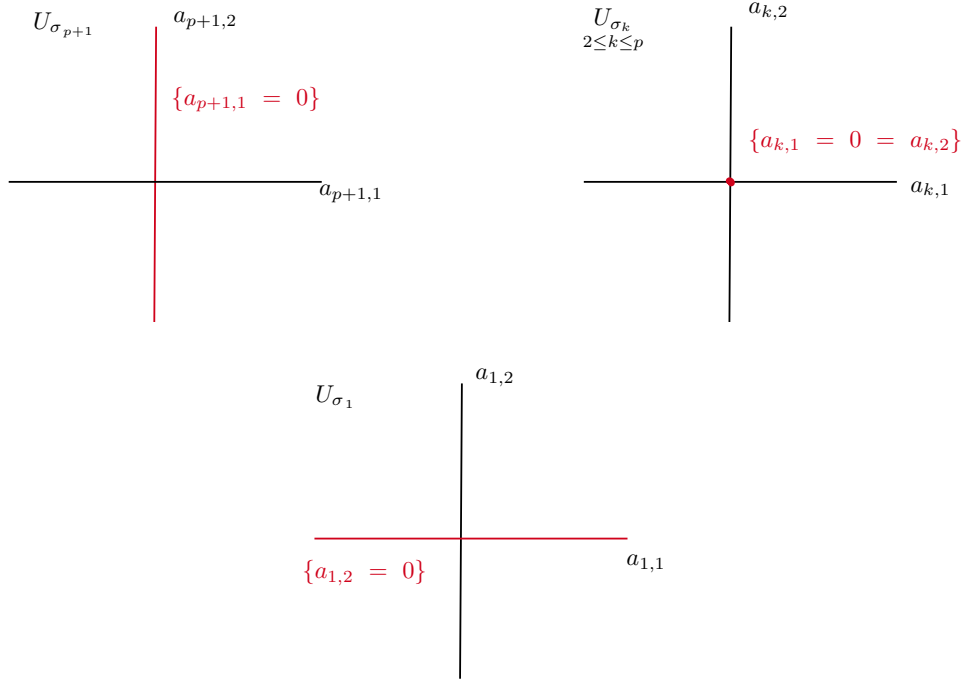
\begin{figure}[H]
        \centering
        %\includegraphics[width=0.5\linewidth]{}
      
    %\tikzset{every picture/.style={line width=0.75pt}} %set default line width to 0.75pt        
    
    \resizebox{\textwidth}{!}{%
    \begin{tikzpicture}[x=0.75pt,y=0.75pt,yscale=-1,xscale=1]
    %uncomment if require: \path (0,458); %set diagram left start at 0, and has height of 458
    
    %Straight Lines [id:da491963583066367] 
    \draw [color={rgb, 255:red, 208; green, 2; blue, 27 }  ,draw opacity=1 ]   (145,54) -- (144,216.25) ;
    %Straight Lines [id:da2731506855667838] 
    \draw    (64,139.25) -- (236,139.25) ;
    %Straight Lines [id:da38136633525137076] 
    \draw    (461,54) -- (460,216.25) ;
    %Straight Lines [id:da8111860080389831] 
    \draw    (380,139.25) -- (552,139.25) ;
    
    %Straight Lines [id:da3002014879628081] 
    \draw    (313,250) -- (312,412.25) ;
    %Straight Lines [id:da04540432304872044] 
    \draw [color={rgb, 255:red, 208; green, 2; blue, 27 }  ,draw opacity=1 ]   (232,335.25) -- (404,335.25) ;
    %Shape: Free Drawing [id:dp8334792399432386] 
    \draw  [color={rgb, 255:red, 208; green, 2; blue, 27 }  ,draw opacity=1 ][line width=3] [line join = round][line cap = round] (460,139.75) .. controls (460,139.75) and (460,139.75) .. (460,139.75) ;
    %Shape: Free Drawing [id:dp030596974114721176] 
    \draw  [color={rgb, 255:red, 208; green, 2; blue, 27 }  ,draw opacity=1 ][line width=3] [line join = round][line cap = round] (460,138.75) .. controls (460.47,138.75) and (461,139.28) .. (461,139.75) ;
    
    % Text Node
    \draw (233,139.4) node [anchor=north west][inner sep=0.75pt]    {$a_{p+1,1}$};
    % Text Node
    \draw (153,44.4) node [anchor=north west][inner sep=0.75pt]    {$a_{p+1,2}$};
    % Text Node
    \draw (558,137.4) node [anchor=north west][inner sep=0.75pt]    {$a_{k,1}$};
    % Text Node
    \draw (469,39.4) node [anchor=north west][inner sep=0.75pt]    {$a_{k,2}$};
    % Text Node
    \draw (402,342.65) node [anchor=north west][inner sep=0.75pt]    {$a_{1,1}$};
    % Text Node
    \draw (319,240.4) node [anchor=north west][inner sep=0.75pt]    {$a_{1,2}$};
    % Text Node
    \draw (152,85.4) node [anchor=north west][inner sep=0.75pt]  [color={rgb, 255:red, 208; green, 2; blue, 27 }  ,opacity=1 ]  {$\{a_{p+1,1} \ =\ 0\}$};
    % Text Node
    \draw (224,345.4) node [anchor=north west][inner sep=0.75pt]  [color={rgb, 255:red, 208; green, 2; blue, 27 }  ,opacity=1 ]  {$\{a_{1,2} \ =\ 0\}$};
    % Text Node
    \draw (469,108.4) node [anchor=north west][inner sep=0.75pt]  [color={rgb, 255:red, 208; green, 2; blue, 27 }  ,opacity=1 ]  {$\{a_{k,1} \ =\ 0\ =\ a_{k,2}\}$};
    % Text Node
    \draw (76,44.4) node [anchor=north west][inner sep=0.75pt]    {$U_{\sigma }{}_{_{p+1}}$};
    % Text Node
    \draw (378,42.4) node [anchor=north west][inner sep=0.75pt]    {$\underset{2 \le k \le p}{U_{\sigma_k}}$};
    % Text Node
    \draw (231,249.4) node [anchor=north west][inner sep=0.75pt]    {$U_{\sigma }{}_{_{1}}$};

    \end{tikzpicture}%
    }
    
      \caption{The fixed loci in the \((p+1)\) affine charts of the local
      toric resolution \(X_\Sigma\)}
        \label{split:fig:fixed loci}

    \end{figure}
    
    In this local resolution there are exactly \(p-1\) isolated fixed points
    \begin{equation}\label{split:iso.fixpt}
    Q_{i-1}=\{a_{i,1}=a_{i,2}=0\},
    \qquad 2\leq i\leq p,
    \end{equation}
    and two fixed boundary branches
    \begin{equation}\label{split:fixcurve}
    \{a_{1,2}=0\},
    \qquad
    \{a_{p+1,1}=0\}.
    \end{equation}
    
    \begin{remark} \label{split:rem2.2}
        In the local Galois model
        \(Y_{\mathscr A}=\operatorname{Spec}k[v,s,z]/(z^N-v^{p+1}s)\),
        the stabilizer
        fixes \(\{v=0\}\) and \(\{s=0\}\) and acts on \(z\) through
        \(\bar\zeta\).  These are the two boundary branches through one
        incidence point.  As the incidence point varies, the first branches
        are the local branches of the strict transforms of the \(N\) curves
        \(E_P\), and the second branches are those of the strict transforms
        of the \(N\) curves \(\widetilde L\).  Thus the global fixed locus contains \(N\) curves
        \(\widetilde E'_P\) and \(N\) curves \(\widetilde L'_L\).
        % More precisely, the two boundary valuations in
        % \(z^N=v^{p+1}s\) are \(p+1\) and \(1\).  Since
        % \(\gcd(N,p+1)=1\), the cover is totally ramified at the generic point
        % of either branch.  Thus every base boundary component has one reduced
        % irreducible curve above it in a fixed normal component, that curve is
        % fixed pointwise by \(\langle\sigma\rangle\), and resolving the
        % incidence singularities does not change its generic point.
           
      %  \[ \mathrm{Spec} \left( \frac{k[v,s,z]}{(z^{N} -  v^{p+1} s) }  \right) \simeq \mathrm{Spec} \left( k[X^pY, XY, XY^{p^2}] \right) \simeq \mathrm{Spec} \left( k [t_2, t_1^N t_2^{-p-1} , t_1]   \right)   \]
    \end{remark}

    %\newpage
    
    \subsection{Intersection number calculations}\label{sec:split-intersections}
    
    \subsubsection{Local computations on the minimal resolution}
    
    The fan already contains the intersection data needed below.  By the
    wall relations and intersection formulas
    \eqref{eq:toric-wall-relations}--\eqref{eq:toric-exceptional-intersections},
    the following computation is entirely local.  Let
    \(D_i=V(\rho_i)\), where \(\rho_i=\operatorname{Cone}(u_i)\).  For the
    Hirzebruch--Jung coefficients
    \[
    (b_1,\ldots,b_p)=(\underbrace{2,\ldots,2}_{p-1\ \text{terms}},p+2),
    \]
    the wall relations
    \(
    u_{i-1}+u_{i+1}=b_i u_i
    \)
    give
    \begin{equation}
    D_i\cdot D_j=
    \begin{cases}
    -b_i,&i=j,\\
    1,&|i-j|=1,\\
    0,&|i-j|>1,
    \end{cases}
    \qquad 1\leq i,j\leq p;
    \label{split:intmateqn}
    \end{equation}
    see \cite[Theorem~10.2.5]{coxToricVarieties2011a}.  Thus the exceptional
    curves \(E_i:=D_i\), \(1\leq i\leq p\), form a negative-definite chain
    with
    \begin{equation}
    E_i^2=-2\quad(1\leq i<p),
    \qquad
    E_p^2=-p-2.
    \label{split:eq-intdiv}
    \end{equation}
    
  The exceptional locus is
\(\phi^{-1}(o_\sigma)=E_1\cup\cdots\cup E_p\).
    The boundary divisors \(D_0\) and \(D_{p+1}\) are noncomplete in the
    local toric surface; their global counterparts on
    \(\widetilde{\overline {\mathbf{Y}}}_\ell(w)\) are complete curves.
    
    For any character \(\chi^m\), its torus-invariant divisor is
    \begin{equation}
    \operatorname{div}(\chi^m) =\sum_{\rho\in\Sigma(1)}\langle m,u_\rho\rangle D_\rho. 
    \label{split:diveqn}
    \end{equation}
    This is the only general toric divisor formula used below; see
    \cite[Proposition~4.1.1]{coxToricVarieties2011a}.
    
    By Remark~\ref{split:rem2.2}, choose one curve of each global type and
    denote the representatives by \(\widetilde L'\) and \(\widetilde E'\).
    Locally over an incidence point of \(\overline {\mathbf{X}}\), they are the prime divisors
    \(D_0\) and \(D_{p+1}\), respectively. The characters
    \(\chi^{(N,-p-1)}=s\) and \(\chi^{(0,1)}=v\) give the total transforms
    of the two base boundary divisors.  The same local description holds
    for every curve of the corresponding type.
    
    % \medskip
    
  Equation~\eqref{split:diveqn} gives
\begin{align}
\operatorname{div}(\chi^{(N,-p-1)})
  &= N D_0 + \sum_{i=1}^{p}\bigl(N-(p+1)i\bigr)E_i,
  \label{split:diveqn2}\\
\operatorname{div}(\chi^{(0,1)})
  &= N D_{p+1} + \sum_{i=1}^{p} iE_i.
\end{align}
    
    \subsubsection{Global intersection number calculations}

    Over each of the \(N(p+1)\) boundary crossing points, equivalently each
    incident pair \((P,L)\) with \(P\in L\), there are \(p-1\) isolated fixed
    points.  Hence there are
    \(N(p+1)(p-1)=N(p^2-1)\) isolated fixed points globally.  In addition,
    there are \(N\) fixed curves \(\widetilde E'_P\), one above each
    exceptional boundary component, and \(N\) fixed curves
    \(\widetilde L'_L\), one above each line-type boundary component.  All
    are fixed pointwise by the component stabilizer
    \(\langle\sigma\rangle\simeq\mathbb Z/N\mathbb Z\).
    
    \medskip
    
    Let \(L\subset\mathbb P^2\) be a \(k\)-rational line, and let
    \(P\in L(k)\).  Write \(\widetilde L\) for the strict transform of
    \(L\) in \(\overline {\mathbf{X}}\) and \(\widetilde E_P\) for the exceptional
    curve above \(P\).  These curves meet transversely in one \(k\)-rational
    point, as shown in Figure~\ref{fig:split-boundary-p2}.  When \(P\) is
    understood, we abbreviate \(\widetilde E_P\) to \(\widetilde E\).
    
    \begin{figure}[H]
        \centering
        %\includegraphics[width=0.5\linewidth]{}
    %\tikzset{every picture/.style={line width=0.75pt}} %set default line width to 0.75pt        
    
   \resizebox{0.7\textwidth}{!}{%
    \begin{tikzpicture}[x=0.75pt,y=0.75pt,yscale=-1,xscale=1]
        % Your existing diagram code
 
    %Straight Lines [id:da49797362836176207] 
    \draw [color={rgb, 255:red, 189; green, 16; blue, 224 }  ,draw opacity=1 ]   (267,196) -- (367,296) ;
    %Straight Lines [id:da04531957030202827] 
    \draw    (374.5,198.88) -- (259.5,293.13) ;
    %Straight Lines [id:da05179352714164187] 
    \draw    (227,230.75) -- (405,259.75) ;
    %Straight Lines [id:da2945264711025648] 
    \draw [color={rgb, 255:red, 208; green, 2; blue, 27 }  ,draw opacity=1 ]   (290,10.75) -- (356,121.75) ;
    %Straight Lines [id:da8349525921367715] 
    \draw [color={rgb, 255:red, 189; green, 16; blue, 224 }  ,draw opacity=1 ]   (242,56) -- (369,16.75) ;
    %Straight Lines [id:da363269182745836] 
    \draw    (257,86) -- (384,46.75) ;
    %Straight Lines [id:da83410627174403] 
    \draw    (270,119) -- (397,79.75) ;
    %Straight Lines [id:da061900641532805833] 
    \draw    (310,134.75) -- (310.95,173) ;
    \draw [shift={(311,175)}, rotate = 268.58] [color={rgb, 255:red, 0; green, 0; blue, 0 }  ][line width=0.75]    (10.93,-3.29) .. controls (6.95,-1.4) and (3.31,-0.3) .. (0,0) .. controls (3.31,0.3) and (6.95,1.4) .. (10.93,3.29)   ;
    
    % Text Node
    \draw (249,182.4) node [anchor=north west][inner sep=0.75pt]  [color={rgb, 255:red, 189; green, 16; blue, 224 }  ,opacity=1 ]  {$L$};
    % Text Node
    \draw (223,40.4) node [anchor=north west][inner sep=0.75pt]  [color={rgb, 255:red, 189; green, 16; blue, 224 }  ,opacity=1 ]  {$\tilde{L}$};
    % Text Node
    \draw (274,4.4) node [anchor=north west][inner sep=0.75pt]  [color={rgb, 255:red, 208; green, 2; blue, 27 }  ,opacity=1 ]  {$\tilde{E}$};
    % Text Node
    \draw (65.6,70.09) node [anchor=north west][inner sep=0.75pt]    {$\overline{\mathbf{X}} \ =\ Bl_{Z}\left( \ \mathbb{P}^{2} \ \right)$};
    % Text Node
    \draw (120,235.4) node [anchor=north west][inner sep=0.75pt]    {$\mathbb{P}^{2}$};

    \end{tikzpicture}%
    }

        \caption{The divisors \(\widetilde L\) and \(\widetilde E_P\) for \(p=2\)}
        \label{fig:split-boundary-p2}
        \label{split:blowupXbarandP2}
    \end{figure}
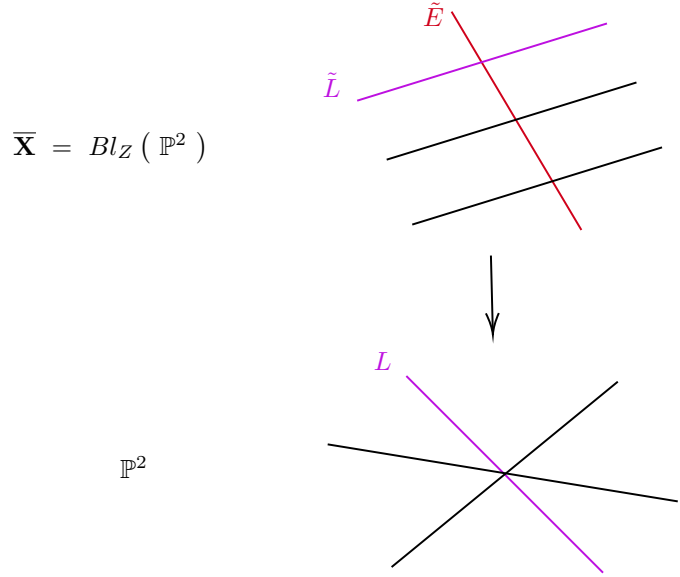
    
    \bigskip
    
    Let
    \(
    \widetilde\pi_\ell:
    \widetilde{\overline {\mathbf{Y}}}_\ell(w)\longrightarrow\overline {\mathbf{X}}(w)
    \)
    be the generically finite morphism of degree \(N\) obtained by composing
    the minimal resolution \(\rho_\ell\) with the finite morphism
    \(\overline\pi_\ell:\overline {\mathbf{Y}}_\ell(w)\to\overline {\mathbf{X}}(w)\).
    Let \(\widetilde E'\) and \(\widetilde L'\) denote arbitrary
    representatives of the two types of strict transforms in
    \(\widetilde{\overline {\mathbf{Y}}}_\ell(w)\).  Both are complete curves isomorphic to
    \(\mathbb P^1\).  Every base boundary component contains exactly
    \(p+1\) incidence points, so the calculation from
    \eqref{split:diveqn2} is identical for every representative and gives
    \begin{align} \label{split:globaleqn}
    \widetilde\pi_\ell^{\,*}(\widetilde L)
    &=N\widetilde L'
    +\sum_{j=1}^{p+1}\sum_{i=1}^{p}
    \bigl(N-(p+1)i\bigr)E_{i,j},\\
    \widetilde\pi_\ell^{\,*}(\widetilde E)
    &=N\widetilde E'
    +\sum_{j=1}^{p+1}\sum_{i=1}^p iE_{i,j}.
    \end{align}
    In each formula, \(j\) indexes the \(p+1\) crossing points on the chosen
    base curve; the two formulas generally involve different exceptional chains.
    
    \bigskip
    
   The following two figures illustrate the exceptional
locus and the fixed locus of \(h\) on the smooth equivariant
resolution \(\widetilde {\mathbf{Y}} = \bigsqcup_{\ell} \widetilde{ Z_{\ell}} \to\overline {\mathbf{Y}} = \bigsqcup_{\ell} \widetilde{ \mathbf{Y}_{\ell}}\) described above,
with the fixed locus shown in red.
 
 
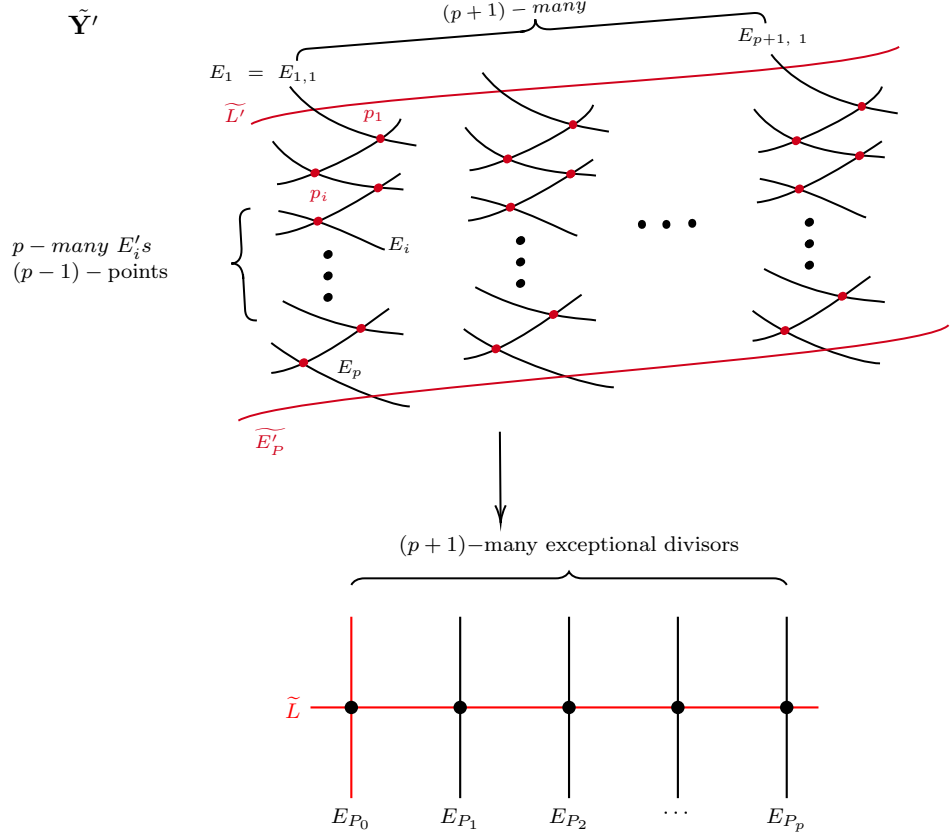
\begin{figure}[H]
    \centering

 % \hspace*{-1cm} 

\begin{tikzpicture}[
    x=0.75pt,
    y=0.75pt,
    yscale=-1,
    xscale=1,
    scale=0.8,
    transform shape=false
]
% \begin{tikzpicture}[x=0.75pt,y=0.75pt,yscale=-1,xscale=1]
%uncomment if require: \path (0,410); %set diagram left start at 0, and has height of 410
\hspace{-1cm}

\draw [color={rgb, 255:red, 208; green, 2; blue, 27 }  ,draw opacity=1 ]   (172.98,91.15) .. controls (209.03,70.96) and (543.44,62.89) .. (579.5,42.7) ;
%Curve Lines [id:da45139308109649945] 
\draw    (197.31,67.6) .. controls (220.75,98.55) and (270.33,103.26) .. (276.64,104.61) ;
%Curve Lines [id:da6410229274767396] 
\draw    (266.72,87.78) .. controls (264.02,100.57) and (205.43,128.16) .. (189.2,128.83) ;
%Curve Lines [id:da5542053443846687] 
\draw    (186.5,101.92) .. controls (209.93,132.87) and (262.21,130.85) .. (268.52,132.2) ;
%Shape: Free Drawing [id:dp6233774636338095] 
\draw  [color={rgb, 255:red, 208; green, 2; blue, 27 }  ,draw opacity=1 ][line width=3] [line join = round][line cap = round] (254.1,99.9) .. controls (254.4,99.9) and (254.1,100.35) .. (254.1,100.57) ;
%Shape: Free Drawing [id:dp8240828806041681] 
\draw  [color={rgb, 255:red, 208; green, 2; blue, 27 }  ,draw opacity=1 ][line width=3] [line join = round][line cap = round] (213.54,121.43) .. controls (213.11,121.43) and (212.64,121.79) .. (212.64,122.1) ;
%Curve Lines [id:da6016325315525106] 
\draw    (266.72,122.1) .. controls (256.81,128.83) and (214.44,157.77) .. (188.3,160.46) ;
%Shape: Free Drawing [id:dp05714342391799454] 
\draw  [color={rgb, 255:red, 208; green, 2; blue, 27 }  ,draw opacity=1 ][line width=3] [line join = round][line cap = round] (253.2,130.85) .. controls (252.78,130.85) and (252.3,131.21) .. (252.3,131.52) ;
%Curve Lines [id:da7098329964291988] 
\draw    (190.1,145.66) .. controls (222.55,151.71) and (250.5,168.53) .. (256.81,169.88) ;
%Shape: Free Drawing [id:dp40989622553796934] 
\draw  [color={rgb, 255:red, 208; green, 2; blue, 27 }  ,draw opacity=1 ][line width=3] [line join = round][line cap = round] (215.34,151.71) .. controls (214.92,151.71) and (214.44,152.07) .. (214.44,152.38) ;
%Shape: Free Drawing [id:dp8534766020493187] 
\draw  [color={rgb, 255:red, 0; green, 0; blue, 0 }  ,draw opacity=1 ][line width=3] [line join = round][line cap = round] (220.75,186.7) .. controls (221.18,186.7) and (221.65,186.35) .. (221.65,186.03) ;
%Shape: Free Drawing [id:dp37045942885780136] 
\draw  [color={rgb, 255:red, 0; green, 0; blue, 0 }  ,draw opacity=1 ][line width=3] [line join = round][line cap = round] (220.75,200.16) .. controls (221.18,200.16) and (221.65,199.81) .. (221.65,199.49) ;
%Shape: Free Drawing [id:dp8174280474816694] 
\draw  [color={rgb, 255:red, 0; green, 0; blue, 0 }  ,draw opacity=1 ][line width=3] [line join = round][line cap = round] (220.75,173.25) .. controls (221.18,173.25) and (221.65,172.89) .. (221.65,172.57) ;
%Curve Lines [id:da6734496178012541] 
\draw    (185.6,228.42) .. controls (211.74,248.61) and (259.88,268.25) .. (272.5,268.25) ;
%Curve Lines [id:da5712551805618181] 
\draw    (259.51,206.89) .. controls (249.59,213.62) and (211.74,244.57) .. (185.6,247.27) ;
%Curve Lines [id:da9253609883277154] 
\draw    (192.81,202.18) .. controls (230.67,222.37) and (262.21,221.7) .. (268.52,223.04) ;
%Shape: Free Drawing [id:dp20609379844283393] 
\draw  [color={rgb, 255:red, 208; green, 2; blue, 27 }  ,draw opacity=1 ][line width=3] [line join = round][line cap = round] (242.38,219) .. controls (241.96,219) and (241.48,219.36) .. (241.48,219.68) ;
%Shape: Free Drawing [id:dp9298785181863807] 
\draw  [color={rgb, 255:red, 208; green, 2; blue, 27 }  ,draw opacity=1 ][line width=3] [line join = round][line cap = round] (206.33,240.54) .. controls (205.9,240.54) and (205.43,240.89) .. (205.43,241.21) ;
%Curve Lines [id:da07938028985907075] 
\draw    (318.1,58.85) .. controls (341.54,89.8) and (391.11,94.51) .. (397.42,95.86) ;
%Curve Lines [id:da28951804126653424] 
\draw    (387.51,79.04) .. controls (384.8,91.82) and (326.21,119.41) .. (309.99,120.08) ;
%Curve Lines [id:da18767917028084546] 
\draw    (307.28,93.17) .. controls (330.72,124.12) and (383,122.1) .. (389.31,123.45) ;
%Shape: Free Drawing [id:dp4328312236464926] 
\draw  [color={rgb, 255:red, 208; green, 2; blue, 27 }  ,draw opacity=1 ][line width=3] [line join = round][line cap = round] (374.89,91.15) .. controls (375.19,91.15) and (374.89,91.6) .. (374.89,91.82) ;
%Shape: Free Drawing [id:dp1955857199218458] 
\draw  [color={rgb, 255:red, 208; green, 2; blue, 27 }  ,draw opacity=1 ][line width=3] [line join = round][line cap = round] (334.32,112.68) .. controls (333.9,112.68) and (333.42,113.04) .. (333.42,113.36) ;
%Curve Lines [id:da20043330643232737] 
\draw    (387.51,113.36) .. controls (377.59,120.08) and (335.23,149.02) .. (309.09,151.71) ;
%Shape: Free Drawing [id:dp6643347261340744] 
\draw  [color={rgb, 255:red, 208; green, 2; blue, 27 }  ,draw opacity=1 ][line width=3] [line join = round][line cap = round] (373.99,122.1) .. controls (373.56,122.1) and (373.08,122.46) .. (373.08,122.78) ;
%Curve Lines [id:da22841188525021472] 
\draw    (310.89,136.91) .. controls (343.34,142.96) and (371.28,159.79) .. (377.59,161.13) ;
%Shape: Free Drawing [id:dp7033062462174565] 
\draw  [color={rgb, 255:red, 208; green, 2; blue, 27 }  ,draw opacity=1 ][line width=3] [line join = round][line cap = round] (336.13,142.96) .. controls (335.7,142.96) and (335.23,143.32) .. (335.23,143.64) ;
%Shape: Free Drawing [id:dp6417268948640699] 
\draw  [color={rgb, 255:red, 0; green, 0; blue, 0 }  ,draw opacity=1 ][line width=3] [line join = round][line cap = round] (341.54,177.96) .. controls (341.96,177.96) and (342.44,177.6) .. (342.44,177.28) ;
%Shape: Free Drawing [id:dp5602402244413223] 
\draw  [color={rgb, 255:red, 0; green, 0; blue, 0 }  ,draw opacity=1 ][line width=3] [line join = round][line cap = round] (341.54,191.41) .. controls (341.96,191.41) and (342.44,191.06) .. (342.44,190.74) ;
%Shape: Free Drawing [id:dp42736917653333273] 
\draw  [color={rgb, 255:red, 0; green, 0; blue, 0 }  ,draw opacity=1 ][line width=3] [line join = round][line cap = round] (341.54,164.5) .. controls (341.96,164.5) and (342.44,164.14) .. (342.44,163.82) ;
%Curve Lines [id:da22728875467966703] 
\draw    (306.38,219.68) .. controls (332.52,239.86) and (387.88,256.25) .. (400.5,256.25) ;
%Curve Lines [id:da7850462179951794] 
\draw    (380.29,198.14) .. controls (370.38,204.87) and (332.52,235.83) .. (306.38,238.52) ;
%Curve Lines [id:da9172343257692173] 
\draw    (313.59,193.43) .. controls (351.45,213.62) and (383,212.95) .. (389.31,214.29) ;
%Shape: Free Drawing [id:dp8081554942536521] 
\draw  [color={rgb, 255:red, 208; green, 2; blue, 27 }  ,draw opacity=1 ][line width=3] [line join = round][line cap = round] (363.17,210.26) .. controls (362.74,210.26) and (362.27,210.61) .. (362.27,210.93) ;
%Shape: Free Drawing [id:dp37243616309222904] 
\draw  [color={rgb, 255:red, 208; green, 2; blue, 27 }  ,draw opacity=1 ][line width=3] [line join = round][line cap = round] (327.11,231.79) .. controls (326.69,231.79) and (326.21,232.14) .. (326.21,232.46) ;
%Shape: Free Drawing [id:dp1480079194894004] 
\draw  [color={rgb, 255:red, 0; green, 0; blue, 0 }  ,draw opacity=1 ][line width=3] [line join = round][line cap = round] (418.15,154.4) .. controls (418.15,154.18) and (418.15,153.96) .. (418.15,153.73) ;
%Shape: Free Drawing [id:dp5084645946189612] 
\draw  [color={rgb, 255:red, 0; green, 0; blue, 0 }  ,draw opacity=1 ][line width=3] [line join = round][line cap = round] (433.48,154.4) .. controls (433.48,154.18) and (433.48,153.96) .. (433.48,153.73) ;
%Shape: Free Drawing [id:dp4416580064403587] 
\draw  [color={rgb, 255:red, 0; green, 0; blue, 0 }  ,draw opacity=1 ][line width=3] [line join = round][line cap = round] (449.7,153.73) .. controls (449.7,153.51) and (449.7,153.28) .. (449.7,153.06) ;
%Curve Lines [id:da2888162360992488] 
\draw    (499.28,47.41) .. controls (522.71,78.36) and (572.29,83.07) .. (578.6,84.42) ;
%Curve Lines [id:da12548494314415914] 
\draw    (568.68,67.6) .. controls (565.98,80.38) and (507.39,107.97) .. (491.16,108.65) ;
%Curve Lines [id:da14965655660488575] 
\draw    (488.46,81.73) .. controls (511.9,112.68) and (564.18,110.66) .. (570.49,112.01) ;
%Shape: Free Drawing [id:dp6723621627024006] 
\draw  [color={rgb, 255:red, 208; green, 2; blue, 27 }  ,draw opacity=1 ][line width=3] [line join = round][line cap = round] (556.06,79.71) .. controls (556.36,79.71) and (556.06,80.16) .. (556.06,80.38) ;
%Shape: Free Drawing [id:dp36653673275415377] 
\draw  [color={rgb, 255:red, 208; green, 2; blue, 27 }  ,draw opacity=1 ][line width=3] [line join = round][line cap = round] (515.5,101.24) .. controls (515.08,101.24) and (514.6,101.6) .. (514.6,101.92) ;
%Curve Lines [id:da8953354002389872] 
\draw    (568.68,101.92) .. controls (558.77,108.65) and (516.4,137.58) .. (490.26,140.27) ;
%Shape: Free Drawing [id:dp7812939580254268] 
\draw  [color={rgb, 255:red, 208; green, 2; blue, 27 }  ,draw opacity=1 ][line width=3] [line join = round][line cap = round] (555.16,110.66) .. controls (554.74,110.66) and (554.26,111.02) .. (554.26,111.34) ;
%Curve Lines [id:da24863273800334162] 
\draw    (492.07,125.47) .. controls (524.52,131.52) and (552.46,148.35) .. (558.77,149.69) ;
%Shape: Free Drawing [id:dp37815369342726113] 
\draw  [color={rgb, 255:red, 208; green, 2; blue, 27 }  ,draw opacity=1 ][line width=3] [line join = round][line cap = round] (517.3,131.52) .. controls (516.88,131.52) and (516.4,131.88) .. (516.4,132.2) ;
%Shape: Free Drawing [id:dp35327372291059367] 
\draw  [color={rgb, 255:red, 0; green, 0; blue, 0 }  ,draw opacity=1 ][line width=3] [line join = round][line cap = round] (522.71,166.52) .. controls (523.14,166.52) and (523.61,166.16) .. (523.61,165.84) ;
%Shape: Free Drawing [id:dp8517885056557342] 
\draw  [color={rgb, 255:red, 0; green, 0; blue, 0 }  ,draw opacity=1 ][line width=3] [line join = round][line cap = round] (522.71,179.97) .. controls (523.14,179.97) and (523.61,179.62) .. (523.61,179.3) ;
%Shape: Free Drawing [id:dp2887948946218132] 
\draw  [color={rgb, 255:red, 0; green, 0; blue, 0 }  ,draw opacity=1 ][line width=3] [line join = round][line cap = round] (522.71,153.06) .. controls (523.14,153.06) and (523.61,152.7) .. (523.61,152.38) ;
%Curve Lines [id:da8862758208940785] 
\draw    (487.56,208.24) .. controls (513.7,228.42) and (554.88,241.25) .. (567.5,241.25) ;
%Curve Lines [id:da8378048261903746] 
\draw    (561.47,186.7) .. controls (551.56,193.43) and (513.7,224.39) .. (487.56,227.08) ;
%Curve Lines [id:da0680035488608377] 
\draw    (494.77,181.99) .. controls (532.63,202.18) and (564.18,201.51) .. (570.49,202.85) ;
%Shape: Free Drawing [id:dp08532992512756299] 
\draw  [color={rgb, 255:red, 208; green, 2; blue, 27 }  ,draw opacity=1 ][line width=3] [line join = round][line cap = round] (544.35,198.82) .. controls (543.92,198.82) and (543.44,199.17) .. (543.44,199.49) ;
%Shape: Free Drawing [id:dp8405520109244589] 
\draw  [color={rgb, 255:red, 208; green, 2; blue, 27 }  ,draw opacity=1 ][line width=3] [line join = round][line cap = round] (508.29,220.35) .. controls (507.87,220.35) and (507.39,220.71) .. (507.39,221.02) ;
%Shape: Brace [id:dp9992909476483512] 
\draw   (172.53,144.14) .. controls (167.87,144.41) and (165.67,146.88) .. (165.94,151.53) -- (166.91,168.25) .. controls (167.3,174.9) and (165.16,178.36) .. (160.5,178.63) .. controls (165.16,178.36) and (167.68,181.56) .. (168.06,188.21)(167.89,185.22) -- (169.19,207.71) .. controls (169.46,212.37) and (171.93,214.56) .. (176.58,214.29) ;
%Shape: Brace [id:dp953419405355643] 
\draw   (495.67,27.9) .. controls (495.3,23.25) and (492.78,21.11) .. (488.13,21.48) -- (361.76,31.6) .. controls (355.11,32.13) and (351.6,30.07) .. (351.23,25.42) .. controls (351.6,30.07) and (348.47,32.67) .. (341.82,33.2)(344.81,32.96) -- (208.24,43.9) .. controls (203.59,44.27) and (201.45,46.78) .. (201.82,51.43) ;
%Straight Lines [id:da975565680530631] 
\draw [color={rgb, 255:red, 0; green, 0; blue, 0 }  ,draw opacity=1 ]   (329.02,284.12) -- (329.48,339.25) ;
\draw [shift={(329.5,341.25)}, rotate = 269.52] [color={rgb, 255:red, 0; green, 0; blue, 0 }  ,draw opacity=1 ][line width=0.75]    (10.93,-3.29) .. controls (6.95,-1.4) and (3.31,-0.3) .. (0,0) .. controls (3.31,0.3) and (6.95,1.4) .. (10.93,3.29)   ;
%Curve Lines [id:da16523123542192997] 
\draw [color={rgb, 255:red, 208; green, 2; blue, 27 }  ,draw opacity=1 ]   (164.98,277.15) .. controls (201.03,256.96) and (591.5,237.25) .. (610.5,217.25) ;

% Text Node
\draw (154.43,76.26) node [anchor=north west][inner sep=0.75pt]  [font=\scriptsize,color={rgb, 255:red, 208; green, 2; blue, 27 }  ,opacity=1 ]  {$\widetilde{L'}$};
% Text Node
\draw (13.08,157.61) node [anchor=north west][inner sep=0.75pt]  [font=\footnotesize,color={rgb, 255:red, 0; green, 0; blue, 0 }  ,opacity=1 ]  {$ \begin{array}{l}
p-many\ E_{i} 's\\
(p-1)-\text{points}\\
\end{array}$};
% Text Node
\draw (144.46,52.5) node [anchor=north west][inner sep=0.75pt]  [font=\scriptsize]  {$E_{1} \ =\ E_{1,1}$};
% Text Node
\draw (256.69,160.33) node [anchor=north west][inner sep=0.75pt]  [font=\scriptsize]  {$E_{i}$};
% Text Node
\draw (207.93,130.55) node [anchor=north west][inner sep=0.75pt]  [font=\scriptsize]  {$\textcolor[rgb]{0.82,0.01,0.11}{p}\textcolor[rgb]{0.82,0.01,0.11}{_{i}}$};
% Text Node
\draw (224.59,237.4) node [anchor=north west][inner sep=0.75pt]  [font=\scriptsize]  {$E_{p}$};
% Text Node
\draw (241.49,80.7) node [anchor=north west][inner sep=0.75pt]  [font=\scriptsize]  {$\textcolor[rgb]{0.82,0.01,0.11}{p}\textcolor[rgb]{0.82,0.01,0.11}{_{1}}$};
% Text Node
\draw (290.37,14.98) node [anchor=north west][inner sep=0.75pt]  [font=\scriptsize,rotate=-356.1]  {$( p+1) -many$};
% Text Node
\draw (475.91,28.94) node [anchor=north west][inner sep=0.75pt]  [font=\scriptsize]  {$E_{p+1,\ 1}$};
% Text Node
\draw (56.28,17.15) node [anchor=north west][inner sep=0.75pt]  [font=\normalsize,color={rgb, 255:red, 0; green, 0; blue, 0 }  ,opacity=1 ]  {$\tilde{\mathbf{Y}'}$};
% Text Node
\draw (173.43,280.26) node [anchor=north west][inner sep=0.75pt]  [font=\scriptsize,color={rgb, 255:red, 208; green, 2; blue, 27 }  ,opacity=1 ]  {$\widetilde{E_{P} '}$};

\end{tikzpicture}

\begin{tikzpicture}[
    scale=1.2,
    every node/.style={font=\footnotesize}
]

\vspace{1.5em}
% Exceptional divisors
\draw[red,thick] (0,-1.0) -- (0,1.0);
\draw[black,thick] (1.2,-1.0) -- (1.2,1.0);
\draw[black,thick] (2.4,-1.0) -- (2.4,1.0);
\draw[black,thick] (3.6,-1.0) -- (3.6,1.0);
\draw[black,thick] (4.8,-1.0) -- (4.8,1.0);

% Exceptional divisor labels
\node[below] at (0,-1.0) {$E_{P_0}$};
\node[below] at (1.2,-1.0) {$E_{P_1}$};
\node[below] at (2.4,-1.0) {$E_{P_2}$};
\node[below] at (3.6,-1.0) {$\cdots$};
\node[below] at (4.8,-1.0) {$E_{P_p}$};

% Strict transform
\draw[red,thick] (-0.45,0) -- (5.15,0);
\node[red,left] at (-0.45,0) {$\widetilde L$};

% Intersection points
\foreach \x in {0,1.2,2.4,3.6,4.8} {
    \fill[black] (\x,0) circle (2.1pt);
}

% Brace
\draw[decorate,decoration={brace,amplitude=5pt}]
    (0,1.35) -- (4.8,1.35)
    node[midway,above=7pt] {$(p+1)-$many exceptional divisors};

% Caption below, optional
\node[align=center] at (2.4,-2.45)
     {$\overline{\mathbf{X}}(w)= \mathrm{Bl}_{\mathbb{P}^2(\mathbb{F}_p)}(\mathbb{P}^2)$};

\end{tikzpicture}

\end{figure}
    
    \medskip
    
 The following calculation gives the two boundary self-intersections used
in the Brauer--Lefschetz formula.

\begin{lemma}[Self-intersections of the lifted boundary curves]
\label{lem:split-lifted-boundary-intersections}
The strict transforms on
\(\widetilde{\overline {\mathbf{Y}}}_\ell(w)\) satisfy
\[
    (\widetilde L')^2=-p,
    \qquad
    (\widetilde E')^2=-1.
\]
\end{lemma}

\begin{proof}
For coefficients \(a_1,\ldots,a_p\), put
\[
    Q_j(a)=\sum_{i=1}^p a_iE_{i,j}.
\]
Using
\[
    E_{i,j}^2=-2\quad(i<p),
    \qquad
    E_{p,j}^2=-p-2,
    \qquad
    E_{i,j}\cdot E_{i+1,j}=1,
\]
one obtains, by direct expansion,
\begin{equation}
    Q_j(a)^2
    =
    -a_1^2
    -\sum_{i=1}^{p-1}(a_i-a_{i+1})^2
    -(p+1)a_p^2.
    \label{eq:split-chain-quadratic-form}
\end{equation}
The chains over distinct crossing points are disjoint. For the line-type divisor, set
\[
    c_i=N-(p+1)i.
\]
Then \(c_1=p^2\), \(c_p=1\), and
\(c_i-c_{i+1}=p+1\). Hence
\[
    Q_j(c)^2
    =-p^4-(p-1)(p+1)^2-(p+1)
    =-p^2N.
\]
Moreover,
\[
    \widetilde L'\cdot Q_j(c)=c_1=p^2.
\]
The projection formula and
\[
    \widetilde\pi_\ell^{\,*}\widetilde L
    =N\widetilde L'+\sum_{j=1}^{p+1}Q_j(c),
    \qquad
    \widetilde L^2=-p,
\]
give
\[
    -pN
    =N^2(\widetilde L')^2+Np^2(p+1),
\]
and therefore
\begin{equation}
    (\widetilde L')^2=-p.
    \label{split:selfLtilda}
\end{equation}

For the exceptional-type divisor, set \(e_i=i\). Equation
\eqref{eq:split-chain-quadratic-form} gives
\[
    Q_j(e)^2
    =-1-(p-1)-(p+1)p^2
    =-pN.
\]
Since \(\widetilde E'\) meets only \(E_{p,j}\), we have \(
    \widetilde E'\cdot Q_j(e)=p.
\) 

\medskip

Now using
\(
    \widetilde\pi_\ell^{\,*}\widetilde E
    =N\widetilde E'+\sum_{j=1}^{p+1}Q_j(e),\) and \(
    \widetilde E^2=-1
\)
the projection formula yields
\(
    -N
    =N^2(\widetilde E')^2+Np(p+1).
\)
Thus we obtain
\begin{equation}
    (\widetilde E')^2=-1.
    \label{split:selfEtilda}
\end{equation}
\end{proof}

\subsection{Brauer--Lefschetz calculation}
\label{sec:split-brauer-calculation}

Work over \(\bar{k}\), and let \(N=p^2+p+1\).
Recall that the group \( T^{wF}\) acting on the right is generated
by an element \(\tau\) of order \((p-1)N\), while
\[
    \sigma=\tau^{p-1}
\]
generates the stabilizer of a chosen component and has exact order \(N\).
The local covering coordinate \(z\) transforms as
\[
    \sigma(z)=\bar\zeta z,
    \qquad
    \operatorname{ord}(\bar\zeta)=N.
\]
Since \(N\mid p^3-1\), we may take
\(\bar\zeta\in\mathbb F_{p^3}^{\times}\). Put
\[
    \widehat\zeta=[\bar\zeta]\in W(\mathbb F_{p^3})^\times.
\]
Then \(\widehat\zeta^N=1\), its reduction modulo \(p\) is
\(\bar\zeta\), and
\[
    [\bar\zeta^m]=\widehat\zeta^m
    \qquad(m\in\mathbb Z).
\]
Because \(0<|m|\leq p<N\), each
\(1-\widehat\zeta^m\) occurring below is a unit in
\(W(\mathbb F_{p^3})\).

    \subsubsection{Fixed-locus data}
    
    The inputs to Lefschetz--Riemann--Roch for one split component are
    summarized below.  For an isolated point \(Q_{i-1}\) indexed by \(2\leq i\leq p\),
    the last column lists the two cotangent eigenvalues; for a curve it lists
    the normal eigenvalue.
    \[
    \begin{array}{c|c|c|c|c}
    \text{fixed stratum}&\text{number}&\text{genus}&\text{self-intersection}
    &\text{eigenvalue(s)}\\ \hline
    Q_{i-1}&N(p+1)\text{ for each }i&-&-&
    (\bar\zeta^i,\bar\zeta^{1-i})\\
    \widetilde L'&N&0&-p&\bar\zeta^{-1}\\
    \widetilde E'&N&0&-1&\bar\zeta^{p}
    \end{array}
    \]
    
    \subsubsection{Isolated fixed points}
    
    For each relevant base point, the toric resolution has the isolated fixed
    points
    \[
    Q_{i-1}\in
    \operatorname{Spec}
    k[t_1^it_2^{-1},t_1^{1-i}t_2],
    \qquad 2\leq i\leq p.
    \]
    In the local coordinates \(a_{i,1},a_{i,2}\), the action of \(\sigma\)
    is
    \[
    a_{i,1}\longmapsto\bar\zeta^{\,i}a_{i,1},
    \qquad
    a_{i,2}\longmapsto\bar\zeta^{\,1-i}a_{i,2}.
    \]
    Thus the lifted cotangent eigenvalues are
    \(\widehat\zeta^i\) and \(\widehat\zeta^{1-i}\), and
    Theorem~\ref{thm:lefschetz-rr} gives
    \begin{equation}\label{eq:split-brauer-local-point}
    a_{\mathrm{Br}}(Q_{i-1})
    =
    \frac{1}
    {(1-\widehat\zeta^i)(1-\widehat\zeta^{1-i})}.
    \end{equation}
    The contribution of the \(p-1\) points over one base point is therefore
    \[
    S_{\mathrm{iso}}
    =
    \sum_{i=2}^{p}
    \frac{1}
    {(1-\widehat\zeta^i)(1-\widehat\zeta^{1-i})}.
    \]
    For every summand,
    \[
    \frac{1}
    {(1-\widehat\zeta^i)(1-\widehat\zeta^{1-i})}
    =-
    \frac{1}{1-\widehat\zeta}
    \left(
    \frac{1}{1-\widehat\zeta^{i-1}}
    -
    \frac{1}{1-\widehat\zeta^i}
    \right).
    \]
    The sum now telescopes, giving
    \begin{equation}\label{eq:split-brauer-isolated-sum}
    \boxed{
    S_{\mathrm{iso}}
    =
    \frac{\widehat\zeta(\widehat\zeta^{p-1}-1)}
    {(1-\widehat\zeta)^2(1-\widehat\zeta^p)}.
    }
    \end{equation}
    
    \subsubsection{Fixed curves}
    
    There are \(N\) fixed curves \(\widetilde E'_P\) of exceptional type and
    \(N\) fixed curves \(\widetilde L'_L\) of line type.  Choose one
    representative of each type and abbreviate it to \(\widetilde E'\) and
    \(\widetilde L'\).  Both are isomorphic to \(\mathbb P^1\), and the
    intersection calculations above give
    \[
    (\widetilde E')^2=-1,
    \qquad
    (\widetilde L')^2=-p.
    \]
    Along \(\widetilde L'\), the normal eigenvalue is \(\bar\zeta^{-1}\).
    Consequently, \eqref{eq:brauer-curve} gives
    \begin{align}
    b_{\mathrm{Br}}(\widetilde L')
    &=
    \frac{1}{1-\widehat\zeta}
    -
    \frac{\widehat\zeta(-p)}{(1-\widehat\zeta)^2}
    \nonumber\\
    &=
    \boxed{\frac{1+(p-1)\widehat\zeta}
    {(1-\widehat\zeta)^2}}.
    \label{eq:split-brauer-L}
    \end{align}
    Along \(\widetilde E'\), the normal eigenvalue is
    \(\bar\zeta^{p}\).  Hence
    \begin{align}
    b_{\mathrm{Br}}(\widetilde E')
    &=
    \frac{1}{1-\widehat\zeta^{-p}}
    -
    \frac{\widehat\zeta^{-p}(-1)}{(1-\widehat\zeta^{-p})^2}
    \nonumber\\
    &=
    \boxed{
    \frac{\widehat\zeta^{2p}}{(1-\widehat\zeta^p)^2}
    }.
    \label{eq:split-brauer-E}
    \end{align}
    The first term in each line is essential: since the curves are rational,
    \(1-g_D=1\), not \(0\).
    The uniformity of these two normal eigenvalues was proved at the generic
    points of all boundary components in
    Lemma~\ref{lem:split-kummer-normal-character}.  In particular, it does
    not depend on choosing projective transformations that lift to
    \(\overline {\mathbf{Y}}_\ell(w)\).
    
    \subsubsection{Proof of the split main theorem}
    
    There are \(N(p+1)\) boundary crossing points, and the fixed locus over
    each one contributes the quantity in
    \eqref{eq:split-brauer-isolated-sum}.  There are also \(N\) curves of
    each boundary type.  Adding all fixed-curve contributions therefore
    gives
    \begin{align}
    L_{\mathrm{Br}}(\sigma,\widetilde{\overline {\mathbf{Y}}}_{\ell}(w))
    ={}&
    N(p+1)\left(
    \frac{\widehat\zeta(\widehat\zeta^{p-1}-1)}
    {(1-\widehat\zeta)^2(1-\widehat\zeta^p)}
    \right)\nonumber\\
    &+
    N\frac{1+(p-1)\widehat\zeta}
    {(1-\widehat\zeta)^2}
    +
    N\frac{\widehat\zeta^{2p}}{(1-\widehat\zeta^p)^2}.
    \label{eq:split-final-brauer}
    \end{align}
    This computes the right-hand side of
    \eqref{eq:intro-split-brauer} on the smooth resolution.  The passage to
    the normal component follows from
    \eqref{eq:common-brauer-resolution-invariance}.  Since \(\sigma\)
    preserves all \(p-1\) components and \(\tau\) identifies their
    Brauer characters, the whole-space value is \(p-1\) times the component
    value.  This proves the trace formulas in
    Theorem~\ref{thm:intro-split-brauer}.
    
    \subsection{Nonvanishing and Galois-module consequence}
    \label{sec:split-nonvanishing}
    
    \begin{proposition}
    \label{prop:split-main-value-nonzero}
    The split Brauer--Lefschetz numbers
    \[
    L_{\mathrm{Br}}(\sigma,\overline {\mathbf{Y}}_\ell(w))
    \quad\text{and}\quad
    L_{\mathrm{Br}}(\sigma,\overline {\mathbf{Y}}(w))
    \]
    are nonzero for every prime \(p\).
    \end{proposition}
    
    \begin{proof}
    Let \(R_p(T)\in\mathbb Q(T)\) denote the rational function on the
    right-hand side of \eqref{eq:intro-split-brauer}.  For a primitive
    \(N\)-th root \(\omega\), neither \(\omega\) nor \(\omega^p\) is one,
    because \(p\nmid N\).  Hence every denominator in \(R_p(\omega)\) is
    nonzero.  Lemma~\ref{lem:cyclotomic-comparison}, with
    \(M=N\) and \(\widehat\omega=\widehat\zeta\), shows that it is enough to
    prove nonvanishing after the complex specialization
    \(\omega=e^{2\pi i/N}\).
    
    Put \(A=\pi/N\) and \(x=e^{2iA}\).  The elementary identities
    \[
    \operatorname{Im}\frac{1}{(1-x^i)(1-x^{1-i})}
    =\frac14\bigl(\cot(iA)-\cot((i-1)A)\bigr)
    \]
    and
    \[
    \operatorname{Im}\frac{1+(p-1)x}{(1-x)^2}
    =\frac12\cot A,
    \qquad
    \operatorname{Im}\frac{x^{2p}}{(1-x^p)^2}
    =-\frac12\cot(pA)
    \]
    give
    \[
    \frac1N\operatorname{Im}
    L_{\mathrm{Br}}(\sigma,\overline {\mathbf{Y}}_\ell(w))
    =\frac{p-1}{4}\bigl(\cot(pA)-\cot A\bigr)<0.
    \]
    Indeed, \(0<A<pA<\pi/2\), and \(\cot\) is strictly decreasing there.
    Hence the value for the irreducible component, and hence also for the entire variety is, nonzero.
    \end{proof}
    
    \begin{corollary}[Split Galois-module consequence]
    \label{cor:split-galois-module-not-free}
    Put \(\Gamma_{\mathrm{sp}}=\langle\sigma\rangle\).  Then
    \[
    \chi_{\Gamma_{\mathrm{sp}}}
    \bigl(\overline {\mathbf{Y}}_\ell(w),\mathcal O_{\overline {\mathbf{Y}}_\ell(w)}\bigr)
    \]
    is not virtually free over \(\bar{k}[\Gamma_{\mathrm{sp}}]\).  Moreover,
    \(
    \chi_{ T^{wF}}
    \bigl(\overline {\mathbf{Y}}(w),\mathcal O_{\overline {\mathbf{Y}}(w)}\bigr)
    \)
    is not virtually free over \(\bar{k}[ T^{wF}]\).
    \end{corollary}
    
    % \begin{proof}
    % The element \(\sigma\) is nonidentity.  The two statements follow from
    % Proposition~\ref{prop:split-main-value-nonzero} and the obstruction
    % \eqref{eq:nonzero-lefschetz-free-obstruction}, first for the component
    % stabilizer and then for the full \( T^{wF}\)-action.
    % \end{proof}
    
    \section{Non-Split Case}\label{part:unitary}
    
    \subsection{Unitary case notation and setup} \label{unitary:notation}
    Let \( \mathbf{G}=\mathrm{GL}_3\) over
    \(\bar{k}\).  We equip \( \mathbf{G}\) with the \textit{unitary
    Frobenius} \(F\) defined by
    \begin{equation}
    F(g)=J(g^{(p)})^{-t}J^{-1},
    \qquad
    J=\begin{pmatrix}
    0&0&1\\
    0&-1&0\\
    1&0&0
    \end{pmatrix},
    \label{eq:unitary-frobenius-initial}
    \end{equation}
    where \(g^{(p)}\) is obtained by raising each entry to the \(p\)-th power
    and \((g^{(p)})^{-t}=((g^{(p)})^{-1})^t\).  The corresponding finite group of Lie type is
    \( \mathbf{G}^F=\mathrm U_3(\mathbb F_p)\).  The choice of the matrix \(J\) is explained in Section~\ref{unitary:choice of J}.
    
    Let \( B\) be the upper triangular Borel subgroup,
    \( U=R_u( B)\) its upper unitriangular radical, and
    \[
     T
    =\{\diag(t_1,t_2,t_3):t_i\in\bar{k}^\times\}
    \simeq(\mathbb G_m)^3
    \]
    the diagonal torus.  On \( T\),
    \[
    F\bigl(\diag(t_1,t_2,t_3)\bigr)
    =\diag(t_3^{-p},t_2^{-p},t_1^{-p}).
    \]
    The Weyl group \(W=N_{ \mathbf{G}}( T)/ T\) is \(S_3\).
    Writing \(\chi_i(\diag(t_1,t_2,t_3))=t_i\), take
    \[
    \alpha_1=\chi_1-\chi_2,
    \qquad
    \alpha_2=\chi_2-\chi_3.
    \]
    Then \(F(\alpha_1)=p\alpha_2\) and \(F(\alpha_2)=p\alpha_1\).  If
    \(e_i\) denotes the cocharacter in the \(i\)-th diagonal entry, the
    corresponding simple coroots are
    \[
    \alpha_1^\vee=e_1-e_2,
    \qquad
    \alpha_2^\vee=e_2-e_3,
    \]
    and
    \[
    F(e_1)=-pe_3,
    \qquad
    F(e_2)=-pe_2,
    \qquad
    F(e_3)=-pe_1.
    \]
    
    \subsubsection{The Frobenius and the order of the acting group}\label{sec:unitary-setup}
    
    Fix the Coxeter element \(w=s_1s_2=(123)\).  The action on the
    cocharacter lattice gives
    \[
    | T^{wF}|=|\det(I-wF\mid Y( T))|
    =(p+1)(p^2-1).
    \]
    The explicit decomposition of this finite abelian group and the
    stabilizer of each irreducible component are determined in
    Section~\ref{unitary:finite group TwF} and
    Proposition~\ref{prop:unitary-components-and-stabilizer}.
    
    % Writing the matrix $M_{wF}$ in the basis $\{\lambda_1, \lambda_2, \lambda_3\}$:
    % \[ M_{wF} = \begin{pmatrix} -p & 0 & 0 \\ 0 & 0 & -p \\ 0 & -p & 0 \end{pmatrix} \]
    % Evaluating the determinant:
    % \[ \det(I - M_{wF}) = \det \begin{pmatrix} 1+p & 0 & 0 \\ 0 & 1 & p \\ 0 & p & 1 \end{pmatrix} = (1+p)(1 - p^2) \]
    % Taking the absolute value yields $(p+1)(p^2-1)$.
    
 \subsection{Explicit description of the Deligne--Lusztig varieties}
\label{unitary:choice of J}

\subsubsection{The Hermitian form}

Let \(V=\mathbb F_{p^2}^3\), equipped with its standard basis, and put
\(V_{\bar k}=V\otimes_{\mathbb F_{p^2}}\bar k\). Vectors are written as
columns. For a matrix \(g\) with entries in \(\bar k\) we denote by
\(g^{(p)}\) the matrix obtained by raising each entry of \(g\) to the
\(p\)-th power, and we write \(g^{-t}=(g^{-1})^t\). Put
\[
    J_+=\antidiag(1,1,1),
    \qquad
    J_-=\antidiag(1,-1,1).
\]
For \(J\in\{J_+,J_-\}\) we have \(J=J^t=J^{-1}\), and we consider the
Hermitian form on \(V\) (extended to \(V_{\bar k}\) by the same formula)
and the isogeny of \(\mathrm{GL}_{3,\bar k}\) given by
\[
    \langle v,v'\rangle_J=v^{(p)t}Jv',
    \qquad
    F_J(g)=J\,(g^{(p)})^{-t}J^{-1}.
\]
G\"ortz and Schr\"oer~\cite{gortzDeligneLusztigVarietiesWhose2026}
describe the unitary Frobenius and its action on the flag variety using
\(J_+\); in this paper we work with \(J_-\) instead. For
\(z=(z_0,z_1,z_2)^t\in V_{\bar k}\) put
\[
    \mathcal H_\pm(z):=\langle z,z\rangle_{J_\pm}
    =z_0z_2^p\pm z_1^{p+1}+z_2z_0^p,
\]
so that \(\mathcal H_\pm=0\) is the Hermitian curve in
\(\mathbb P(V_{\bar k})\) attached to \(J_\pm\).

The two conventions are equivalent. Since the norm map
\(\mathbb F_{p^2}^\times\to\mathbb F_p^\times\) is surjective, there is
\(\varepsilon\in\mathbb F_{p^2}^\times\) with \(\varepsilon^{p+1}=-1\);
for \(p=2\) we take \(\varepsilon=1\), and then \(J_+=J_-\). With
\(D=\diag(1,\varepsilon,1)\) we have
\begin{equation}
    D^{(p)t}J_+D=J_-.
    \label{eq:J-plus-minus-isometry}
\end{equation}
Hence \(D\) is an isometry from
\((V,\langle\,\cdot\,,\cdot\,\rangle_{J_-})\) onto
\((V,\langle\,\cdot\,,\cdot\,\rangle_{J_+})\); in particular,
\(\mathcal H_+(Dz)=\mathcal H_-(z)\). Moreover,
\eqref{eq:J-plus-minus-isometry} implies
\(F_{J_+}(DgD^{-1})=D\,F_{J_-}(g)\,D^{-1}\), and therefore
\[
    (Dg)^{-1}F_{J_+}(Dg)
    =g^{-1}F_{J_-}(g)\cdot D^{-1}F_{J_+}(D)
    \in g^{-1}F_{J_-}(g)\,T,
\]
where \(T\subset\mathrm{GL}_{3,\bar k}\) denotes the diagonal torus.
Consequently, left multiplication by \(D\) identifies the
Deligne--Lusztig varieties attached to \(F_{J_-}\), together with their
compactifications, with those attached to \(F_{J_+}\). This allows us to
apply the results of~\cite{gortzDeligneLusztigVarietiesWhose2026} in our
setting.

Following~\cite[(7.2)]{gortzDeligneLusztigVarietiesWhose2026}
and~\cite{langerBirationalGeometryCompactifications2019}, we regard the
flag variety as an incidence subvariety
\[
    \operatorname{Flag}(V_{\bar k})
    \hookrightarrow
    \mathbb P^2_{\bar k}\times\mathbb P^2_{\bar k},
    \qquad
    (L\subset P)\longmapsto([x],[y]),
\]
where \(L=\bar kx\) and \(P=\{v\in V_{\bar k}:y^tv=0\}\) for nonzero
column vectors \(x=(x_0,x_1,x_2)^t\) and \(y=(y_0,y_1,y_2)^t\). Thus
\([x]\) is the point of \(\mathbb P(V_{\bar k})\) defined by \(L\), and
\([y]\) is the coefficient vector of a linear equation for \(P\); in
other words, we identify the dual projective plane with
\(\mathbb P^2_{\bar k}\) by means of the standard basis. The image is the
incidence divisor
\[
    x^ty=x_0y_0+x_1y_1+x_2y_2=0.
\]
The isometry \(D\) sends a flag \(L\subset P\) to \(DL\subset DP\).
Since \(DP\) is the kernel of the covector \(y^tD^{-1}\), in coordinates
this reads
\begin{equation}
    ([x],[y])\longmapsto([Dx],[D^{-t}y]).
    \label{eq:J-coordinate-change-flags}
\end{equation}
This map preserves the incidence equation and satisfies
\(\mathcal H_+(Dx)=\mathcal H_-(x)\) and
\(\mathcal H_+(D^{-t}y)=\mathcal H_-(y)\); thus it carries the Hermitian
equations for \(J_-\) to those for \(J_+\).

From now on we fix \(J=J_-\) and write
\(\langle\,\cdot\,,\cdot\,\rangle=\langle\,\cdot\,,\cdot\,\rangle_{J}\),
\(F=F_J\) and \(\mathcal H=\mathcal H_-\); explicitly,
\[
    F(g)=J(g^{(p)})^{-t}J,
    \qquad
    \mathcal H(z)=z_0z_2^p-z_1^{p+1}+z_2z_0^p.
\]
Let \(B\subset\mathrm{GL}_{3,\bar k}\) be the Borel subgroup of upper
triangular matrices and \(U\) its unipotent radical; the isogeny \(F\)
preserves \(B\), \(U\) and \(T\). Let \(w=s_1s_2\), which corresponds to
the \(3\)-cycle \((1\,2\,3)\in S_3\), and choose the representative
\[
    \dot w=
    \begin{pmatrix}
        0&0&1\\
        1&0&0\\
        0&1&0
    \end{pmatrix}.
\]
The Deligne--Lusztig variety and its covering variety are then
\[
    \mathbf{X}(w)=\{gB:g^{-1}F(g)\in B\dot wB\},
    \qquad
    \mathbf{Y}(w)=\{gU:g^{-1}F(g)\in U\dot wU\},
\]
and the covering map \(\mathbf{Y}(w)\to \mathbf{X}(w)\) is \(gU\mapsto gB\).

\subsubsection[Geometry and boundary of the compactification]
{The compactification \(\overline {\mathbf{X}}(w)\) and its boundary}
\label{sec:unitary-base-boundary}

For a subspace \(W\subseteq V_{\bar k}\) put
\(W^\perp=\{v'\in V_{\bar k}:\langle v,v'\rangle=0
\text{ for all }v\in W\}\).
By~\cite[Lemma~7.1]{gortzDeligneLusztigVarietiesWhose2026}, the
Frobenius acts on flags by
\[
    F(L\subset P)=(P^\perp\subset L^\perp).
\]
Two flags \((L\subset P)\) and \((L'\subset P')\) are in relative
position \(w=s_1s_2\) if and only if \(L\not\subset P'\) and
\(L'\subset P\)~\cite[\S7.3]{gortzDeligneLusztigVarietiesWhose2026}.
Hence
\(\mathbf{X}(w)=\{(L\subset P):P^\perp\subset P,\ L\not\subset L^\perp\}\), and
the smooth compactification attached to the reduced expression
\(s_1s_2\) is the closure
\begin{equation}
    \overline {\mathbf{X}}(w)=\{(L\subset P):P^\perp\subset P\}.
    \label{eq:unitary-compactification-flags}
\end{equation}
It is the union of the four Deligne--Lusztig strata \(\mathbf{X}(1)\),
\(\mathbf{X}(s_1)\), \(\mathbf{X}(s_2)\) and \(\mathbf{X}(w)\). Its boundary has two irreducible
components,
\begin{align}
    B_u&:=\overline {\mathbf{X}}(s_1)=\{(L\subset P):P=L^\perp\},
    \label{eq:unitary-Bu-flags}\\
    B_v&:=\overline {\mathbf{X}}(s_2)=\{(L\subset P):L=P^\perp\},
    \label{eq:unitary-Bv-flags}
\end{align}
which meet transversally along \(\mathbf{X}(1)=B_u\cap B_v\). The latter
consists of the \(p^3+1\) flags \(L\subset L^\perp\) with \(L\) an
\(\mathbb F_{p^2}\)-rational isotropic line. Thus
\begin{equation}
    \mathbf{X}(w)=\overline {\mathbf{X}}(w)\setminus(B_u\cup B_v).
    \label{eq:unitary-open-as-complement}
\end{equation}
These statements follow from the relative-position description
and~\cite[\S7.3, especially Proposition~7.2]{gortzDeligneLusztigVarietiesWhose2026};
the transversality can also be read off from the local computation
\eqref{eq:unitary-h-factorization} below.

We now pass to the coordinates \(([x],[y])\). Since
\(\langle x,v\rangle=(Jx^{(p)})^tv\), the plane \(L^\perp\) has
coefficient vector \(Jx^{(p)}\). As \(v'\) runs through \(P\), the
vector \(v'^{(p)}\) runs through the kernel of \(y^{(p)t}\); hence
\(P^\perp=\bar k\,Jy^{(p)}\). Since \(J\) is symmetric,
\(\mathcal H(y)=y^tJy^{(p)}\), so that \(P^\perp\subset P\) if and only
if \(\mathcal H(y)=0\). We
recover~\cite[Proposition~7.2]{gortzDeligneLusztigVarietiesWhose2026}:
\begin{equation}
    \overline {\mathbf{X}}(w)=
    \left\{([x],[y]):
    \begin{array}{l}
        x_0y_0+x_1y_1+x_2y_2=0,\\
        y_0y_2^p-y_1^{p+1}+y_2y_0^p=0
    \end{array}
    \right\},
    \label{eq:unitary-compactification-equations}
\end{equation}
and the boundary curves are
\begin{align}
    B_u&=\{([x],[y])\in\overline {\mathbf{X}}(w):[y]=[Jx^{(p)}]\},
    \label{eq:unitary-Bu-coordinates}\\
    B_v&=\{([x],[y])\in\overline {\mathbf{X}}(w):[x]=[Jy^{(p)}]\}.
    \label{eq:unitary-Bv-coordinates}
\end{align}
Moreover, \(L\subset L^\perp\) if and only if \(\mathcal H(x)=0\), so
\begin{equation}
    \mathbf{X}(w)=\overline {\mathbf{X}}(w)\cap\{\mathcal H(x)\neq0\},
    \qquad
    B_u\cup B_v=\overline {\mathbf{X}}(w)\cap\{\mathcal H(x)=0\}.
    \label{eq:unitary-boundary-hermitian}
\end{equation}

Let
\[
    C_x=\{\mathcal H(x)=0\}\subset\mathbb P^2_x,
    \qquad
    C_y=\{\mathcal H(y)=0\}\subset\mathbb P^2_y
\]
be the two copies of the Hermitian curve; they are smooth plane curves
of degree \(p+1\), hence of genus \(p(p-1)/2\). The second projection
exhibits \(\overline {\mathbf{X}}(w)\) as a \(\mathbb P^1\)-bundle over \(C_y\),
the fibre over \([y]\) being the line \(\{[x]:x^ty=0\}\). By
\eqref{eq:unitary-Bv-coordinates}, \(B_v\) is the graph of the morphism
\(C_y\to\mathbb P^2_x\), \([y]\mapsto[Jy^{(p)}]\), so the second
projection restricts to an isomorphism
\[
    B_v\xrightarrow{\ \sim\ }C_y,
\]
and \(B_v\) is a section of the ruling. By
\eqref{eq:unitary-Bu-coordinates} and the identity
\(x^tJx^{(p)}=\mathcal H(x)\), the first projection identifies \(B_u\)
with \(C_x\), and under this identification the morphism
\(B_u\to C_y\) factors as
\[
    C_x\xrightarrow{F_{C_x/k}}C_x^{(p)}
    \xrightarrow[\sim]{\ J\ }C_y,
\]
where \(F_{C_x/k}\) is the relative Frobenius and the second arrow is
the linear isomorphism induced by \(J\). Consequently, \(B_u\to C_y\) is
finite and purely inseparable of degree \(p\), whereas \(B_v\to C_y\) has degree
one: \(B_v\) is a section and \(B_u\) is a purely inseparable
\(p\)-section of the ruling. In particular, both boundary components are
Hermitian curves of genus \(p(p-1)/2\).

\subsubsection[The finite acting group]{The finite group \(T^{wF}\)}
\label{unitary:finite group TwF}

For \(t=\diag(t_1,t_2,t_3)\in T\) we have
\[
    F(t)=\diag(t_3^{-p},t_2^{-p},t_1^{-p}).
\]
Put \(F_w=\Ad(\dot w)\circ F\) and
\(T^{wF}=\{t\in T:F_w(t)=t\}\).
Since \(\dot w\diag(a_1,a_2,a_3)\dot w^{-1}=\diag(a_3,a_1,a_2)\), we get
\(F_w(t)=\diag(t_1^{-p},t_3^{-p},t_2^{-p})\). Hence \(t\in T^{wF}\) if
and only if \(t_1=t_1^{-p}\), \(t_2=t_3^{-p}\) and \(t_3=t_2^{-p}\),
that is,
\[
    t_1^{p+1}=1,
    \qquad
    t_2^{p^2-1}=1,
    \qquad
    t_3=t_2^{-p}.
\]
Therefore
\[
    T^{wF}
    =\left\{\diag(\zeta,a,a^{-p}):\zeta^{p+1}=1,\ a^{p^2-1}=1\right\}
    \cong\mu_{p+1}\times\mu_{p^2-1}.
\]
The group \(T^{wF}\) acts on \(\mathbf{Y}(w)\) from the right by
\(gU\cdot t=gtU\); this is well defined because \(t\) normalizes \(U\)
and \(t^{-1}\dot wF(t)=\dot w\). The action is free, and
\(\mathbf{Y}(w)\to \mathbf{X}(w)\) is a \(T^{wF}\)-torsor.

\subsubsection[Coordinates on \(\mathbf{Y}(w)\)]{Coordinates on \(\mathbf{Y}(w)\)}

We now describe the \(T^{wF}\)-torsor \(\mathbf{Y}(w)\to \mathbf{X}(w)\) explicitly. Let
\(g\in\mathrm{GL}_3(\bar k)\) have columns \(v_1,v_2,v_3\), and put
\[
    x:=v_1,
    \qquad
    y:=v_1\times v_2,
    \qquad
    \delta:=\det(g),
\]
where \(\times\) denotes the usual cross product on \(\bar k^3\). Right
multiplication by an element of \(U\) replaces \(v_2\) by
\(v_2+c_1v_1\) and \(v_3\) by \(v_3+c_2v_1+c_3v_2\) for some
\(c_1,c_2,c_3\in\bar k\), which changes none of \(x\), \(y\) and
\(\delta\). Hence the entries of \(x\) and \(y\), and \(\delta\), are
regular functions on \(\mathrm{GL}_{3,\bar k}/U\).

By Cramer's rule,
\(g^{-t}=\delta^{-1}(v_2\times v_3,\ v_3\times v_1,\ v_1\times v_2)\).
Applying this to \(g^{(p)}\), and noting that right multiplication by
\(J\) moves the third column to the first position, we find that the
first column of \(F(g)\) is
\begin{equation}
    F(g)_1=\delta^{-p}\,Jy^{(p)}.
    \label{eq:unitary-first-column}
\end{equation}

\begin{proposition}
\label{prop:unitary-Y-coordinates}
Let \(Z^\circ\subset\mathbb A^3\times\mathbb A^3\times\mathbb G_m\) be
the locally closed subscheme of triples \((x,y,\delta)\) with
\(y\neq0\) and
\begin{equation}
    x\times Jy^{(p)}=\delta^py,
    \qquad
    \mathcal H(x)=1,
    \qquad
    \delta^{p+1}=1.
    \label{eq:unitary-Y-equations}
\end{equation}
Then \(gU\mapsto(x,y,\delta)\) defines an isomorphism
\(\mathbf{Y}(w)\xrightarrow{\sim}Z^\circ\).
\end{proposition}

\begin{proof}
Put \(A=g^{-1}F(g)\). Multiplying out \(u\dot wu'\) with
\(u,u'\in U\) shows that
\[
    U\dot wU
    =\left\{
    A=\begin{pmatrix}
        *&*&*\\
        1&*&*\\
        0&1&*
    \end{pmatrix}
    :\det A=1
    \right\}.
\]
We translate these conditions on \(A\) into conditions on
\((x,y,\delta)\). First, \(\det F(g)=\det(g^{(p)})^{-1}=\delta^{-p}\),
so \(\det A=\delta^{-(p+1)}\), and \(\det A=1\) if and only if
\(\delta^{p+1}=1\). Second, \(A_{21}=1\) and \(A_{31}=0\) mean that
\(F(g)_1\in v_2+\bar kv_1\). Since \(v_1\times v_2\) and
\(v_1\times v_3\) are linearly independent, this is equivalent to
\(x\times F(g)_1=y\), that is, by \eqref{eq:unitary-first-column}, to
\(x\times Jy^{(p)}=\delta^py\). Third,
\(A^{-1}=F(g)^{-1}g=Jg^{(p)t}Jg\); since the last row of \(J\) is
\((1,0,0)\), the \((3,1)\)-entry of \(A^{-1}\) is
\(v_1^{(p)t}Jv_1=\mathcal H(x)\). On the other hand, the cofactor
formula gives \((A^{-1})_{31}=(A_{21}A_{32}-A_{22}A_{31})/\det A\).
Hence, once the first two conditions hold, \(A_{32}=\mathcal H(x)\),
and \(A_{32}=1\) if and only if \(\mathcal H(x)=1\).

It follows that \(gU\in \mathbf{Y}(w)\) if and only if \((x,y,\delta)\)
satisfies \eqref{eq:unitary-Y-equations}. Since \(v_1\) and \(v_2\) are
linearly independent, \(y\neq0\), so we obtain a morphism
\(\mathbf{Y}(w)\to Z^\circ\). Conversely, let \((x,y,\delta)\in Z^\circ\).
Taking the dot product of the first equation with \(x\) gives
\(\delta^p\,x^ty=0\), so \(x^ty=0\); moreover \(x\neq0\) because
\(\mathcal H(x)=1\). Hence there is a vector \(v_2\) with
\(x\times v_2=y\), unique modulo \(\bar kx\), and, as \(y\neq0\), a
vector \(v_3\) with \(y^tv_3=\delta\), unique modulo
\(\bar kx+\bar kv_2\). The matrix \(g=(x,v_2,v_3)\) has determinant
\(y^tv_3=\delta\neq0\), its class \(gU\) does not depend on the choices,
and \(gU\in \mathbf{Y}(w)\) by what we have just shown. Locally on \(Z^\circ\)
the choices can be made by regular functions: for instance,
\(v_2=x_0^{-1}(0,y_2,-y_1)^t\) on \(\{x_0\neq0\}\), and
\(v_3=(\delta/y_0,0,0)^t\) on \(\{y_0\neq0\}\), and similarly on the
other standard opens. This defines a morphism \(Z^\circ\to \mathbf{Y}(w)\)
inverse to the one above.
\end{proof}

Explicitly, since \(Jy^{(p)}=(y_2^p,-y_1^p,y_0^p)^t\), the equations
\eqref{eq:unitary-Y-equations} read
\begin{align*}
    x_1y_0^p+x_2y_1^p&=\delta^py_0,\\
    x_2y_2^p-x_0y_0^p&=\delta^py_1,\\
    -x_0y_1^p-x_1y_2^p&=\delta^py_2,\\
    x_0x_2^p-x_1^{p+1}+x_2x_0^p&=1,\\
    \delta^{p+1}&=1.
\end{align*}
% The open condition \(y\neq0\) cannot be dropped: these equations also
% have the extraneous solutions \((x,0,\delta)\) with \(\mathcal H(x)=1\)
% and \(\delta^{p+1}=1\).

\begin{lemma}[Coordinate form of the projection]
\label{lem:unitary-coordinate-projection}
Under the identification of
Proposition~\textup{\ref{prop:unitary-Y-coordinates}}, the projection
\(\mathbf{Y}(w)\to \mathbf{X}(w)\) is given by
\[
    (x,y,\delta)\longmapsto([x],[y]).
\]
\end{lemma}

\begin{proof}
Let \(g=(v_1,v_2,v_3)\). The point \(gB\) is the flag
\(L=\bar kv_1\subset P=\bar kv_1+\bar kv_2\), and \(L=\bar kx\). The
vector \(y=v_1\times v_2\) is nonzero and satisfies
\(y^tv_1=y^tv_2=0\). Hence \(P\subseteq\{v:y^tv=0\}\), with equality
because both spaces have dimension two. Thus the flag \(gB\) has
coordinates \(([x],[y])\).
\end{proof}

Since \(p\nmid p+1\), the group scheme \(\mu_{p+1}\) is finite \'etale
over \(\bar k\), with \(p+1\) points, all lying in
\(\mathbb F_{p^2}\). The function \(\delta\colon \mathbf{Y}(w)\to\mu_{p+1}\)
therefore yields a decomposition into open and closed subschemes
\begin{equation}
    \mathbf{Y}(w)=\coprod_{\delta_0\in\mu_{p+1}}\mathbf{Y}_{\delta_0}(w),
    \qquad
    \mathbf{Y}_{\delta_0}(w):=\{\delta=\delta_0\}.
    \label{eq:unitary-affine-component-decomposition}
\end{equation}
Proposition~\ref{prop:unitary-components-and-stabilizer} below shows
that each \(\mathbf{Y}_{\delta_0}(w)\) is irreducible and determines its
stabilizer in \(T^{wF}\).

\subsection{Local equations and normalization of the cover}
\label{sec:unitary-normalization}

Throughout this subsection we abbreviate \(\mathbf{X}=\mathbf{X}(w)\) and
\(\overline {\mathbf{X}}=\overline {\mathbf{X}}(w)\), and we put \(N_{\mathrm{un}}=p^2-1\).

\subsubsection[An affine chart]{An affine chart of \(\overline {\mathbf{X}}\)}

Let \(\overline {\mathbf{X}}_{0,2}=\overline {\mathbf{X}}\cap\{x_0\neq0,\ y_2\neq0\}\), with
affine coordinates \(t=x_1/x_0\), \(r=x_2/x_0\) and \(s=y_1/y_2\). The
incidence equation gives \(y_0/y_2=-(ts+r)\), so the points of
\(\overline {\mathbf{X}}_{0,2}\) are
\[
    ([x],[y])=\bigl([1:t:r],\,[-(ts+r):s:1]\bigr).
\]
Substituting into the Hermitian equation for \(y\) gives
\(\mathcal H(-(ts+r),s,1)=-\Phi\), where
\[
    \Phi=s^{p+1}+ts+t^ps^p+r+r^p.
\]
Hence
\begin{equation}
    \overline {\mathbf{X}}_{0,2}=\Spec\bar k[t,r,s]/(\Phi).
    \label{eq:unitary-chart}
\end{equation}
Since \(\partial\Phi/\partial r=1\), the projection
\(\overline {\mathbf{X}}_{0,2}\to\mathbb A^2_{t,s}\) is \'etale. Put
\[
    h:=\mathcal H(1,t,r)=r+r^p-t^{p+1},
\]
so that \(\mathcal H(x)=x_0^{p+1}h\) on \(\overline {\mathbf{X}}_{0,2}\). By
\eqref{eq:unitary-boundary-hermitian},
\[
    \mathbf{X}\cap\overline {\mathbf{X}}_{0,2}=\Spec R,
    \qquad
    R:=\bar k[t,r,s,h^{-1}]/(\Phi).
\]

\subsubsection[The boundary parameters]{The boundary parameters \(u\) and \(v\)}

Define functions on \(\overline {\mathbf{X}}_{0,2}\) by
\[
    u:=t^p+s,
    \qquad
    v:=-(t+s^p).
\]
Using \(\Phi=0\) to eliminate \(r+r^p\), we obtain
\begin{equation}
    h=-(s^{p+1}+ts+t^ps^p+t^{p+1})
    =-(t+s^p)(t^p+s)
    =uv.
    \label{eq:unitary-h-factorization}
\end{equation}
Comparing with \eqref{eq:unitary-Bu-coordinates} and
\eqref{eq:unitary-Bv-coordinates}, one checks that
\[
    B_u\cap\overline {\mathbf{X}}_{0,2}=V(u),
    \qquad
    B_v\cap\overline {\mathbf{X}}_{0,2}=V(v).
\]
For instance, a point of \(\overline {\mathbf{X}}_{0,2}\) lies on \(B_v\) if and
only if \([1:t:r]=[1:-s^p:-(ts+r)^p]\); the second coordinates give
\(v=0\), and the third coordinates then agree because \(\Phi=0\). The
argument for \(B_u\) is the same.

Let \(o=([1:0:0],[0:0:1])\). It is the point \(t=r=s=0\) of
\(\overline {\mathbf{X}}_{0,2}\), and it lies in \(\mathbf{X}(1)=B_u\cap B_v\). Since the
projection to \(\mathbb A^2_{t,s}\) is \'etale,
\(\widehat{\mathcal O}_{\overline {\mathbf{X}},o}=\bar k[[t,s]]\). The linear parts
of \(u\) and \(v\) are \(s\) and \(-t\), so \((u,v)\) is again a regular
system of parameters:
\begin{equation}
    \widehat{\mathcal O}_{\overline {\mathbf{X}},o}=\bar k[[u,v]].
    \label{eq:unitary-completed-local-ring}
\end{equation}
Thus, near \(o\), the boundary \(V(h)=B_u\cup B_v\) consists of the two
smooth branches \(V(u)=B_u\) and \(V(v)=B_v\), which meet
transversally. In particular, \(u\) is a uniformizer of the discrete
valuation ring \(\mathcal O_{\overline {\mathbf{X}},B_u}\) and \(v\) is a unit in
it, and symmetrically for \(B_v\).

\subsubsection{The covering over the chart}

Fix \(\delta_0\in\mu_{p+1}\). Put \(x'=(1,t,r)^t\) and
\(y'=(-(ts+r),s,1)^t\). By
Proposition~\ref{prop:unitary-Y-coordinates} and
Lemma~\ref{lem:unitary-coordinate-projection}, the points of
\(\mathbf{Y}_{\delta_0}(w)\) over \(\mathbf{X}\cap\overline {\mathbf{X}}_{0,2}\) are the triples
\((x_0x',\,y_2y',\,\delta_0)\) with \(x_0,y_2\in\mathbb G_m\)
satisfying \eqref{eq:unitary-Y-equations}. A direct computation, using
\(\Phi=0\) for the first two coordinates, gives
\[
    x'\times Jy'^{(p)}=v\,y'.
\]
Hence \eqref{eq:unitary-Y-equations} becomes
\[
    x_0^{p+1}h=1,
    \qquad
    x_0y_2^{p-1}v=\delta_0^p.
\]
On \(\mathbf{X}\) the function \(h=uv\) is a unit, hence so is \(v\). Solving the
second equation for \(x_0\), substituting into the first, and using
\(\delta_0^{p(p+1)}=1\), we obtain \(y_2^{N_{\mathrm{un}}}v^{p+1}=h=uv\),
that is,
\begin{equation}
    y_2^{N_{\mathrm{un}}}=u\,v^{-p}.
    \label{eq:unitary-local-cover}
\end{equation}
Conversely, every solution of \eqref{eq:unitary-local-cover} determines
\(x_0=\delta_0^py_2^{1-p}v^{-1}\). Therefore
\[
    \mathbf{Y}_{\delta_0}(w)\times_{\mathbf{X}}(\mathbf{X}\cap\overline {\mathbf{X}}_{0,2})
    \cong
    \Spec R[y_2]\big/\bigl(v^py_2^{N_{\mathrm{un}}}-u\bigr),
\]
an \'etale Kummer covering of degree \(N_{\mathrm{un}}\).

\subsubsection{Components and their stabilizer}

\begin{proposition}[Components and their stabilizer]
\label{prop:unitary-components-and-stabilizer}
Each subscheme \(\mathbf{Y}_{\delta_0}(w)\) in
\eqref{eq:unitary-affine-component-decomposition} is irreducible.
Consequently, \(\mathbf{Y}(w)\) has exactly \(p+1\) irreducible components.
Right multiplication by \(t=\diag(\zeta,a,a^{-p})\in T^{wF}\) acts on
the coordinates by
\begin{equation}
    (x,y,\delta)
    \longmapsto
    (\zeta x,\ \zeta a\,y,\ \zeta a^{1-p}\delta).
    \label{eq:unitary-torus-action-on-xy-delta}
\end{equation}
Hence the stabilizer of every component is the cyclic group
\begin{align}
    H
    &=\ker\bigl(\det\colon T^{wF}\longrightarrow\mu_{p+1}\bigr)
    \nonumber\\
    &=\left\{\diag(a^{p-1},a,a^{-p}):a^{p^2-1}=1\right\}
    \cong\mu_{N_{\mathrm{un}}},
    \label{eq:unitary-component-stabilizer}
\end{align}
and \(T^{wF}/H\cong\mu_{p+1}\) acts simply transitively on the set of
components.
\end{proposition}

\begin{proof}
Since \(\mathbf{Y}(w)\to \mathbf{X}\) is a \(T^{wF}\)-torsor, the morphism
\(\mathbf{Y}_{\delta_0}(w)\to \mathbf{X}\) is finite \'etale. As \(\mathbf{X}\) is normal and
irreducible, \(\mathbf{Y}_{\delta_0}(w)\) is normal and each of its irreducible
components dominates \(\mathbf{X}\); hence \(\mathbf{Y}_{\delta_0}(w)\) is irreducible
if its generic fibre is a field. By \eqref{eq:unitary-local-cover}, the
generic fibre is \(K[T]/(T^{N_{\mathrm{un}}}-uv^{-p})\), where
\(K=\bar k(\mathbf{X})\). Since \(u\) is a uniformizer of
\(\mathcal O_{\overline {\mathbf{X}},B_u}\) and \(v\) is a unit there, this
polynomial is Eisenstein, hence irreducible. This proves that
\(\mathbf{Y}_{\delta_0}(w)\) is irreducible. The same argument shows that the
covering is totally ramified along \(B_u\); it is also totally ramified
along \(B_v\), because \(\operatorname{ord}_{B_v}(uv^{-p})=-p\) is
prime to \(N_{\mathrm{un}}\).

Right multiplication by \(t\) multiplies the first two columns of \(g\)
by \(\zeta\) and \(a\), respectively. Hence \(x=v_1\) is multiplied by
\(\zeta\), \(y=v_1\times v_2\) by \(\zeta a\), and \(\delta=\det(g)\)
by \(\det(t)=\zeta a^{1-p}\); this proves
\eqref{eq:unitary-torus-action-on-xy-delta}. In particular, \(t\) maps
\(\mathbf{Y}_{\delta_0}(w)\) onto \(\mathbf{Y}_{\det(t)\delta_0}(w)\). The homomorphism
\(\det\colon T^{wF}\to\mu_{p+1}\) is surjective (take \(a=1\)), and its
kernel is obtained by setting \(\zeta=a^{p-1}\). This proves
\eqref{eq:unitary-component-stabilizer} and the last assertion.
\end{proof}

\subsubsection{The normalization at a boundary point}

Let \(\overline {\mathbf{Y}}(w)\) be the normalization of \(\overline {\mathbf{X}}\) in
\(\mathbf{Y}(w)\). By \eqref{eq:unitary-affine-component-decomposition},
\begin{equation}
    \overline {\mathbf{Y}}(w)
    =\coprod_{\delta_0\in\mu_{p+1}}\overline {\mathbf{Y}}_{\delta_0}(w),
    \label{eq:unitary-normal-component-decomposition}
\end{equation}
where \(\overline {\mathbf{Y}}_{\delta_0}(w)\) is the normalization of
\(\overline {\mathbf{X}}\) in \(\mathbf{Y}_{\delta_0}(w)\). We fix \(\delta_0\) from now on
and write \(\overline {\mathbf{Y}}^\circ(w):=\overline {\mathbf{Y}}_{\delta_0}(w)\).

In \(\widehat{\mathcal O}_{\overline {\mathbf{X}},o}=\bar k[[u,v]]\) the function
\(v\) is not a unit, so \eqref{eq:unitary-local-cover} does not exhibit
\(y_2\) as integral over \(\widehat{\mathcal O}_{\overline {\mathbf{X}},o}\). The
function \(z:=vy_2\), however, satisfies
\[
    z^{N_{\mathrm{un}}}=u\,v^{N_{\mathrm{un}}-p}=u\,v^{p^2-p-1}.
\]
The polynomial \(T^{N_{\mathrm{un}}}-uv^{p^2-p-1}\) is Eisenstein with
respect to the prime \(u\) of \(\bar k[[u,v]]\), hence irreducible.
Since \(\overline {\mathbf{X}}\) is excellent, normalization commutes with
completion. It follows that \(\overline {\mathbf{Y}}_{\delta_0}(w)\) has a unique
point \(o'\) over \(o\), and that
\begin{equation}
    \widehat{\mathcal O}_{\overline {\mathbf{Y}}_{\delta_0}(w),o'}
    \cong
    \bigl(A_{\mathrm{Kum}}\bigr)^{\mathrm{norm}},
    \qquad
    A_{\mathrm{Kum}}
    :=\bar k[[u,v,z]]\big/\bigl(z^{p^2-1}-u\,v^{p^2-p-1}\bigr).
    \label{eq:unitary-fixed-component-kummer}
\end{equation}

Let \(\alpha\in\mathbb F_{p^2}^\times\) be a generator, so that
\(\alpha\) has order \(N_{\mathrm{un}}\), and put
\[
    \tau=\diag(\alpha^{p-1},\alpha,\alpha^{-p})\in H.
\]
Then \(\tau\) generates \(H\). By
\eqref{eq:unitary-torus-action-on-xy-delta}, \(\tau\) fixes \(\delta\)
and multiplies \(y\) by \(\alpha^p\). Hence \(\tau\) fixes \(u\) and
\(v\), which are pulled back from \(\overline {\mathbf{X}}\), multiplies \(y_2\) by
\(\alpha^p\), and acts on the Kummer coordinate by
\[
    \tau(z)=\bar\xi z,
    \qquad
    \bar\xi:=\alpha^p.
\]
Since \(\gcd(p,N_{\mathrm{un}})=1\), \(\bar\xi\) is again a primitive
\(N_{\mathrm{un}}\)-th root of unity. This is the action of the
component stabilizer used in all toric and Lefschetz computations
below.

\subsection{The Bonnaf\'e--Rouquier slice in the non-split case}
\label{sec:br-slice-unitary}

We next evaluate the slice construction of
Section~\ref{sec:br-slice-singularities} for the unitary Frobenius. On
the cocharacter lattice
\(Y(T)=\mathbb Ze_1\oplus\mathbb Ze_2\oplus\mathbb Ze_3\) we have
\[
    F(e_1)=-pe_3,
    \qquad
    F(e_2)=-pe_2,
    \qquad
    F(e_3)=-pe_1,
\]
and \(w=s_1s_2\) acts by \(e_1\mapsto e_2\mapsto e_3\mapsto e_1\).
Hence \(wF=-ps_2\) on \(Y(T)\). With \(N_{\mathrm{un}}=p^2-1\) as
before, the solutions of \eqref{eq:br-cocharacter-equation} are
\begin{equation}
    \lambda_1=(p-1,1,-p),
    \qquad
    \lambda_2=(p-1,-p,1),
    \qquad
    m_1=m_2=N_{\mathrm{un}}.
    \label{eq:br-unitary-lambdas}
\end{equation}
Indeed,
\begin{align*}
    \lambda_1-wF(\lambda_1)
    &=\lambda_1+ps_2(\lambda_1)
    =N_{\mathrm{un}}(e_1-e_2)
    =N_{\mathrm{un}}\,\alpha_1^\vee,\\
    \lambda_2-wF(\lambda_2)
    &=\lambda_2+ps_2(\lambda_2)
    =N_{\mathrm{un}}(e_1-e_3)
    =N_{\mathrm{un}}\,s_1(\alpha_2^\vee).
\end{align*}
As in the split case, both cocharacters are primitive. For
\((z_1,z_2)\in\mathbb G_m^2\),
\[
    \lambda_1(z_1)\lambda_2(z_2)
    =\diag\bigl((z_1z_2)^{p-1},\ z_1z_2^{-p},\ z_1^{-p}z_2\bigr),
\]
and \eqref{eq:br-deep-stabilizer} becomes
\begin{equation}
    K_{12}
    =\left\{(\zeta,\zeta^p):\zeta^{N_{\mathrm{un}}}=1\right\}
    \cong\mu_{N_{\mathrm{un}}}.
    \label{eq:br-unitary-stabilizer}
\end{equation}
Thus \(\zeta\in\mu_{N_{\mathrm{un}}}\) acts on the two-dimensional slice
through \(\diag(\zeta,\zeta^p)\). Consequently, at every point lying
over the deepest boundary stratum \(\mathbf{X}(1)\), each irreducible component
of the normal unitary compactification has a cyclic quotient
singularity of type
\begin{equation}
    \frac{1}{p^2-1}(1,p).
    \label{eq:br-unitary-quotient-type}
\end{equation}
This agrees with the direct computation in
Section~\ref{sec:unitary-toric-interpretation}, where the normalization
of the local Kummer model is identified with the affine toric surface
of the cone \eqref{eq:unitary-cox-normal-form}, whose Cox normal form
has parameters \((p^2-1,p)\). Since \((p^2-1)/p=[[p,p]]\), both
descriptions lead to the same minimal resolution, a chain of two
\((-p)\)-curves.

\subsection{Toric normalization, resolution, and fixed locus}
\label{sec:unitary-toric-interpretation}

We use the conventions for normalizations and toric resolutions
collected in Section~\ref{sec:toric-surface-preliminaries}, including
its divisor and character-coordinate formulas.

\subsubsection{The toric normalization}

Since normalization commutes with completion for excellent rings,
\eqref{eq:unitary-fixed-component-kummer} shows that it suffices to
normalize the affine surface
\[
    Y_{\mathrm{Kum}}=\Spec R_{\mathrm{Kum}},
    \qquad
    R_{\mathrm{Kum}}
    =\bar k[u,v,z]\big/\bigl(z^{p^2-1}-uv^{p^2-p-1}\bigr),
\]
and to complete at the origin. The polynomial
\(z^{p^2-1}-uv^{p^2-p-1}\) has degree one in \(u\) and coprime
coefficients, so it is irreducible and \(R_{\mathrm{Kum}}\) is a domain.
On the dense torus \(\{vz\neq0\}\) the relation gives
\(u=z^{p^2-1}v^{-(p^2-p-1)}\); hence
\[
    R_{\mathrm{Kum}}
    \cong
    \bar k\bigl[v,\ z,\ z^{p^2-1}v^{-(p^2-p-1)}\bigr]
    \subset\bar k[v^{\pm1},z^{\pm1}].
\]
Let \(M\cong\mathbb Z^2\) be the character lattice of this torus, with
\(\chi^{(m_1,m_2)}=v^{m_1}z^{m_2}\), and let \(N\) be the dual lattice.
Then
\[
    v=\chi^{(1,0)},
    \qquad
    z=\chi^{(0,1)},
    \qquad
    u=\chi^{m_u},
    \quad
    m_u:=\bigl(-(p^2-p-1),\,p^2-1\bigr),
\]
and \(R_{\mathrm{Kum}}=\bar k[S]\) for the semigroup
\[
    S=\mathbb N(1,0)+\mathbb N(0,1)+\mathbb N m_u\subset M.
\]
Since
\[
    (0,1)=\frac{p^2-p-1}{p^2-1}(1,0)+\frac1{p^2-1}\,m_u,
\]
the vector \((0,1)\) lies in the interior of
\(\Cone\bigl((1,0),m_u\bigr)\), which is therefore the cone spanned by
\(S\). Its dual cone is
\[
    \sigma=\bigl\{(n_1,n_2)\in N_{\mathbb R}:
    n_1\geq0,\ -(p^2-p-1)n_1+(p^2-1)n_2\geq0\bigr\}
    =\Cone(u_0,u_3),
\]
where \(u_0=(0,1)\) and \(u_3=(p^2-1,p^2-p-1)\). By
\eqref{eq:toric-semigroup-normalization}, the normalization of
\(Y_{\mathrm{Kum}}\) is the affine toric surface
\[
    Y_{\mathrm{tor}}:=Y_{\mathrm{Kum}}^{\mathrm{norm}}
    =\Spec\bar k[\sigma^\vee\cap M].
\]

\subsubsection{The minimal toric resolution}

The vectors \(e_1=(1,1)\) and \(e_2=(0,1)\) form a basis of \(N\), and
\(u_3=(p^2-1)e_1-pe_2\). Hence
\begin{equation}
    \sigma=\Cone\bigl(e_2,\,(p^2-1)e_1-pe_2\bigr)
    \label{eq:unitary-cox-normal-form}
\end{equation}
is in Cox normal form with \(d=p^2-1\) and \(a=p\), where
\(\gcd(d,a)=1\). Thus \(Y_{\mathrm{tor}}\) is a cyclic quotient
singularity of type \(\frac1{p^2-1}(1,p)\), in accordance with
\eqref{eq:br-unitary-quotient-type}. By \eqref{eq:toric-hj-general},
the relevant Hirzebruch--Jung continued fraction is
\[
    \frac{p^2-1}{p}=p-\frac1p=[[p,p]],
\]
so \(r=2\) and \(b_1=b_2=p\). The recursion
\(P_i=b_iP_{i-1}-P_{i-2}\), \(Q_i=b_iQ_{i-1}-Q_{i-2}\) with
\((P_0,Q_0)=(1,0)\) and \((P_1,Q_1)=(b_1,1)\) gives
\[
    (P_0,P_1,P_2)=(1,\,p,\,p^2-1),
    \qquad
    (Q_0,Q_1,Q_2)=(0,\,1,\,p),
\]
and the rays \(u_i=P_{i-1}e_1-Q_{i-1}e_2\) for \(1\leq i\leq3\),
together with \(u_0=e_2\), are
\[
    u_0=(0,1),
    \qquad
    u_1=(1,1),
    \qquad
    u_2=(p,p-1),
    \qquad
    u_3=(p^2-1,p^2-p-1).
\]
They satisfy the wall relations \eqref{eq:toric-wall-relations},
\[
    u_0+u_2=pu_1,
    \qquad
    u_1+u_3=pu_2,
\]
they form the Hilbert basis of \(\sigma\cap N\), and
\(\det(u_{i-1},u_i)=-1\) for \(i=1,2,3\). By
\eqref{eq:toric-smoothness-determinant}, the fan \(\Sigma\) with maximal
cones
\[
    \sigma_1=\Cone(u_0,u_1),
    \qquad
    \sigma_2=\Cone(u_1,u_2),
    \qquad
    \sigma_3=\Cone(u_2,u_3)
\]
is smooth, and \(\widetilde Y_{\mathrm{tor}}:=X_\Sigma\to
Y_{\mathrm{tor}}\) is the minimal resolution. Its exceptional curves are
the divisors corresponding to the rays \(u_1\) and \(u_2\), and both
have self-intersection \(-p\).

\subsubsection{Pullbacks of the boundary parameters}

For \(0\leq i\leq3\), let \(D_i\) be the torus-invariant divisor of
\(\widetilde Y_{\mathrm{tor}}\) corresponding to \(u_i\). By
\eqref{eq:toric-character-divisor},
\(\divisor(\chi^m)=\sum_i\langle m,u_i\rangle D_i\). The pairings of
the exponents of \(v\), \(u\) and \(z\) with the rays are
\begin{center}
\renewcommand{\arraystretch}{1.2}
\begin{tabular}{c|cccc}
    & \(u_0\) & \(u_1\) & \(u_2\) & \(u_3\)\\
    \hline
    \(v=\chi^{(1,0)}\) & \(0\) & \(1\) & \(p\) & \(p^2-1\)\\
    \(u=\chi^{m_u}\) & \(p^2-1\) & \(p\) & \(1\) & \(0\)\\
    \(z=\chi^{(0,1)}\) & \(1\) & \(1\) & \(p-1\) & \(p^2-p-1\)
\end{tabular}
\end{center}
and therefore
\begin{equation}
    \begin{aligned}
        \divisor(v)&=D_1+pD_2+(p^2-1)D_3,\\
        \divisor(u)&=(p^2-1)D_0+pD_1+D_2,\\
        \divisor(z)&=D_0+D_1+(p-1)D_2+(p^2-p-1)D_3.
    \end{aligned}
    \label{eq:unitary-toric-divisors}
\end{equation}
In particular, the strict transforms of the reduced curves
\(\{u=0\}\) and \(\{v=0\}\) are
\[
    \widetilde{\{u=0\}}=D_0,
    \qquad
    \widetilde{\{v=0\}}=D_3,
\]
while \(D_1\) and \(D_2\) are the exceptional curves.

\subsubsection{The global resolution}

Let
\(
    \rho_{\delta_0}\colon
    \widetilde{\overline {\mathbf{Y}}}_{\delta_0}(w)
    \longrightarrow
    \overline {\mathbf{Y}}_{\delta_0}(w)
\)
be the minimal resolution of the chosen component. Away from \(\mathbf{X}(1)\),
the branch locus \(B_u\cup B_v\) is smooth and the covering is tamely
ramified, so \(\overline {\mathbf{Y}}_{\delta_0}(w)\) is smooth there. The finite
group \(\mathrm{GL}_3(\bar k)^F\cap\mathrm{SL}_3(\bar k)\) acts on
\(\mathbf{Y}(w)\) by left multiplication; this action commutes with that of
\(T^{wF}\), preserves \(\delta\), and is compatible with its action on
\(\overline {\mathbf{X}}\), which is transitive on \(\mathbf{X}(1)\). Hence every point of
\(\overline {\mathbf{Y}}_{\delta_0}(w)\) over \(\mathbf{X}(1)\) looks like \(o'\), and
\(\rho_{\delta_0}\) is, in completed local form, the toric resolution
\(\widetilde Y_{\mathrm{tor}}\to Y_{\mathrm{tor}}\) at each of these
\(p^3+1\) points. As in the split case, uniqueness of the minimal
resolution lifts the action of \(H\) to
\(\widetilde{\overline {\mathbf{Y}}}_{\delta_0}(w)\).

\medskip

We fix the following notation. Write \(I:=p^3+1\) and
\(\mathbf{X}(1)=\{o_1,\dots,o_I\}\). Let \(C_u\) and \(C_v\) be the strict
transforms in \(\widetilde{\overline {\mathbf{Y}}}_{\delta_0}(w)\) of the prime
divisors of \(\overline {\mathbf{Y}}_{\delta_0}(w)\) lying over \(B_u\) and
\(B_v\); these are unique by
Proposition~\ref{prop:unitary-components-and-stabilizer}. Over \(o_j\),
let \(E_{1,j}\) and \(E_{2,j}\) be the exceptional curves corresponding
to \(D_1\) and \(D_2\), and let \(Q_j=E_{1,j}\cap E_{2,j}\). Near the
fibre over \(o_j\), the curves \(C_u\) and \(C_v\) correspond to \(D_0\)
and \(D_3\), so that \(\widetilde{\overline {\mathbf{Y}}}_{\delta_0}(w)\) contains
the chains
\[
    C_u-E_{1,j}-E_{2,j}-C_v
    \qquad(1\leq j\leq I)
\]
of transversally intersecting curves.

\subsubsection{Affine charts and the fixed locus}

By Section~\ref{sec:unitary-normalization}, \(H=\langle\tau\rangle\) acts on the local
model by
\[
    \tau(z)=\bar\xi z,
    \qquad
    \tau(u)=u,
    \qquad
    \tau(v)=v,
\]
so that \(\tau(\chi^{(m_1,m_2)})=\bar\xi^{m_2}\chi^{(m_1,m_2)}\): the
character \(\chi^{(m_1,m_2)}\) has weight \(m_2\) modulo \({p^2-1}\), in
accordance with the character-coordinate rule of
Section~\ref{sec:toric-surface-preliminaries}. Each chart
\(U_{\sigma_i}\) is an affine plane whose coordinates are the
characters dual to \(u_{i-1},u_i\). In terms of \(u\), \(v\) and \(z\),
these coordinates, the divisors they cut out, and their weights are as
follows:
\begin{center}
\renewcommand{\arraystretch}{1.3}
\begin{tabular}{cccc}
    chart & coordinates & zero loci & weights\\
    \hline
    \(U_{\sigma_1}\) & \(z/v,\ v\) & \(D_0,\ D_1\) & \(1,\ 0\)\\
    \(U_{\sigma_2}\) & \(z^pv^{1-p},\ v/z\) & \(D_1,\ D_2\) & \(p,\ -1\)\\
    \(U_{\sigma_3}\) & \(u,\ v^{p-1}z^{-p}\) & \(D_2,\ D_3\) & \(0,\ -p\)
\end{tabular}
\end{center}
Since \(\gcd(p,p^2-1)=1\), the weights \(\pm1\), \(p\) and \(-p\) are
nonzero modulo \(p^2-1\). Hence the fixed locus of \(H\) is
\(D_0\cap U_{\sigma_1}\) in the first chart, the origin
\(D_1\cap D_2\) in the second, and \(D_3\cap U_{\sigma_3}\) in the
third. In particular, \(D_1\) and \(D_2\) are not fixed pointwise, and
\(D_1\cap D_2\) is the only fixed point on \(D_1\cup D_2\) that does
not lie on \(D_0\cup D_3\).

Globally, \(H\) acts freely over \(\mathbf{X}\), and it fixes \(C_u\) and \(C_v\)
pointwise, because the covering is totally ramified along \(B_u\) and
\(B_v\), so that the inertia groups of \(C_u\) and \(C_v\) are all of
\(H\). Together with the chart computation this gives
\begin{equation}
    \bigl(\widetilde{\overline {\mathbf{Y}}}_{\delta_0}(w)\bigr)^H
    =C_u\sqcup C_v\sqcup\{Q_1,\dots,Q_{p^3+1}\}.
    \label{eq:unitary-fixed-locus}
\end{equation}

\subsection{Boundary intersection calculations}
\label{sec:unitary-intersections}

As before, \(B_u\) and \(B_v\) denote the boundary curves of the base
\(\overline {\mathbf{X}}=\overline {\mathbf{X}}(w)\), and \(C_u\) and \(C_v\) their strict
transforms on the resolved component
\(\widetilde{\overline {\mathbf{Y}}}_{\delta_0}(w)\).

\subsubsection{Boundary curves on the base compactification}

Let
\[
    A=\mathcal O_{\mathbb P^2_x}(1)|_{\overline {\mathbf{X}}},
    \qquad
    B=\mathcal O_{\mathbb P^2_y}(1)|_{\overline {\mathbf{X}}}.
\]
By \eqref{eq:unitary-compactification-equations}, the surface
\(\overline {\mathbf{X}}\) is a complete intersection of bidegrees \((1,1)\) and
\((0,p+1)\) in \(\mathbb P^2_x\times\mathbb P^2_y\). Computing in the
Chow ring of \(\mathbb P^2_x\times\mathbb P^2_y\), we obtain
\begin{equation}
    A^2=A\cdot B=p+1,
    \qquad
    B^2=0.
    \label{eq:unitary-base-hyperplane-intersections}
\end{equation}
By Section~\ref{sec:unitary-base-boundary}, the first projection maps
\(B_u\) isomorphically onto \(C_x\) and the second maps it to \(C_y\)
with degree \(p\), while the second projection maps \(B_v\)
isomorphically onto \(C_y\) and the first maps it to \(C_x\) with degree
\(p\). Hence
\[
    A\cdot B_u=p+1,
    \quad
    B\cdot B_u=p(p+1),
    \qquad
    A\cdot B_v=p(p+1),
    \quad
    B\cdot B_v=p+1.
\]
Since \(\overline {\mathbf{X}}\) is a \(\mathbb P^1\)-bundle over \(C_y\), the numerical divisor class
group \(\operatorname{Num}(\overline {\mathbf{X}})\) has rank two, and since the
intersection matrix of \(A\) and \(B\) is non-singular, \(A\) and \(B\)
form a basis of \(\operatorname{Num}(\overline {\mathbf{X}})_{\mathbb Q}\). Solving
for the classes of the boundary curves yields
\begin{equation}
    B_u\equiv pA-(p-1)B,
    \qquad
    B_v\equiv A+(p-1)B,
    \label{eq:unitary-base-boundary-classes}
\end{equation}
and consequently
\begin{equation}
    \begin{aligned}
        B_u^2&=p(p+1)(2-p),\\
        B_v^2&=(p+1)(2p-1),\\
        B_u\cdot B_v&=p^3+1.
    \end{aligned}
    \label{eq:unitary-base-boundary-intersections}
\end{equation}
The last equality agrees with the fact that \(B_u\) and \(B_v\) meet
transversally in the \(p^3+1\) points of \(\mathbf{X}(1)\).

\subsubsection{Strict transforms on the resolved cover}

Let
\[
    \widetilde\pi_{\delta_0}\colon
    \widetilde{\overline {\mathbf{Y}}}_{\delta_0}(w)
    \xrightarrow{\ \rho_{\delta_0}\ }
    \overline {\mathbf{Y}}_{\delta_0}(w)
    \longrightarrow
    \overline {\mathbf{X}}
\]
be the composite of the minimal resolution with the normal covering,
which has degree \(N_{\mathrm{un}}=p^2-1\). Since the covering is totally
ramified along \(B_u\) and \(B_v\)
(Proposition~\ref{prop:unitary-components-and-stabilizer}), the divisors
\eqref{eq:unitary-toric-divisors} of the local equations \(u\) and \(v\)
give the global pullback formulas
\begin{align}
    \widetilde\pi_{\delta_0}^{*}B_u
    &=N_{\mathrm{un}}C_u
      +p\sum_{j=1}^{I}E_{1,j}
      +\sum_{j=1}^{I}E_{2,j},
    \label{eq:unitary-pullback-Bu}\\
    \widetilde\pi_{\delta_0}^{*}B_v
    &=\sum_{j=1}^{I}E_{1,j}
      +p\sum_{j=1}^{I}E_{2,j}
      +N_{\mathrm{un}}C_v.
    \label{eq:unitary-pullback-Bv}
\end{align}
Here \(C_u\) meets each \(E_{1,j}\) transversally in one point and is
disjoint from the \(E_{2,j}\) and from \(C_v\); symmetrically for
\(C_v\).

Total ramification also implies that \(C_u\to B_u\) and \(C_v\to B_v\)
have degree one. By the projection formula,
\(\widetilde\pi_{\delta_0}^{*}B_u\cdot C_u
=B_u\cdot\widetilde\pi_{\delta_0*}C_u=B_u^2\), while
\eqref{eq:unitary-pullback-Bu} gives
\(\widetilde\pi_{\delta_0}^{*}B_u\cdot C_u=N_{\mathrm{un}}C_u^2+pI\).
Arguing in the same way for \(B_v\), we obtain
\[
    B_u^2=N_{\mathrm{un}}C_u^2+pI,
    \qquad
    B_v^2=N_{\mathrm{un}}C_v^2+pI.
\]
Substituting \eqref{eq:unitary-base-boundary-intersections} and
\(I=p^3+1\) yields
\begin{equation}
    C_u^2=-p(p+1),
    \qquad
    C_v^2=-(p^2-1).
    \label{eq:unitary-resolved-fixed-self-intersections}
\end{equation}
Finally, \(C_u\to B_u\) and \(C_v\to B_v\) are finite birational
morphisms onto smooth curves, hence isomorphisms. Therefore
\begin{equation}
    g(C_u)=g(C_v)=\frac{p(p-1)}2.
    \label{eq:unitary-resolved-fixed-genera}
\end{equation}
    
    We conclude with the following figure:
    
    \begin{figure}[ht]
        \centering

    \tikzset{every picture/.style={line width=0.75pt}} %set default line width to 0.75pt        
    
    \begin{tikzpicture}[x=0.75pt,y=0.75pt,yscale=-1,xscale=1]
    
    %Curve Lines [id:da030171443257922426] 
    \draw [color={rgb, 255:red, 208; green, 2; blue, 27 }  ,draw opacity=1 ]   (17.9,106.05) .. controls (31.15,129.03) and (314.15,52.6) .. (335.22,69.58) ;
    %Curve Lines [id:da31144763834058997] 
    \draw    (93.53,57.67) .. controls (106.56,73.56) and (187.32,127.63) .. (190.83,128.32) ;
    %Curve Lines [id:da2802605627780218] 
    \draw [color={rgb, 255:red, 0; green, 0; blue, 0 }  ,draw opacity=1 ]   (183.63,108.04) .. controls (182.13,114.6) and (95.82,187.62) .. (96.42,201.11) ;
    %Shape: Free Drawing [id:dp40318410857968856] 
    \draw  [color={rgb, 255:red, 0; green, 0; blue, 0 }  ,draw opacity=1 ][line width=3] [line join = round][line cap = round] (217.06,118.75) .. controls (217.06,118.73) and (217.06,118.71) .. (217.06,118.69) ;
    %Shape: Free Drawing [id:dp8367839133145111] 
    \draw  [color={rgb, 255:red, 0; green, 0; blue, 0 }  ,draw opacity=1 ][line width=3] [line join = round][line cap = round] (225.57,118.94) .. controls (225.57,118.92) and (225.57,118.9) .. (225.58,118.88) ;
    %Shape: Free Drawing [id:dp5944790594226712] 
    \draw  [color={rgb, 255:red, 0; green, 0; blue, 0 }  ,draw opacity=1 ][line width=3] [line join = round][line cap = round] (234.59,119.08) .. controls (234.59,119.06) and (234.59,119.04) .. (234.59,119.01) ;
    %Shape: Brace [id:dp20746905221998546] 
    \draw   (274.65,63.02) .. controls (273.83,58.43) and (271.12,56.54) .. (266.53,57.36) -- (218.63,65.91) .. controls (212.07,67.08) and (208.38,65.37) .. (207.56,60.78) .. controls (208.38,65.37) and (205.51,68.25) .. (198.95,69.42)(201.9,68.9) -- (148.32,78.46) .. controls (143.73,79.28) and (141.84,81.99) .. (142.66,86.58) ;
    %Curve Lines [id:da3664970103882268] 
    \draw    (265.98,40.54) .. controls (279,56.43) and (340.4,135.39) .. (337.27,122.97) ;
    %Curve Lines [id:da7723451401153737] 
    \draw [color={rgb, 255:red, 0; green, 0; blue, 0 }  ,draw opacity=1 ]   (326.05,98.7) .. controls (324.55,105.26) and (241.89,192.55) .. (256.34,208.6) ;
    %Curve Lines [id:da9551685447222913] 
    \draw [color={rgb, 255:red, 208; green, 2; blue, 27 }  ,draw opacity=1 ] [dash pattern={on 3.75pt off 3pt on 7.5pt off 1.5pt}]  (316.07,170.07) .. controls (278.26,186.79) and (46.32,160.43) .. (25.13,176.49) ;
    %Straight Lines [id:da6705756695503318] 
    \draw    (128.32,218.8) -- (128.32,270.25) ;
    \draw [shift={(128.32,272.25)}, rotate = 270] [color={rgb, 255:red, 0; green, 0; blue, 0 }  ][line width=0.75]    (10.93,-3.29) .. controls (6.95,-1.4) and (3.31,-0.3) .. (0,0) .. controls (3.31,0.3) and (6.95,1.4) .. (10.93,3.29)   ;
    %Shape: Free Drawing [id:dp44529755529313464] 
    \draw  [color={rgb, 255:red, 208; green, 2; blue, 27 }  ,draw opacity=1 ][line width=3] [line join = round][line cap = round] (175.29,116.96) .. controls (174.6,116.96) and (173.91,116.96) .. (173.22,116.96) .. controls (172.68,116.96) and (173.74,117.71) .. (173.74,118.12) .. controls (173.74,118.4) and (173.61,117.35) .. (173.22,117.35) ;
    %Shape: Free Drawing [id:dp7651663995367923] 
    \draw  [color={rgb, 255:red, 208; green, 2; blue, 27 }  ,draw opacity=1 ][line width=3] [line join = round][line cap = round] (321.26,106.05) .. controls (320.57,106.05) and (319.88,106.05) .. (319.2,106.05) .. controls (318.65,106.05) and (319.71,106.79) .. (319.71,107.2) .. controls (319.71,107.48) and (319.58,106.43) .. (319.2,106.43) ;
    
    % Text Node
    \draw (41.08,196.21) node [anchor=north west][inner sep=0.75pt]  [font=\footnotesize,color={rgb, 255:red, 208; green, 2; blue, 27 }  ,opacity=1 ,rotate=-358.09]  {$C_v$};
    % Text Node
    \draw (47.42,84.55) node [anchor=north west][inner sep=0.75pt]  [font=\footnotesize,color={rgb, 255:red, 208; green, 2; blue, 27 }  ,opacity=1 ]  {$C_u$};
    % Text Node
    \draw (97.67,108.86) node [anchor=north west][inner sep=0.75pt]  [font=\footnotesize]  {$E_{1} \ =\ E_{1,1}$};
    % Text Node
    \draw (153.15,45.89) node [anchor=north west][inner sep=0.75pt]  [font=\footnotesize,rotate=-356.1]  {\(p^{3}+1\) branches};
    % Text Node
    \draw (253.87,81.28) node [anchor=north west][inner sep=0.75pt]  [font=\footnotesize]  {$E_{1,\ p^{3}+1}$};
    % Text Node
    \draw (109.7,185.74) node [anchor=north west][inner sep=0.75pt]  [font=\footnotesize]  {$E_{2} \ =\ E_{2,1}$};
    % Text Node
    \draw (22.14,147.22) node [anchor=north west][inner sep=0.75pt]  [font=\footnotesize]  {$\ E_{2,\ p^{3}+1}$};
    % Text Node
    \draw (1.19,37.04) node [anchor=north west][inner sep=0.75pt]  [font=\normalsize,color={rgb, 255:red, 0; green, 0; blue, 0 }  ,opacity=1 ]  {$\widetilde{\overline {\mathbf{Y}}}_{\delta_0}$};

    \end{tikzpicture}

       \begin{tikzpicture}[
        scale=1.1,
        line cap=round,
        line join=round,
        boundary/.style={draw=black, line width=0.9pt},
        horizontaldiv/.style={draw=black!65, line width=0.55pt},
        verticaldiv/.style={
            draw=black!65,
            line width=0.55pt,
            dash pattern=on 5pt off 4pt on 10pt off 2pt
        },
        labelstyle/.style={font=\footnotesize}
    ]
    
    % ------------------------------------------------------------
    % Boundary of the ruled surface
    % ------------------------------------------------------------
    
    \def\surfacepath{
      (-3.0,1.4) .. controls (-1.7,1.8) and (-0.6,1.2) .. (0.6,1.45)
                 .. controls (1.7,1.7) and (2.5,1.35) .. (3.0,1.55)
      --
      (3.0,-1.05) .. controls (2.45,-1.35) and (1.8,-0.85) .. (0.7,-1.20)
                   .. controls (-0.6,-1.55) and (-1.8,-0.9) .. (-3.0,-1.3)
      -- cycle
    }
    
    % ------------------------------------------------------------
    % One single connected horizontal divisor
    % ------------------------------------------------------------
    
    \draw[horizontaldiv]
      % top horizontal curve, left to right
      (-2.85,0.70)
        .. controls (-1.5,0.95) and (-0.4,0.45) .. (0.8,0.72)
        .. controls (1.8,0.95) and (2.35,0.65) .. (2.9,0.82)
    
      % wrap around right side to middle horizontal curve
        .. controls (3.35,0.63) and (3.35,0.30) .. (2.9,0.12)
    
      % middle horizontal curve, right to left
        .. controls (2.35,-0.05) and (1.85,0.15) .. (0.75,-0.05)
        .. controls (-0.55,-0.28) and (-1.55,0.20) .. (-2.85,-0.10)
    
      % wrap around left side to bottom horizontal curve
        .. controls (-3.35,-0.28) and (-3.35,-0.60) .. (-2.85,-0.80)
    
      % bottom horizontal curve, left to right
        .. controls (-1.65,-0.50) and (-0.55,-1.00) .. (0.75,-0.78)
        .. controls (1.85,-0.55) and (2.30,-0.85) .. (2.9,-0.65);
    
    % ------------------------------------------------------------
    % One single connected vertical divisor, dashed 5-4-10-2
    % ------------------------------------------------------------
    
    \draw[verticaldiv]
      % first vertical curve, top to bottom
      (-2.30,1.55)
        .. controls (-2.15,0.70) and (-2.25,-0.20) .. (-2.35,-1.05)
    
      % wrap around bottom to second vertical curve
        .. controls (-2.12,-1.42) and (-1.78,-1.42) .. (-1.60,-1.15)
    
      % second vertical curve, bottom to top
        .. controls (-1.50,-0.20) and (-1.42,0.65) .. (-1.55,1.55)
    
      % wrap around top to third vertical curve
        .. controls (-1.25,1.86) and (-0.98,1.72) .. (-0.80,1.35)
    
      % third vertical curve, top to bottom
        .. controls (-0.67,0.45) and (-0.83,-0.55) .. (-0.75,-1.38)
    
      % wrap around bottom to fourth vertical curve
        .. controls (-0.45,-1.64) and (-0.15,-1.62) .. (0.05,-1.35)
    
      % fourth vertical curve, bottom to top
        .. controls (-0.05,-0.45) and (0.12,0.55) .. (0.00,1.35)
    
      % wrap around top to fifth vertical curve
        .. controls (0.25,1.62) and (0.52,1.72) .. (0.75,1.48)
    
      % fifth vertical curve, top to bottom
        .. controls (0.67,0.48) and (0.72,-0.35) .. (0.80,-1.15)
    
      % wrap around bottom to sixth vertical curve
        .. controls (1.05,-1.38) and (1.35,-1.30) .. (1.55,-0.95)
    
      % sixth vertical curve, bottom to top
        .. controls (1.45,-0.05) and (1.62,0.75) .. (1.55,1.63)
    
      % wrap around top to seventh vertical curve
        .. controls (1.78,1.83) and (2.05,1.75) .. (2.30,1.45)
    
      % seventh vertical curve, top to bottom
        .. controls (2.42,0.50) and (2.18,-0.35) .. (2.25,-1.25);
    
    % ------------------------------------------------------------
    % Draw boundary last
    % ------------------------------------------------------------
    
    \draw[boundary] \surfacepath;
    
    % ------------------------------------------------------------
    % Projection arrow to the base curve and the base curve
    % ------------------------------------------------------------
    
    % Arrow from the ruled surface down to the base curve
    \draw[-{Stealth[length=6pt,width=6pt]}, line width=0.7pt]
      (0,-1.85) -- (0,-2.75)
      node[midway,right] {$\pi_2$};
    
    % Base curve downstairs
    \draw[black, line width=0.8pt]
      (-2.4,-3.05)
        .. controls (-1.4,-2.75) and (-0.4,-3.35) .. (0.55,-3.05)
        .. controls (1.35,-2.78) and (2.05,-3.20) .. (2.65,-2.95);
    
    % Label for the base curve
    \node[labelstyle] at (0,-3.45) {$\overline{\mathbf{X}}(s_2)$};
    
    % ------------------------------------------------------------
    % Labels
    % ------------------------------------------------------------
    
    \node[labelstyle] at (-2.75,2.25) {$\overline{\mathbf{X}}(s_1 s_2)$};
    
    \node[labelstyle] at (1.35,1.95) {$B_u = \overline{\mathbf{X}}(s_1)$};
    
    \node[labelstyle, anchor=west] at (3.15,0.35) {$B_v = \overline{\mathbf{X}}(s_2)$};
    \end{tikzpicture}
    \caption{The surface \( \overline{\mathbf{X}}(s_1 s_2)\) ruled over \( \overline{\mathbf{X}}(s_2)\)}
    %\smallskip
    %{\footnotesize The surface \( \overline{\mathbf{X}}(s_1 s_2)\) ruled over \( \overline{\mathbf{X}}(s_2)\)}
    %\end{minipage}
    %\hfill
    %\begin{minipage}{0.48\textwidth}
    \end{figure}
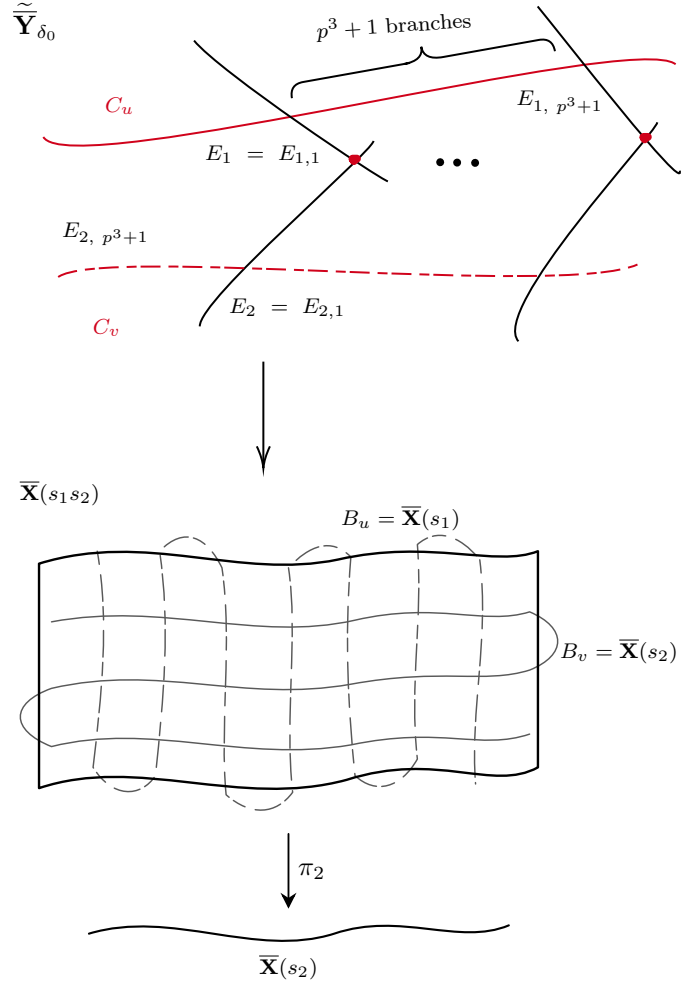
    
    % \newpage

    \subsection{Lefschetz number calculations}
    \label{sec:unitary-brauer-calculation}

    We employ the same methods discussed in Section~\ref{split:minres}.
    Put \(N_{\mathrm{un}}=p^2-1\), and let \(\tau\) generate
    \(
    H=\ker(\det\colon T^{wF}\to\mu_{p+1})
    \simeq\mu_{N_{\mathrm{un}}}.
    \)
    On the local cyclic cover,
    \[
    \tau(z)=\bar\xi z,
    \qquad
    \operatorname{ord}(\bar\xi)=N_{\mathrm{un}},
    \qquad
    \widehat\xi=[\bar\xi]\in W(\mathbb F_{p^2})^\times.
    \]
    The factors \(1-\widehat\xi\) and
    \(1-\widehat\xi^p\) are units.  Set
    \begin{equation}\label{eq:unitary-integral-coefficients}
    c_0:=\frac{(p+1)(p-2)}2,
    \qquad
    c_u:=\frac{(p+1)(3p-2)}2,
    \qquad
    c_v:=\frac{(p+1)(3p-4)}2.
    \end{equation}
    All three coefficients are integers.
    
    \subsubsection{Fixed-locus data}
    
    The table separates the fixed curves on the resolved cover from their
    images on the base.  The last column gives cotangent eigenvalues at a point
    and the normal eigenvalue along a curve.
    \[
    \begin{array}{c|c|c|c|c}
    \text{fixed stratum}&\text{number}&\text{genus}&\text{self-intersection}
    &\text{eigenvalue(s)}\\ \hline
    Q& p^3+1&-&-&(\bar\xi^{-1},\bar\xi^p)\\
    C_u&1&\frac{p(p-1)}2&-p(p+1)&\bar\xi^{-1}\\
    C_v&1&\frac{p(p-1)}2&-(p^2-1)&\bar\xi^{p}
    \end{array}
    \]
    
    \subsubsection{Isolated fixed points}
    
    In the middle toric chart
    \[
    U_{\sigma_2}=\operatorname{Spec}\bar{k}[y_1,y_2],
    \qquad
    y_1=t_1t_2^{-1},
    \quad
    y_2=t_1^{-p+1}t_2^p,
    \]
    the lifted cotangent eigenvalues at the isolated fixed point \(Q\) are
    \(\widehat\xi^{-1}\) and \(\widehat\xi^p\).  Hence
    \begin{equation}
    a_{\mathrm{Br}}(Q)
    =-\frac{\widehat\xi}
    {(1-\widehat\xi)(1-\widehat\xi^p)}.
    \label{eq:unitary-brauer-point}
    \end{equation}
    There are \(p^3+1\) such points.
    
    \subsubsection{Fixed curves}
    
    By \eqref{eq:unitary-resolved-fixed-self-intersections} and
    \eqref{eq:unitary-resolved-fixed-genera},
    \[
    1-g(C_u)=1-g(C_v)=-c_0.
    \]
    
    \subsubsection*[The fixed curve Cu]{The curve \(C_u\)}
    
    In \(U_{\sigma_1}=\operatorname{Spec}\bar{k}[x_1,x_2]\), the curve \(C_u\)
    is given by \(x_2=0\), and the normal eigenvalue is \(\bar\xi^{-1}\).
    Formula \eqref{eq:brauer-curve}, together with
    \(C_u^2=-p(p+1)\), gives
    \begin{align}
    b_{\mathrm{Br}}(C_u)
    &=-\frac{c_0}{1-\widehat\xi}
    +\frac{p(p+1)\widehat\xi}{(1-\widehat\xi)^2}\nonumber\\
    &=\boxed{
    \frac{c_u\widehat\xi-c_0}
    {(1-\widehat\xi)^2}
    }.
    \label{eq:unitary-brauer-Cu}
    \end{align}
    
    \subsubsection*[The fixed curve Cv]{The curve \(C_v\)}
    
    In \(U_{\sigma_3}=\operatorname{Spec}\bar{k}[w_1,w_2]\), the curve \(C_v\)
    is given by \(w_1=0\), and the normal eigenvalue is
    \(\bar\xi^{p}\).  Since \(C_v^2=-(p^2-1)\),
    \begin{align}
    b_{\mathrm{Br}}(C_v)
    &=\frac{c_0\widehat\xi^p}
    {1-\widehat\xi^p}
    +\frac{(p^2-1)\widehat\xi^p}
    {(1-\widehat\xi^p)^2}\nonumber\\
    &=\boxed{
    \frac{\widehat\xi^p(c_v-c_0\widehat\xi^p)}
    {(1-\widehat\xi^p)^2}
    }.
    \label{eq:unitary-brauer-Cv}
    \end{align}
    
    \subsubsection{Proof of the non-split main theorem}
    
    The fixed locus of \(\tau\) on
    \(\widetilde{\overline {\mathbf{Y}}}_{\delta_0}(w)\) is the disjoint union of
    \(C_u\), \(C_v\), and the \(p^3+1\) isolated points.  Adding the three
    types of contribution gives
    \begin{align}
    L_{\mathrm{Br}}(\tau,\widetilde{\overline {\mathbf{Y}}}_{\delta_0}(w))
    ={}&
    \frac{c_u\widehat\xi-c_0}
    {(1-\widehat\xi)^2}
    -\frac{(p^3+1)\widehat\xi}
    {(1-\widehat\xi)(1-\widehat\xi^p)}\nonumber\\
    &+\frac{\widehat\xi^p(c_v-c_0\widehat\xi^p)}
    {(1-\widehat\xi^p)^2}.
    \label{eq:unitary-final-brauer}
    \end{align}
    By \eqref{eq:common-brauer-resolution-invariance}, this is the value on
    \(\overline {\mathbf{Y}}_{\delta_0}(w)\), so it proves
    \eqref{eq:intro-unitary-brauer}.  Translations by
    \( T^{wF}\) identify the \(p+1\) components and commute with
    \(\tau\).  Therefore
    \[
    L_{\mathrm{Br}}(\tau,\overline {\mathbf{Y}}(w))
    =(p+1)L_{\mathrm{Br}}(\tau,\overline {\mathbf{Y}}_{\delta_0}(w)).
    \]

    \subsection{Nonvanishing and Galois-module consequence}
    \label{sec:unitary-nonvanishing}
    
    \begin{proposition}
    \label{prop:unitary-main-value-nonzero}
    The non-split Brauer--Lefschetz numbers
    \[
    L_{\mathrm{Br}}(\tau,\overline {\mathbf{Y}}_{\delta_0}(w))
    \quad\text{and}\quad
    L_{\mathrm{Br}}(\tau,\overline {\mathbf{Y}}(w))
    \]
    are nonzero for every prime \(p\).
    \end{proposition}
    
    \begin{proof}
    Let \(R_p(T)\in\mathbb Q(T)\) denote the rational function on the
    right-hand side of \eqref{eq:intro-unitary-brauer}.  If \(\omega\) is a
    primitive \((p^2-1)\)-st root, then \(\omega\neq1\) and
    \(\omega^p\neq1\), because \(p\nmid p^2-1\).  Thus every denominator in
    \(R_p(\omega)\) is nonzero.  Lemma~\ref{lem:cyclotomic-comparison}, with
    \(M=p^2-1\) and \(\widehat\omega=\widehat\xi\), reduces the claim to
    nonvanishing at \(\omega=e^{2\pi i/(p^2-1)}\).
    
    Put
    \[
    A=\frac{\pi}{p^2-1},
    \qquad
    x=e^{2iA},
    \qquad
    K=p^3+1.
    \]
    The self-intersection parts of
    \eqref{eq:intro-unitary-brauer} are real, because
    \[
    \frac{x}{(1-x)^2}
    \qquad\text{and}\qquad
    \frac{x^p}{(1-x^p)^2}
    \]
    are real.  Applying the elementary cotangent identities to the remaining
    terms gives
    \[
    \operatorname{Im}L_{\mathrm{Br}}
    (\tau,\overline {\mathbf{Y}}_{\delta_0}(w))
    =\frac{K+2c_0}{4}\bigl(\cot(pA)-\cot A\bigr).
    \]
    Here
    \[
    K+2c_0=(p+1)(p^2-1)>0,
    \]
    whereas \(0<A<pA<\pi\).  Since \(\cot\) is strictly decreasing on
    \((0,\pi)\), the second factor is negative.  Hence the component value is
    nonzero.  The whole-space value is \(p+1\) times the component value, so
    it is nonzero as well.
    \end{proof}
    
    \begin{corollary}[Non-split Galois-module consequence]
    \label{cor:unitary-galois-module-not-free}
    Put \(A_{\mathrm{un}}=H=\langle\tau\rangle\).  Then
    \[
    \chi_{A_{\mathrm{un}}}
    \bigl(\overline {\mathbf{Y}}^\circ(w),\mathcal O_{\overline {\mathbf{Y}}^\circ(w)}\bigr)
    \]
    is not virtually free over \(\bar{k}[A_{\mathrm{un}}]\).  Moreover,
    \[
    \chi_{ T^{wF}}
    \bigl(\overline {\mathbf{Y}}(w),\mathcal O_{\overline {\mathbf{Y}}(w)}\bigr)
    \]
    is not virtually free over \(\bar{k}[ T^{wF}]\).
    \end{corollary}

    \clearpage
    \bibliographystyle{amsalpha}
    \bibliography{references}
    
    \end{document}

%% file: preamble.tex
\usepackage[T1]{fontenc}
\usepackage[utf8]{inputenc}
\usepackage{amssymb,amsfonts}
\usepackage{mathrsfs}
\usepackage{mathtools}
\usepackage{xcolor}
\usepackage{float}
\usepackage[normalem]{ulem}
\usepackage{tikz,tikz-cd}
\usetikzlibrary{decorations.pathreplacing}
\usepackage{hyperref}

\hypersetup{
  colorlinks=true,
  linkcolor=blue!55!black,
  citecolor=green!40!black,
  urlcolor=blue!60!black,
  pdftitle={Galois module structure of Deligne--Lusztig varieties for GL3},
  pdfauthor={Adithya Sathyanarayana}
}

\allowdisplaybreaks
\numberwithin{equation}{section}
\tikzset{every picture/.style={line width=0.75pt}}

\theoremstyle{plain}
\newtheorem{theorem}{Theorem}[section]
\newtheorem{lemma}[theorem]{Lemma}
\newtheorem{corollary}[theorem]{Corollary}
\newtheorem{proposition}[theorem]{Proposition}

\theoremstyle{definition}
\newtheorem{definition}[theorem]{Definition}

\theoremstyle{remark}
\newtheorem{remark}{Remark}[section]

\renewcommand{\P}[1]{\mathbb{P}({#1})}

\newcommand{\Spec}{\operatorname{Spec}}
\newcommand{\diag}{\operatorname{diag}}
\newcommand{\Ad}{\operatorname{Ad}}
\newcommand{\antidiag}{\operatorname{antidiag}}

\newcommand{\Cone}{\operatorname{Cone}}
\newcommand{\divisor}{\operatorname{div}}

\title{Galois module structure of Deligne--Lusztig varieties for
\(\mathbf G=\mathrm{GL}_3\)}
\author{Adithya Sathyanarayana}
\date{}